\documentclass[12pt]{article}
\usepackage{amsmath, amsfonts, latexsym, amssymb, graphicx, mathtools,  wasysym, enumitem, pifont}

\usepackage[all]{xy}
\usepackage{hyperref}

\hypersetup{colorlinks,citecolor=blue,linkcolor=blue}

\title{A Manual for Ends, Semistability and Simple Connectivity at Infinity for Groups and Spaces}
\author{Michael Mihalik}
\usepackage{tikz}
\newtheorem{theorem}{Theorem}[subsection]
\newtheorem{proposition}[theorem]{Proposition}
\newtheorem{lemma}[theorem]{Lemma}
\newtheorem{corollary}[theorem]{Corollary}
\newtheorem{remark}[theorem]{Remark}
\newtheorem{question}[theorem]{Question}
\newtheorem{conjecture}[theorem]{Conjecture}
\newtheorem{example}[theorem]{Example}

\newcounter{claimnum}

\newcounter{definitionnum}
\newenvironment{definition}{\addvspace{12pt}\refstepcounter{definitionnum}
\noindent{\bf Definition \arabic{definitionnum}.}}{\par\addvspace{12pt}}

\newenvironment{proof}{\addvspace{12pt}\noindent{\bf Proof:}}{
$\Box$\par\addvspace{12pt}}

\date{\today}
\begin{document}
\maketitle
\newpage
\centerline{\bf Table of Contents}

\medskip

\noindent {\bf Chapter} \ref{Ends}. {\bf Ends of Groups and Spaces}

\noindent\hspace{.6in} \S \ref{Esp} {\bf Ends of Spaces}

\noindent\hspace{.6in} \S \ref{Egps} {\bf Ends of Groups}

\noindent\hspace{.6in} \S \ref{PairsG} {\bf Ends and Pairs of Groups}

\noindent\hspace{.6in} \S \ref{1eg} {\bf One Ended Groups}

\noindent\hspace{1.2in} \S \ref{KGs} {\bf Knot Groups}

\noindent\hspace{1.2in} \S \ref{normalE} {\bf Normal and Commensurated Subgroups} 

\noindent\hspace{1.2in} \S \ref{compGP}  {\bf Combination Results and Graph Products} 

\noindent\hspace{1.2in} \S \ref{ACGE} {\bf Artin and Coxeter Groups}

\noindent\hspace{1.2in} \S \ref{noF2} {\bf Groups with No Free Subgroups of Rank 2}

\noindent\hspace{1.2in} \S \ref{AHNN} {\bf Ascending HNN-Extensions} 

\noindent\hspace{1.2in} \S \ref{Touik} {\bf Touikan's Theorem for Ends of Graphs of} 

\noindent \hspace{1.7in}{\bf Groups}

\noindent\hspace{1.2in} \S \ref{Hspace} {\bf  Halfspaces of Splittings of Finitely Generated} 

\noindent \hspace{1.7in}{\bf Groups}

 \medskip
 
\noindent {\bf Chapter} \ref{SSscH}. {\bf Semistability and Simple Connectivity at $\infty$ of Groups and Spaces and $H^2(G,\mathbb ZG)$}

\noindent\hspace{.6in} \S \ref{DefSSsp} {\bf Semistability and Simple Connectivity at $\infty$ of a Space}

\noindent\hspace{.6in} \S \ref{prosp} {\bf Fundamental Group and Pro-Group at $\infty$ of a Space}

\noindent \hspace{.95in}{\bf and the Equivalence of Several Definitions} 

\noindent\hspace{.6in} \S \ref{SSproGp} {\bf Proper $n$-Equivalence, Semistability at $\infty$, Simple} 
 
\noindent \hspace{.95in}{\bf Connectivity at $\infty$, Fundamental Pro-Group at $\infty$} 

\noindent \hspace{.95in}{\bf and Fundamental Group at $\infty$ - for a Group}

\noindent\hspace{1.2in} \S \ref{Pr2eq} {\bf The Definitions} 

\noindent\hspace{1.2in} \S \ref{Coax} {\bf Pro-mono and Coaxial Actions}

\noindent\hspace{1.2in} \S \ref{SCS} {\bf Semistability and Co-semistability at $\infty$}

\noindent\hspace{1.2in} \S \ref{C2C} {\bf The Cayley 2-Complex for a Finite} 

\noindent \hspace{1.7in}{\bf Presentation of a Group}

\noindent\hspace{1.2in} \S \ref{FGss} {\bf Semistability at $\infty$ for Finitely Generated Groups}

\noindent\hspace{1.2in} \S \ref{GSSCinf} {\bf Simply Connected at $\infty$ Finitely Generated} 

\noindent \hspace{1.7in}{\bf Groups}

\noindent\hspace{.6in} \S \ref{Gresults} {\bf Semistability and Simple Connectivity at $\infty$ Results} 

\noindent\hspace {.95in}{\bf for Classes of Groups}

\noindent\hspace{1.2in} \S \ref{CC} {\bf Semistability, Simple Connectivity at $\infty$ and} 

\noindent\hspace{1.7in}{\bf Manifolds (Historical)} 

\noindent\hspace{1.2in} \S \ref{SubCom} {\bf Normal, Subnormal, Commensurated and Sub} 

\noindent\hspace{1.7in}{\bf Commensurated Subgroups}

\noindent\hspace{1.2in} \S \ref{BSgps} {\bf Bieri-Stallings Groups}

\noindent\hspace{1.2in} \S \ref{OUT} {\bf Out($F_n$)}

\noindent\hspace{1.2in} \S \ref{MCG} {\bf Mapping Class Groups of Surfaces}

\noindent\hspace{1.2in} \S \ref{GLNZ} {\bf GL($\mathbb Z,n)$}

\noindent\hspace{1.2in} \S \ref{SolMet} {\bf Solvable and Metanilpotent Groups}

\noindent\hspace{1.2in} \S \ref{WHyp} {\bf Word Hyperbolic and CAT(0) Groups}

\noindent\hspace{1.2in} \S \ref{RHGs} {\bf Relatively Hyperbolic Groups}

\noindent\hspace{1.2in} \S \ref{AHNNE} {\bf Ascending HNN Extensions and Eventually} 

\noindent\hspace{1.7in} {\bf Injective Endomorphisms}

\noindent\hspace{1.2in} \S \ref{TGF} {\bf Thompson's Group $F$} 

\noindent\hspace{1.2in} \S \ref{AGCG} {\bf Artin Groups and Coxeter Groups}

\noindent\hspace{1.2in} \S \ref{Sed} {\bf Sidki Doubles} 

\noindent\hspace{1.2in} \S \ref{StablePP1} {\bf Stable Fundamental Group at $\infty$}

\noindent\hspace{1.2in} \S \ref{WildPro} {\bf Exotic Fundamental Pro-group and} 

\noindent\hspace{1.7in} {\bf Fundamental Group at $\infty$ for a Group}

\noindent\hspace{1.2in} \S \ref{NonSS} {\bf The Lamplighter Group and Other Non-Semistable} 

\noindent\hspace{1.7in} {\bf at $\infty$ Finitely Generated Examples}

\noindent\hspace{1.2in} \S \ref{Comb1R} {\bf Amalgamated Products, HNN-Extensions,}

\noindent\hspace{1.7in} {\bf 1-Relator Groups, and Subgroup} 

\noindent\hspace{1.7in} {\bf Combination Results}

\noindent\hspace{1.2in} \S \ref{GPofG} {\bf Graph Products of Groups}

\noindent\hspace{1.2in} \S \ref{AD2} {\bf Asymptotic Dimension 2 and Simple} 

\noindent\hspace{1.7in} {\bf Connectivity at $\infty$}

\noindent\hspace{1.2in} \S \ref{tech} {\bf Useful Technical Results}

\noindent\hspace{.6 in} \S \ref{SSHom} {\bf Semistable Homology at Infinity}

\noindent\hspace{1.2in} \S \ref{RedH2-1} {\bf Reducing the $H^2(G,\mathbb ZG)$ Problem to 1-ended}

\noindent\hspace{1.7in} {\bf Groups}

\noindent\hspace{1.2in} \S \ref{EVHS} {\bf Equivalent Versions of Homological}

\noindent\hspace{1.7in} {\bf Semistability}

\noindent\hspace{1.2in} \S \ref{ac1} {\bf Groups, 1-acyclicity at Infinity and Pro-finite}  

\noindent\hspace{1.7in} {\bf First Homology at $\infty$ }

\noindent\hspace{1.2in} \S \ref{OnePt} {\bf One-point Compactifications and Semistabilty}

\noindent\hspace{1.2in} \S \ref{CatRe} {\bf A Catalogue of Homological Results} 

\noindent\hspace{.6 in} \S \ref{LSCIR} {\bf List of Simply Connected at Infinity Group Results}

\noindent\hspace{.6 in} \S \ref{GSI} {\bf Group/Subgroup Index}

\newpage
\begin{abstract}\label{abst}
This $2^{nd}$-edition article is intended to be an up-to-date archive of the current state of the questions: Which finitely generated groups $G$: have semistable fundamental group at infinity; are simply connected at infinity; are such that $H^2(G,\mathbb ZG)$ is free abelian or trivial. The idea is not to reprove these results, but to provide a historical record of the progress on these questions and provide a  list of the most general results. We also prove or cite all of the results that make up the basic theory. The first Chapter is devoted to ends of groups and spaces, and the second to semistability at infinity, simple connectivity at infinity and second cohomology of groups. Definitions, basic facts and lists of general results are given in each Chapter. A number of results  proven here are new and a number of authors have contributed results. We end with an Index for simply connected at infinity groups and an Index of Groups and Subgroups which is intended to help a reader quickly locate results about certain types of groups/subgroups. The main updates from the first edition is section 2.4.5 on mapping class groups and the addition of the simply connected at infinity index.
\end{abstract}
\section{Ends of Groups and Spaces}\label{Ends}
\subsection{Ends of Spaces} \label{Esp} 

Sections 13.4 and 13.5 of  R. Geoghegan's book \cite{G} give an excellent account of the notion of ends of groups and spaces. We select a few results from these sections and provide others that lead into results for semistability at $\infty$ and simple connectivity at $\infty$. Ends make sense in a wide variety of spaces. As far as groups are concerned CW-complexes, and perhaps absolute neighborhood retracts (ANRs) generally suffice. The following result is one of the most general results we know that allows for a definition for ends of a space. 

\begin{theorem} \label{ES}  Suppose $X$ is connected, locally compact, locally connected
and Hausdorff and $C$ is compact in $X$,
then $C$ union all bounded components of $X - C$ is compact, and $X-C$ has only finitely many unbounded components.
Here bounded means compact closure.
\end{theorem}

\begin{proof}
We may assume that $X$ is not compact and $C$ is not empty. Since $X$ is Hausdorff, compact subsets of $X$ are closed in $X$. Since $X$ is locally compact there is a compact set $D$ such that $C$ is a subset of $int(D)$, the interior of $D$. Let $Bd(D)=D-int(D)$. Every point of $Bd(D)$ belongs to some component of $X-C$. Since $X$ is locally connected each component of $X-C$ is open in $X$. 

{\it We first show that if $K$ is a component of $X-C$ then $K\cap D\ne \emptyset$.} Otherwise, let $A$ be the union of all components of $X-C$ that intersect $D$ and $B$ be the union of all other components. Consider the disjoint open sets $int(D) \cup A$ and $B$. Since $K\subset B$, these sets are non-empty and separate the connected space $X$, a contradiction.

{\it If $K$ is a component of $X-C$ that is not a subset of $D$, then $K\cap Bd(D)\ne\emptyset$} (since $K$ is connected).

Next we show: {\it All but finitely many components of $X-C$ are subsets of $D$.} Suppose that for $i\geq 1$, the sets $C_i$ are distinct components of $X-C$ such that $C_i\not\subset D$. Let $x_i$ be a point of $C_i\cap Bd(D)$. Since $Bd(D)$ is a closed subset of $D$, it is compact. Let $x$ be a cluster point of $\{x_i\}_{i=1}^\infty$ (any open neighborhood of $x$ contains a point of $\{x_i\}$ other than $x$).  
Then $x\in K$ for some component $K$ of $X-C$. Since $K$ is open in $X$, there is $x_i\ne x$ such that $x_i\in K$. Since points are closed in a Hausdorff space, $K-\{x_i\}$ is a neighborhood of $x$ and so there is $x_j\in K$ such that $x_i\ne x_j\ne x$.   But the $x_k$ belong to exactly one component of $X-C$, a contradiction.

Any component that belongs to $D$ is bounded, so {\it all but finitely many components of $X-C$ are bounded.} If $E$ is $C$ union all components of $X-C$ that are subsets of $D$ then $X-E$ is a union of components and hence open. Thus $E$ is a closed subset of $D$ and hence compact. There are only finitely many bounded components of $X-C$ that are not a subset of $D$. List these as $B_1, \ldots , B_n$. Let $F=E\cup_{i=1}^nB_i$. Then $F$ is bounded and $X-F$ is a union of components of $X-C$, so that $F$ is closed. This means $F$ is compact. As $F$ is equal to $C$ union all bounded components of $X-C$, we are finished.
\end{proof}

From this point on, all spaces considered will be  connected, locally compact, locally connected, $\sigma$-compact (a countable union of compact sets)
and Hausdorff.  We only reiterate when making definitions. 
For $C$ compact in $X$ let $U(C)$ denote the number of unbounded components of $X-C$. 
Note that if $C\subset D$ are compact in  $X$ then for each unbounded component $K$ of $X-C$ there is an unbounded component $K'$ of $X-D$ such that $K'\subset K$. In particular $U(D)\geq U(C)$. If $\mathcal U(X-C)$ is the set of unbounded components of $X-C$ then there is an epimorphism of sets from $\mathcal U(X-D)$ to $\mathcal U(X-C)$ induced by inclusion. 

\begin{definition}\label{DefEnd}
Suppose $X$ is connected, locally compact, locally connected and Hausdorff. The {\it number of ends} of $X$ is the largest number $U(C)$ taken over all compact sets $C\subset X$. If these numbers are unbounded, then $X$ has infinitely many ends. 
\end{definition}

\begin{definition} 
A collection $\{C_i\}_{i\in \{0,1,\ldots\}}$ of compact sets in a space $X$ is {\it cofinal} if $C_i$ is a subset of the interior of $C_{i+1}$ and $\cup_{i=0}^\infty C_i=X$.\end{definition}

\begin{remark}
The reason we use cofinal sequences of compact sets (as opposed to a collection of (nested) compact sets that cover our space $X$) is to ensure that every compact set in $X$ belongs to some member of the cofinal sequence. As an example consider the following (not cofinal) nested sequence of compact subsets of $ [0,\infty)$. For $n\geq 1$, let 
$$C_n=[0, 1-{1\over n}] \cup [1,n].$$ 
While $[0,\infty)=\cup_{n=1}^\infty C_n$, no $C_n$ contains the interval $[0,1]$. 
\end{remark}
See \S 11.2 of \cite{G} for the basic notions of inverse systems and inverse limits of sets. 

\begin{definition}\label{DProI1}  
Suppose $G_1{\buildrel p_1\over \longleftarrow}G_2 {\buildrel p_2\over \longleftarrow}\cdots$ and 
$H_1{\buildrel q_1\over \longleftarrow}H_2 {\buildrel q_2\over \longleftarrow}\cdots$ are inverse sequences of topological spaces with bonding maps continuous functions. These sequences are {\it pro-isomorphic} if there are sequences of integers $0< m_1<m_2<\cdots$ and $0<n_1<n_2<\cdots$ and continuous functions $f_{n_{k}}:G_{m_{k}}\to H_{n_{k}}$ and $g_{m_k}:H_{n_{k+1}}\to G_{m_k}$ making the following diagram commute for all $k$. (Horizontal arrows are compositions of bonding maps.)
\end{definition}

\vspace {-.7in}
\vbox to 2in{\vspace {-1in} \hspace {-.2in}
\includegraphics[scale=1]{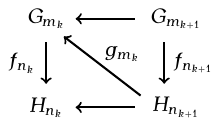}
\vss }
\vspace{-.3in}

The following result is straightforward. 

\begin{theorem} 
If $C_i$ and $D_i$ are cofinal sequences of compact subsets of $X$, then there is a pro-isomorphism between the inverse sequences of sets (each with the discrete topology and with bonding maps induced by inclusion) 
$$\mathcal U(X-C_0) \leftarrow \mathcal U(X-C_1)\leftarrow\cdots\hbox{ and }\mathcal U(X-D_0)\leftarrow \mathcal U(X-D_1)\leftarrow\cdots .$$  Hence there is a homeomorphism between the inverse limits of these two inverse sequences (topologies on the inverse limits are subspace topologies from the product topologies). 
\end{theorem}

\begin{definition}
The {\it space of ends} of a space $X$ is the inverse limit of  the inverse sequence $\pi_0(X-C_1)\leftarrow \pi_0(X-C_2)\leftarrow\cdots$ for any cofinal sequence of compact sets $\{C_i\}_{i=1}^\infty$ in $X$ and denoted $\mathcal E(X)$. 
\end{definition}

For the following result, consider Proposition 13.4.7 of \cite{G}  and the paragraph that follows.  
\begin{theorem} 
The space $\mathcal E(X)$ is a compact, totally disconnected metrizable space. If $\mathcal E(X)$ is finite, then its cardinality is the number of ends of $X$. Otherwise $X$ has infinitely many ends. 
\end{theorem}

Note that $X$ is compact if and only if $U(X)=0$ and if $X$ has $n$ ends for $0\leq n<\infty$ then, the space of ends is discrete with $n$ points. 

\begin{definition}
A map (continuous function) $f:X\to Y$ is {\it proper} if $f^{-1}(C)$ is compact for each compact $C\subset Y$. 
\end{definition}

\begin{theorem} 
If $X$ and $Y$ are locally compact spaces and each is the union of a countable collection of compact sets, then a proper map $f:X\to Y$ induces a continuous map of $\mathcal E(X)$ into  $\mathcal E(Y)$. 
\end{theorem}
\begin{proof}
Let $C_i$ be a cofinal sequence of compact sets for $X$. Choose a cofinal sequence of compact sets $D_i$ in $Y$ such that (the compact sets) $f(C_i)\subset D_i$. Suppose $K_1$, $K_2,\ldots$ is a sequence of unbounded components of $X-C_i$ such that $K_i\subset K_{i-1}$ for all $i>1$. Since $f$ is proper, $f(K_i)$ is an unbounded connected subset of $L_i$, an unbounded component of $Y-D_i$. Note that $L_i\subset L_{i-1}$ for all $i>1$, and $f$ induces a function from $\mathcal E(X)$ to $\mathcal E(Y)$. Suppose $K_i'$, is another component of $X-C_i$ for each $i$, $K_i'\subset K+{i-1}'$ for all $i$ and $L_i'$ is the (unbounded) component of $Y-D_i$ containing $f(K_i')$. If $K_i=K_i'$ for $i\leq n$, then $L_i=L_i'$ for $i\leq n$. Hence $f$ induces a continuous function from $\mathcal E(X)$ to $\mathcal E(Y)$. 
\end{proof}

\begin{theorem}  [Proposition 13.4.1, \cite{G}] \label{E1sk}  
If $X$ is a connected, locally finite CW complex, then the spaces of ends $\mathcal E(X)$ and $\mathcal E(X^1)$ are homeomorphic. In particular, the number of ends of $X$ is the same as the number of ends of its 1-skeleton. 
\end{theorem}

\begin{definition}
Suppose $X$ is a locally finite connected CW complex. A {\it ray} in $X$ is a map $r:[0,\infty)\to X$. The proper rays $r,s$  {\it converge to the same end of $X$} (denoted $r\sim s$) if for any compact set $C$ in $X$ there is an integer $N(C)$ such that $r([N(C),\infty)$ and $s([N(C),\infty)$ are subsets of the same unbounded path component of $X-C$. Let $\ast$ be a base point in $X$, then $R(X,\ast)\subset R(X)$ is the set of proper rays $r$ in $R(X)$ such that $r(0)=\ast$.  
\end{definition} 

\begin{theorem} \label{EQend} 
Let $X$ be a locally finite connected CW complex and $R(X)$ be the set of proper rays $r:[0,\infty)\to X$. Then $\sim$ is an equivalence relation on $R(X)$.  There is an natural bijection from  the equivalence classes $R(X,\ast)/\sim$ to $R(X)/\sim$ induced by inclusion and a natural bijection from $R(X)/\sim$ to $\mathcal E(X)$.
\end{theorem}

\begin{definition} If $X$ is a locally finite connected CW complex, then the set of equivalence classes $R(X)$ (equivalently $R(X,\ast)$ for any base point $\ast\in X$) is called the {\it set of ends of $X$} and is denoted $E(X)$ (respectively $E(X,\ast)$).
\end{definition}

\begin{example} \label{EE}
The wedge of $n$ proper rays at a point has $n$-ends and the space of ends is discrete on $n$ points.  If $n\geq 2$ then $R^n$ has 1-end. If  the valence of each vertex of a locally finite tree is $\geq 3$, then the tree has infinitely many ends and the space of ends for this tree is a Cantor set.
\end{example}

\begin{theorem} \label{geoend} 
If $X$ is a locally finite connected CW complex then for any proper ray $r:[0,\infty)\to X$ there is a proper edge path ray $r':[0,\infty)\to X^1$ such that $[r]=[r']$. If $r$ begins at a vertex $v$ of $X$, then $r'$ can be selected to begin at $v$. Furthermore, for any vertex $w\in X$ there is a geodesic edge path ray $s$ at $w$, such that $[r]=[s]$. 
\end{theorem}
\begin{proof}
The first part of the result is evident. We construct the geodesic $s$. Let $r'$ be a proper edge path ray such that $[r]=[r']$. List the consecutive vertices of $r'$ as $v_i=r'(i)$. For each $i$, let $\alpha_i$ be a geodesic edge path from $w$ to $v_i$. Since $X$ is locally finite, infinitely many of the $\alpha_i$ have the same first edge - call it $e_0$. Of these, infinitely many have the same second edge, $e_1$. Continuing, we obtain the edge path $s=(e_0, e_1,\ldots)$ based at $w$. Certainly $s$ is geodesic since any initial segment is a subsegment of a geodesic. If $C$ is compact in $X^1$, choose $N$ such that $C\subset B_N(w)$ (here $B_N(w)$ is the ball of radius $N$ (in $X^1$) about $w$). Choose $M$ such that $r'([M,\infty))\subset X-B_N(w)$. Let $w_i$ be the initial vertex of $e_i$. Observe that $(e_{N+2}, e_{N+3},\ldots)$ avoids $B_N(w)$. We need only connect $(e_{N+2}, e_{N+3},\ldots)$ and $r'([M,\infty))$ in $X^1-B_N(w)$.
By construction, infinitely many of the $\alpha_i$ are extensions of the geodesic $(e_1,\ldots, e_{N+1})$.  Select one that terminates at a vertex $v_K$ of $r'$, where $K>M$. Then the tail of this geodesic from the initial point of $e_{N+2}$ to a vertex of $r'([M,\infty))$ is in  $X^1-B_N(w)$. 
\end{proof}
\begin{lemma}\label{geodesicline} 
Suppose $G$ is a group acting transitively on the vertices of a locally finite connected infinite graph $\Gamma$. Then there is a geodesic line through any vertex of $\Gamma$.
\end{lemma}
\begin{proof} Choose geodesics $\alpha_1,\alpha_2,\ldots$ in $\Gamma$ such that $|\alpha_i|=2i$. Let $v$ be a vertex of $\Gamma$. By the action of $G$, we may assume the midpoint of $\alpha_i$ is $v$ for all $i$. By local finiteness, infinitely many of the $\alpha_i$ have the same edge before and after $v$ (say, $(e_{-1},e_1)$). Of these, infinitely many have the same edge preceding $e_{-1}$ and following $e_1$ (say $(e_{-2},e_{-1}, e_1,e_2)$). Continuing, we have the desired geodesic line is $(\ldots, e_{-2}, e_{-1}, e_1,e_2, \ldots)$. 
\end{proof}

\begin{theorem} \label{line=2} 
If $X$ is a connected locally finite graph and $l$ is a geodesic line in $X$ such that $X=B_k(im(l))$ for some $k>0$, then $X$ is 2-ended. 
\end{theorem}
\begin{proof} 
Let $\ast$ be a vertex of $l$ and $n>2k$ an integer. Let $K_n$ be the compact subcomplex of $X$ consisting of $B_{n}(\ast)$ union all bounded components of $X-B_n(\ast)$. It suffices to show that $X-K_n$ has exactly 2 components. Assume that $l$ is parametrized so that $l(0)=\ast$. Note that $l([n+1, \infty))$ and $l([-n-1,\-\infty))$ belong to components $Q^+$ and $Q^-$ (respectively) of $X-K_n$. 

First we show that $Q^+\ne Q^-$ (so that $X$ has at least 2-ends). Otherwise there is an edge path $\alpha$ from $l(n+1)$ to $l(-n-1)$ that avoids $B_n(\ast)$. Each vertex of $\alpha$ is within $k$ of a vertex of $l$ and no vertex of $\alpha$ is within $k$ of a vertex of $l([-k,k])$ (since $n>2k$). Let $m$ be the first integer such that $\alpha(m)$ is within $k$ of a vertex of $l([-k-1,-\infty))$. Note that $m\ne 0$ or $l(n+1)$ is within $k$ of $l(q)$ for $q<-k$. But that is impossible since $d(l(q), l(n+1)=n+1-q>k$. There is an edge path of length $\leq k$ from $\alpha(m-1)$ to $l(q_1)$ for some $q_1>k$, and an edge path from $\alpha(m)$ to $l(q_2)$ where $q_2<-k$. Then there is an edge path of length $\leq 2k+1$ from $l(q_1)$ to $l(q_2)$. But $d(l(q_1), l(q_2))=q_1-q_2\geq 2k+2$. Instead $Q^+\ne Q^-$ and $X$ has at least 2-ends. 

If $Q_3$ is an unbounded component of $X-K_n$. We show $Q_3$ is either $Q^+$ or $Q^-$. Let $v$ be a vertex of $Q_3$ such that $d(v,\ast)>n+k$. Since $d(v,\ast)> n+k>3k$, $v$ is not within $k$ of a vertex of $l([-k,k])$. Without loss, say $d(v,l(m))\leq k$ for $m>k$. If $\beta$ is an edge path of length $\leq k$ from $v$ to $l(m)$, then since $d(v,\ast)>n+k$, each vertex of $\alpha$ avoids $B_n(\ast)$. In particular $m>n$. Then $v$ is in the same component of $X-B_n(\ast)$ as is $l(m)$ (which is the unbounded component $Q^+$).  This means $Q_3=Q^+$. 
\end{proof}

\begin{definition} Let $(X, d_X )$ and $(Y, d_Y )$ be metric spaces.

(i) A map $f : X \to Y$ is called a {\it quasi-isometric embedding} if there exist non-negative constants
$K, C\in \mathbb R$, such that
$${1\over K} d_X (x_1 , x_2 ) -C\leq  d_Y (f (x_1 ), f (x_2 )) \leq K d_X (x_1 , x_2 ) + C$$

(ii) A quasi-isometric embedding $f : X \to Y$ is called a {\it quasi-isometry}, if there exists a constant $D>0$ such that for every $y\in Y$ there exists an $x\in X$ so that $d_Y (f(x),y)\leq D$.

(iii) The metric spaces $(X, d_X )$ and $(Y, d_Y )$ are {\it quasi-isometric} if there exists a quasi-isometry $f : X\to Y$.
\end{definition}



\begin{theorem} [Proposition 2.2, \cite{BR93}] \label{QIGra} 
If two locally finite connected graphs are quasi-isometric, then they have the same number of ends. 
\end{theorem} 
Combining this result with Theorem \ref{E1sk} produces: 

\begin{theorem} \label {QIE} 
The number of ends of a connected, locally finite CW-complex is a quasi-isometry invariant of its 1-skeleton.
\end{theorem}

\begin{definition}\label{n-equiv}
A (proper) cellular map $f:X\to Y$ between CW-complexes is a {\it (proper) $k$-equivalence} if there is a (proper) cellular map $g:Y\to X$ such that $gf:X^{k-1}\to X$ is (properly) homotopic to the inclusion $X^{k-1}\to X$ and $fg:Y^{k-1}\to Y$ is (properly) homotopic to the inclusion $Y^{k-1}\to Y$.  
\end{definition}

The following is elementary:
\begin{theorem}
Suppose $X$ and $Y$ are connected locally finite CW-complexes.  A proper 1-equivalence $f:X^1\to Y^1$ induces a bijection between $E(X)$ and $E(Y)$ as well as a homeomorphism between $\mathcal E (X)$ and $\mathcal E (Y)$.\end{theorem}  

\newpage

\subsection{Ends of Groups}\label{Egps}
Ends of groups were introduced by H. Hopf \cite{H44} and H. Freudenthal \cite{F45}.

\begin{definition} \label{Cayley}
Suppose $S$ is a generating set for the group $G$. The {\it Cayley graph} $\Gamma(G,S)$ is a labeled locally finite 1-complex with vertex set $G$. There is an edge between vertices $a$ and $b$ if and only if there is an element $s\in S$ such that $as=b$. The label of the edge $(a,b)$ is $s$. 
\end{definition}

\begin{definition}\label{DQIG} 
Suppose $G$ and $H$ are finitely generated groups. The group $G$ is quasi-isometric to $H$ if for some (equivalently any by Theorem \ref{QIEFG}) finite generating sets $A$ and $B$ for $G$ and $H$ respectively, $\Gamma(G,A)$ is quasi-isometric to $\Gamma(H,B)$. 
\end{definition}

Suppose $X$ is a connected locally finite graph and $G$ is a finitely generated group acting properly and cocompactly by graph isomorphisms on $X$, (in particular, if $X$ is the Cayley graph on a finite generating set for $G$). Let   $\ast$ be a vertex in $X$. Suppose $A$ is a finite generating set for $G$ and for each $a\in A$, $\alpha_a$ is a shortest edge path from $\ast$ to $a\ast$. Let $M(X,A):\Gamma(G,A)\to X$ as follows: If $v\in G$ is a vertex of $\Gamma$, then $M(v)=v\ast$. If $e$ is an edge between vertices $v$ and $av$ in $\Gamma$  (for some $a\in A$), then let $M(e)$ be the edge path $v\alpha_a$ between $v\ast$ and $av\ast$. The map $M(X,A)$ is a quasi-isometry from $\Gamma(G,A)$ to  $X$. Theorem \ref{QIE} implies $\Gamma(G,A)$ and $ X$ have the same number of ends (for any finite generating set $A$ of $G$). The following definition is now consistent:






\begin{definition}\label{NuE} 
If $G$ is a finitely generated group, define the number of ends of $G$ to be the number of ends of $\Gamma(G,A)$ for any finite generating set $A$ for $G$. The space of ends of $G$, denoted $\mathcal E(G)$, is the space of ends of $\Gamma(G,A)$. 
\end{definition}

We immediately have: 

\begin{theorem} [Corollary 2.3, \cite{BR93}] \label{QIEFG} 
The number of ends of a finitely generated group is a quasi-isometry invariant of the group. 
\end{theorem}

\begin{theorem} Suppose $K$ is a finite connected complex and $\tilde K$ is the universal cover of $K$ then $\tilde K$ has the same number of ends as $\pi_1(K)$. 
\end{theorem}
\begin{proof}
If $A$ is a finite generating set for $\pi_1(K)$, then the map (defined above) $M(\tilde K^1,A):\Gamma(G,A)\to \tilde K^1$ is a quasi-isometry. Theorem \ref{QIE} implies $\tilde K$ and $\tilde K^1$ have the same number of ends.
\end{proof}

\begin{theorem}\label{E3inf} 
Suppose a group $G$ acts (by homeomorphisms) on a connected, locally connected, locally compact Hausdorff space $X$, such that for each compact set $C\subset X$ there is an element $g_C\in G$ such that $g_C(C)\cap C=\% !TEX spellcheck = emptyset$. Then $X$ has $0$, $1$, $2$ or infinitely many ends.
\end{theorem}

\begin{proof}
Suppose $X$ has at least $n(\geq 3)$ ends, then it is enough to show that $X$ has at least $n+1$ ends. Otherwise, there is a compact subset $C$ of $X$ such that $X-C$ has $n$ unbounded components but no compact subset $D$ of $X$ such that $X-D$ has $n+1$ unbounded components. Theorem \ref{ES} implies that we may assume $X-C$ is a union of unbounded components $U_1,\ldots, U_n$. Choose $g\in G$ such that $gC\cap C=\emptyset$. Without loss assume $gC\subset U_1$. Then $U_2,\ldots, U_n$ are unbounded connected subsets of $X-(C_1\cup gC_1)$. Certainly for $i\geq 2$, $U_i$ is not properly contained in a connected subset of $X-(C_1\cup C_2)$. So, for $i\geq 2$, $U_i$ is an unbounded component of $X-(C_1\cup C_2)$. Since $g$ is a homeomorphism, $X-gC$ is the union of disjoint unbounded connected sets $V_1,\ldots, V_n$. Without loss, assume that $C\subset V_1$. As above, for $i\geq 2$, $V_i$ is an unbounded component of $X-(C\cup gC)$. 

If $u$ is a point in the topological boundary of $U_i$, then $U_i\cup\{u\}$ is connected. Since $U_i$ is a maximal connected subset of $X-C$, $u\in C$. Hence the topological boundary of $U_i$ is a (non-empty) subset of $C$ and so for $i\geq 2$, the closure of $U_i$ in $X$ is a connected subset of $X-gC$ that contains a point of $V_1$. This implies that for $i\geq 2$, $U_i\subset V_1$. Similarly, for $i\geq 2$, $V_i\subset U_1$. This implies that the $2n-2$ sets $U_2,\ldots, U_n, V_2,\ldots, V_n$ are disjoint. Since $n\geq 3$, $2n-2=n+(n-2)\geq n+1$.
\end{proof}

\begin{corollary}
The number of ends of a finitely generated group is 0, 1, 2 or $\infty$. If a group has infinitely many ends, then the space of ends is a Cantor set.
\end{corollary}

\begin{theorem}
A finitely generated group $G$ has 0-ends if and only if $G$ is finite.
\end{theorem}
 
 \begin{theorem} The group $G$ has 2-ends if and only if any one of the following statements holds:
 
 (1) $G$ has an infinite cyclic subgroup of finite index and so any infinite cyclic subgroup of $G$ has finite index in $G$. (Originally due to H. Hopf \cite{H44} - See for instance Theorem 13.5.9, \cite{G}.)
 
 (2) $G$ has a normal infinite cyclic subgroup of finite index. 
 
 (3) $G$ has a finite normal subgroup $N$ such that $G/N$ is isomorphic to $\mathbb Z$ or $\mathbb Z_2\ast \mathbb Z_2$. (Lemma 4.1, \cite{Wall67}) or (Theorem 5.1, \cite{St68})
 \end{theorem}
 
 \begin{corollary}
 If $H$ is a subgroup of a 2-ended group, then $H$ is either 2-ended or finite. 
 \end{corollary}
 \begin{proof}
 If $H$ is infinite, then let $g\in G$ be such that $\langle g\rangle$ has finite index in $G$. Then $\langle g\rangle\cap H$ has finite index in $H$.
 \end{proof}

 \begin{theorem} [See Theorem 13.5.5, \cite{G}] \label{HEnds} 
If $R$ is a PID, then the number of ends of a finitely generated infinite group $G$ is equal to 
$$rank_{R}(H^1(G, RG))+1.$$ 
 \end{theorem}
 
 More general conditions, which include the case $R=\mathbb ZG$ are given in \cite{Co73}, P. 36.
 
 Two of the most important results in the theory of ends of groups are due to J. Stallings and M. Dunwoody. In an as yet unpublished historical article (Ends and Accessibility), Dunwoody gives an illuminating account of the development of these two results and their connection to 3-manifold theory. 
 
 \begin{theorem} [J. Stallings, \cite{Stall71}] \label{Stall}   The finitely generated group $G$ has more than 1-end if and only if $G$ splits as a non-trivial amalgamated product $A\ast_CB$ (so $A\ne C\ne B$) or an HNN-extension $A\ast_C$ where $C$ is finite. The group $G$ has 2-ends if and only if $C$ has index 2 in both $A$ and $B$ in the amalgamated product case, and $C=A$ in the HNN case.  
 \end{theorem}
 
 A finitely generated group $G$ is {\it accessible} over splitting by finite groups if there is no infinite sequence of groups $\{G_i\}_{i=0}^\infty$ where $G=G_0$, $G_i$ splits non-trivially as $A_i\ast_{C_i} B_i$ or $A_i\ast_{C_i}$ where $C_i$ is finite and $G_i=A_{i-1}$ for all $i\geq 1$.
 
 \begin{theorem} [P. Linnell,\cite{Linn83}] \label{LinnAcc} 
 If there is a global bound on the order of finite order elements in a finitely generated group, then the group is accessible over splittings by finite groups. 

 \end{theorem}
 
\begin{theorem} [M. Dunwoody, \cite{Dun85}]  \label{DunAcc} 
All finitely presented groups (and in fact, all ``almost" finitely presented groups) are accessible over splittings by finite groups. In particular, if $G$ is a finitely presented group, then $G$ has a graph of groups decomposition, where each edge group is finite and each vertex group is either finite or a 1-ended finitely presented group. 
\end{theorem} 

 Dunwoody \cite{Dun93} gives examples of finitely generated (recursively presented) groups that are not accessible over splittings by finite groups. 
 
Since splittings over a finite group are always ``compatible" with any other splitting (if $C$ is a finite subgroup of a group $G$ with graph of groups decomposition $\mathcal G$, then $C$ is a subgroup of a conjugate of a vertex group of $\mathcal G$). This mean a vertex group of $\mathcal G$ can be split over $C$ and this splitting can be incorporated into the splitting $\mathcal G$, by replacing that vertex group by it's splitting over $C$. Since this process ends by Dunwoody's Theorem, a final graph of groups decomposition of $G$ (with each vertex group either 0 or 1-ended and each edge group 0-ended) is called a {\it Dunwoody decomposition of $G$}

\begin{theorem} If $G$ is an almost finitely presented group, then $G$ has a finite graph of groups decomposition where each edge group is finite and each vertex group is either finite or 1-ended. 
\end{theorem}

J. Dydak asked the following question: If a finitely generated group $G$ has infinitely many ends, is there an end $E$ such that $GE$ is dense in the space of ends of $G$? As it turns out $GE$ is dense in the space of ends of $G$ for every end $E$.
\begin{theorem}
If $G$ is a finitely generated infinite ended group and $E$ is an end of $G$, then $GE$ is dense in the space of ends of $G$.
\end{theorem}
\begin{proof}
Let $\Gamma$ be a Cayley graph for $G$ with respect to some finite generating set $S$. Let $E$ be an end of $\Gamma$. Theorem \ref{geoend} implies there is a geodesic edge path ray $s$ at the identity vertex $\ast$ of $\Gamma$ such that $s\in E$. Let $N_1>0$ be an integer and let $U_1$ be an unbounded component of $\Gamma-B_{N_1}(\ast)$. It suffices to show that there is a vertex $v\in U_1$ (so $v\in G$) such that $vs$ has image in $U_1$. Since $G$ is infinite ended, the proof of Theorem \ref{E3inf} shows there are integers $N_2,N_3,\ldots $ such that if $U_i$ is an unbounded component of $\Gamma-B_{N_i}(\ast)$ then $U_i$ contains at least 2 unbounded components of $\Gamma-B_{N_{i+1}}(\ast)$. This implies that $U_1$ contains at least $2^k$ unbounded components of $\Gamma-B_{N_{k+1}}$. Given any of these unbounded components, Theorem \ref{geoend} supplies a geodesic from $\ast$ with a tail in that unbounded component. Hence given these unbounded components each of contains a vertex such that all of these vertices are the same distance from $\ast$.
 
\medskip

\noindent $(\ast)$ Hence for any integer $K$, there are at least $K$ distinct vertices in $U$ such that each is the same distance from $\ast$. 

\medskip

Say $B_{N_1}$ has $K_1$ vertices. Choose $A$, a set of $(2N_1+1)(K_1+1)$ distinct vertices of $U$, each of distance $M$ from $\ast$. If some vertex $a\in A$ is such that $as$ has image in $\Gamma-B_{N_1}(\ast)$ we are finished. 

Assume $as$ intersects $B_{N_1}(\ast)$ for each $a\in A$. For $a\in A$, let $D_a$ be the length of the shortest initial segment of $as$ that terminates on $B_{N_1}$. By the triangle inequality
$$M+N_1\geq D_a\geq M-N_1.$$
Since $A$ has $(2N_1+1)(K_1+1)+1$ distinct vertices of $U_1$ and there are only $2N_1+1$ choices for the $D_a$ ($a\in A$), there must be some number $D$ such that for at least $K_1+1$ of the $a\in A$, $D_a=D$. Since $B_{N_1}(\ast)$ has only $K_1$ vertices, we have distinct vertices $a,b\in A$ such $D_a=D_b=D$ and the initial segments of $as$ and $bs$ of length $D$ end at the same point. Since these two segments have the same edge labelings, this is impossible. Instead some $a\in A$ is such that $as$ avoids $B_{N_1}(\ast)$ and so has image entirely in $U_1$. 
\end{proof}

\newpage

\subsection{Ends and Pairs of Groups}\label{PairsG}
If $H$ is a subgroup of a finitely generated groups $G$, then {\it the number of ends of the pair} $(G,H)$ is considered in \S13.5 of \cite{G}. Then number of {\it filtered} ends of $(G,H)$ is considered in \S14.5 of \cite{G}. We are interested in a different notion here, namely the number of ends of $H$ in $G$. This is a particularly useful idea that generalizes to the notion of a subgroup being semistable in an over group (see Definition \ref{SSSin} and Theorem \ref{SSSComb}).

\begin{definition} Suppose $H$ is a subgroup of a group $G$ and $S$ is a finite generating set for $G$. Say $H$ {\it has $K$ ends in} $G$ (or $H$ {\it is $K$-ended in }$G$) if $K$ is the largest integer such that for some compact set $C\subset \Gamma(G,S)$ there are $K$ components of $X-C$, each containing infinitely many points of $H$. If there is no such largest integer, then we say $H$ {\it has infinitely many ends in} $G$ or $H$ {\it is $\infty$-ended in }$G$).
\end{definition}

It is straightforward to show that this definition does not depend on the generating set $S$. 

\begin{theorem} 
The subgroup $H$ of $G$ has 0, 1, 2 or infinitely many ends in $G$.
\end{theorem}
\begin{proof}
Follow the proof of Theorem \ref{E3inf} using elements of $H$.
\end{proof}

\begin{theorem}
The subgroup $H$ of $G$ has 0 ends if and only if $H$ is finite.
\end{theorem}

\begin{theorem}
If $H$ has 1-end in $G$, then every infinite subgroup $K$ of $H$ has one end in $G$. In particular, if $G$ is 1-ended, then every infinite subgroup of $G$ has 1-end in $G$.
\end{theorem}

\begin{theorem}
If $H$ is 2-ended in $G$ then $H$ is finitely generated and 2-ended. 
\end{theorem}

\begin{theorem}\label{FI}
If $H$ has finite index in $K$ for $K$ a subgroup of $G$ then $H$ and $K$ have the same number of ends in $G$. 
\end{theorem}

\begin{theorem}
Suppose $G$ is a finitely generated group that splits non-trivially as  $A\ast_C B$ or $A\ast_C$ where $C$ is not equal to $A$ or $B$ and $C$ is finite. If $H$ is 1-ended in $G$ and $H$ is a subgroup of $A$, then $H$ is 1-ended in $A$.  
\end{theorem}
\begin{proof}
Let $X$ be the Cayley graph of $G$ on generating sets for $A$, $B$ and $C$. Let $K$ be a finite connected subgraph of $X$. Let $K_1$ be $K$ union all cosets $aC$ such that $a\in A$ and $aC\cap K\ne\emptyset$. Since $H$ is 1-ended in $X$, there is a finite subgraph $D$ of $X$ such that any two elements of $H-D$ can be connected by a path in $X-K_1$. Say $h_1$ and $h_2$ are elements of $H-D$ and $\alpha$ is a path connecting $h_1$ to $h_2$ in $X-K_1$. If $\beta$ is the first subpath of $\alpha$ with end points in $A$ and every other vertex in $X-A$, then the initial and end points of $\beta$ belong to some coset $aC$ for $a\in A$. By the definition of $K_1$, $aC\cap K=\emptyset$. Hence $\beta$ can be replaced by a path with vertices in $aC\subset A$. Inductively there is a path from $h_1$ to $h_2$ with vertices in $A-K$. 
\end{proof}

\begin{theorem} \label{OneEndSplit} 
Suppose $G$ is a finitely generated group and $G$ splits as $A\ast_C B$ or $A\ast_C$ where $C$ is not equal to $A$ or $B$ and $C$ is finite. If $H$ is 1-ended in $G$, then $H$ is a subgroup of a conjugate of $A$ or $B$. 
\end{theorem}

\begin{proof}
We consider the case where $G=A\ast_CB$, the amalgamated product case is completely similar. Choose finite generating sets for $A$ and $B$ each containing  the same generating set for $C$.
Let $\Gamma$ be the Cayley graph of $G$ where the generating set used is the union of the chosen generating sets for $A$ and $B$. The vertices of $\mathcal T$, the Bass-Serre tree for the splitting $A\ast_CB$, are the cosets of $A$ and $B$ in $G$. The edges of $\mathcal T$  are the cosets of $C$ in $G$. 
The first case we consider is when $H$ is locally finite. List the elements of $H$ as $h_1,h_2,\ldots$. Say $C$ contains $K$ elements. Let $H_k=\langle h_1,\ldots h_k\rangle$. Since $H_k$ is finite, it is a subgroup of a conjugate of $A$ or $B$. Say $H_i\subset g_i^{-1} Ag_i$ (or $g_i^{-1} Bg_i$) for some $g_i\in G$. Then note that $H_i\leq H_j$ for $i\leq j$.
If $i>K$ then $H_i$ then $H_{K+1}$ stabilizes the vertices $g_{K+1}A $ and $g_iA$ of $\mathcal T$. If $i\ne K+1$, then $H_{K+1}$ stabilizes the edges between these two vertices of $\mathcal T$ and hence is a subgroup of a conjugate of $C$. This is impossible since $C$ only contains $K$ elements. Instead, $i=K+1$ for all $i>K$. Since every element of $H$ is in $H_i$ for some $i>K$, we have that $H<g_{K+1}^{-1}Ag_{K+1}$. 

The final case we must consider is when $H$ contains a finitely generated infinite subgroup $H_1$. If $H_1$ is not a subgroup of $g^{-1} Ag$ or $g^{-1}Bg$ for some $g\in G$, then since $H_1$ acts on $\mathcal T$, it splits nontrivially over a subgroup $C_1$ of $C$. Say $H_1=J_1\ast _{C_1} J_2$. Choose $j_1\in J_1-C_1$ and $j_2\in J_2-C_1$. Then $j_1j_2$ is an element of infinite order in $H_1$ and elements $(j_1j_2)^n$ and $(j_2j_1)^n$ are on opposite sides of $C$ in $\Gamma$ for all $n\ne 0$. This implies that $H_1$ (and hence $H$) is not 1-ended in $G$, contrary to our hypothesis. If $H_1=J_1\ast_{C_1}$ then let $t$ be the stable letter of this splitting. The elements $t^n$ and $t^{-n}$ are on opposite sides of $C$ in $\Gamma$ for all $n\ne 0$. This is impossible since $H_1$ is 1-ended in $G$. 
Instead, every finitely generated infinite subgroup of $H$ is contained in $g^{-1}Ag$ or $g^{-1} Bg$ for some $g\in G$. 
Assume that $H_1<g^{-1}Ag$ for some $g\in G$. Let $h$ be any element of $H$ and let $H_2=\langle H_1\cup h\rangle$. Then $H_2<g_1^{-1}Ag_1$ or $H_2<g_1^{-1}Bg_1$ for some $g_1\in G$. It must be that $H_2<g^{-1}Ag$ since otherwise,  $H_1$ stabilizes the edges between these two distinct vertices. But that is impossible since stabilizers of edges are finite. We conclude that $H$ is a subgroup of a conjugate of $A$ or $B$.
\end{proof}

\begin{corollary} 
Suppose $G$ is a finitely generated group and $\mathcal D$ is a Dunwoody decomposition of $G$ (a graph of groups decomposition such that each edge group is finite and each vertex group is either 0 or 1-ended). Then $H$ is 1-ended in $G$ if and only if $H$ is a subgroup of a conjugate of a 1-ended vertex group of $\mathcal D$. 
\end{corollary} 
\begin{proof}
This follows inductively from Theorem \ref{OneEndSplit}.
\end{proof}

\begin{theorem}
Suppose $G$ is a finitely generated group that splits non-trivially as  $A\ast_C B$ or $A\ast_C$ where $C$ is not equal to $A$ or $B$ and $C$ is finite. Let $\mathcal T$ be the Bass-Serre tree for this splitting and suppose that the subgroup  $H$ is 1-ended in  $G$. Let $\mathcal T_1$ be the subtree of $\mathcal T$ containing all vertices $hA$ and $hB$ for $h\in H$, so that $H$ acts on $\mathcal T_1$ by graph isomorphisms. Theorem \ref{OneEndSplit} implies $H$ is a subgroup of an $A$ or $B$ coset $Q$ of $G$. Then

(1) If $D$ is the distance in $\mathcal T$ from $A$ to $Q$, then for all $h\in H$, the distance from $hA$ to $Q$ is $D$. Similarly for $B$. 

(2) If $h_1, h_2\in H$ are such that $h_1A\ne h_2A$ and $\alpha$ is the (even length) geodesic in $\mathcal T_1$ between $h_1A$ and $h_2A$, then the closest point $M$ of $\alpha$ to $Q$  is the midpoint (vertex) of $\alpha$. The geodesics from $h_1A$ to $Q$ and $h_2A$ to $Q$, both contain $M$. Furthermore, $M$ is stabilized by $h_2h_1^{-1}$. In particular, if $E$ is the first edge of $\alpha$, then the last edge of $\alpha$ is $h_2h_1^{-1}(E)$. Similarly for $B$.

(3) Either $H\subset A$ or $H\subset B$, or both $H\cap hA$ and $H\cap hB$ are finite for all $h\in H$. Note that if $H\subset A$ or $H\subset B$, then $\mathcal T_1$ is a single vertex. 

(4) If $H\not\subset A$ and $H\not\subset B$, then either all cosets $hA$ in $\mathcal T_1$ are of valence 1 and all cosets $hB$ in $\mathcal T_1$ have the same finite valence which is equal to 1 plus the index of $H\cap C$ in $H\cap A$, or all cosets $hB$ in $\mathcal T_1$ are of valence 1 and all cosets $hA$ in $\mathcal T_1$ have the same finite valence which is equal to 1 plus the index of $H\cap C$ in $H\cap B$. Furthermore, if $hA$ is not valence 1, then all but one edge at $hA$ ends with an (valence 1) $h'B$ vertex for some $h'\in H$. Similarly if $hB$ is not valence 1. 

(5) Assume that $A$ is closer to $Q$ than is $B$ (so that $B$ has valence 1 in $\mathcal T_1$). If $h_1,h_2\in H$ and $h_1A\ne h_2A$ then there are no $hA$ or $hB$ vertices between $h_1A$ and $h_2A$ in $\mathcal T_1$. Similarly for $B$. 

(6) The vertex $Q$ is the midpoint of the geodesic between $A$ and $hA$ for some $h\in H$. 
\end{theorem}
\begin{proof}
Let $\Gamma$ be the Cayley graph of $G$ on generating sets for $A$, $B$ and $C$.
Call the vertices of $\mathcal T$ that contain elements of $H$, $H$-vertices. Note that if $h\in H$ is such that $h\in gA$ for some $g\in G$, then $gA=hA$. So the $H$-vertices of $\mathcal T$ are all cosets of the form $hA$ or $hB$ for $h\in H$. 
Since $H$ fixes $Q$, conclusion (1) is immediate. Conclusion (2) follows from (1) and the following fact: If $\alpha$ is a geodesic in $\mathcal T_1$ and $M$ is a closest point of $\alpha$ to $Q$ then the geodesic from $Q$ to $M$, followed by the segment of $\alpha$ from $M$ to an end point of $\alpha$, is a geodesic. This fact and (1) certainly imply that $M$ is the midpoint of $\alpha$. Note that $h_2h_1^{-1}$ maps $h_1A$ to $h_2A$ and so maps the geodesic from $Q$ to $h_1A$ to the geodesic from $Q$ to $h_2A$. Hence $h_2h_1^{-1}$ stabilizes the segment between $M$ and $Q$. Conclusion $(2)$ is verified. 

If $hA\cap H$ is infinite, for some $h\in H$, then for any $h'\in H$, $(h'(hA\cap H))=(hh'A)\cap H$ is infinite and so $hA\cap H$ is infinite for all $h\in H$. Suppose $A\cap H$ is infinite. If $hA\ne A$ for some $h\in H$ and $E$ is an edge in $\mathcal T$ separating $A$ and $hA$, then the finite subset of $\Gamma$ corresponding to $E$ separates infinitely many elements of $H$ in $A$ from infinitely many elements of $H$ in $hA$. But this implies that $H$ is not 1-ended in $G$. Instead, if $A\cap H$ is infinite, then $hA=A$ for all $h\in H$, and so $H\subset A$ and $\mathcal T_1$ consists of a single vertex. Conclusion (3) is verified. 

\medskip

$(\dagger)$ From this point on we assume that $hA\cap H$ and $hB\cap H$ are finite for all $h\in H$. 

\medskip 

Consider the vertex $A$ of $\mathcal T_1$ and an element $h\in H$ such that $hA\ne A$. 
Suppose $E$ is an edge of $\mathcal T_1$ separating $A$ and $hA$. If there are infinitely many distinct elements of $H$ contained in the $\mathcal T_1$ vertices (cosets) on the $A$ side of $E$ and infinitely many elements of $H$ contained in the $\mathcal T_1$ vertices on the $hA$ side of $E$, then the finite graph corresponding to $E$ in $\Gamma$ would separate two infinite sets of elements of $H$. But then $H$ would not be 1-ended in $\Gamma$. 

Instead, there are only finitely many elements of $H$ in the union of all $hA$ and $hB$ cosets on exactly one side of $E$. We are interested in the case where $A$ is a vertex of $E$. (See Figure \ref{FigEndPair1}). The last part of conclusion $(2)$ implies that $hE$ is the last edge of the geodesic between $A$ and $hA$. 

\begin{figure}
\vbox to 3in{\vspace {-2in} \hspace {-.1in}
\hspace{-1 in}
\includegraphics[scale=1]{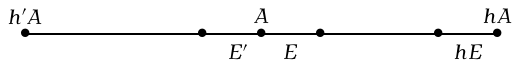}
\vss }
\vspace{-3in}
\caption{Separating elements of $H$} 
\label{FigEndPair1}
\end{figure}

Suppose there are infinitely many elements of $H$ on the $A$ side of $E$ and say there are $K$ elements of $H$ on the $hA$ side of $E$.  Now consider the graph isomorphism $h$ acting on $\mathcal T_1$. 
Since $hE$ is the last edge of the geodesic between $A$ and $hA$, then there are infinitely many elements of $H$ on $hA$ side of $hE$, but all of these elements of $H$ are on the $hA$ side of $E$. This contradicts the fact that there are only $K$ elements of $H$ on the $hA$ side of $E$. 
This implies: 

\medskip

$(I)$ Suppose $h\in H$, $A\ne hA$ and $E$ is the first edge of a geodesic in $\mathcal T_1$ from $A$ and $hA$, then there are only finitely many elements of $H$ on the $A$ side of $E$ in $\mathcal T_1$, and infinitely many elements of $H$ on the $hA$ side of $E$.  
Similarly for $B$. 
 

\medskip

Next we show:

\medskip

$(II)$ Under the hypotheses of $(I)$,  if $E'\ne E$ is an edge containing $A$, then the branch of $\mathcal T_1$ on the side of $E'$ opposite $A$ cannot contain a vertex $h'A$ for any $h'\in H$. Similarly for $B$. 

\medskip

Otherwise, $(I)$ implies that there are only finitely many elements of $H$ on the $A$ side of $E'$. (See Figure \ref{FigEndPair1}.) But all  (infinitely many) elements of $H$ on the $hA$ side of $E$ are on the $A$ side of $E'$, a contradiction and $(II)$ is verified.




\medskip

Again we assume the hypotheses of $(I)$: Suppose $h\in H$, $A\ne hA$ and $E$ is the first edge of a geodesic $\alpha$ between $A$ and $hA$. Let $F$ be the edge connecting $A$ and $B$ in $\mathcal T$. Then $F$ is an edge of $\mathcal T_1$. 
If $F\ne E$, then $(II)$ (with $F$ playing the role of $E'$) implies:

\medskip

\noindent $(\ast)$ There is no $h'\in H$ such that $h'A$ is a vertex of $\mathcal T_1$ on the side of $F$ opposite $A$. 

\medskip

\begin{figure}
\vbox to 3in{\vspace {-1.5in} \hspace {.6in}
\hspace{-1 in}
\includegraphics[scale=1]{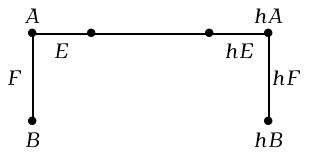}
\vss }
\vspace{-1.7in}
\caption{Separating elements of $H$} 
\label{FigEndPair2}
\end{figure}

\noindent Since $hB$ is adjacent to $hA$ we have that $B\ne hB$. We have the configuration of Figure \ref{FigEndPair2} (see the last part of conclusion $(2)$). We apply $(II)$ to $B$ and $F$ to see that the branch of $\mathcal T_1$ on the side of $F$ opposite $B$ cannot contain a vertex $h'B$ for any $h'\in H$. Combining with $(\ast)$, we have that if $F\ne E$, then $B$ has valence 1 in $\mathcal T_1$.  Note that if $F=E$ then $A$ is valence 1 in $\mathcal T_1$. In any case: 

\medskip

$(III)$ Under the hypotheses of $(I)$, either $A$ is on the geodesic between $B$ and $hB$ and $B$ has valence 1 in $\mathcal T_1$ or $B$ is on the geodesic between $A$ and $hA$ and $A$ has valence 1 in $\mathcal T_1$. 

Without loss, assume that $B$ (and hence $h'B$ for all $h'\in H$) has valence 1 in $\mathcal T_1$.

\medskip

Next suppose $E'$ is an edge containing $A$, $F\ne E'\ne E$ and $h_1B$ is on the side of $E'$ opposite $A$. Then $h_1A$ is adjacent to $h_1B$ and $(II)$ implies that $h_1A=A$. Then $h_1\in A$. Since $H\cap A$ is finite, we have: 

\medskip

$(IV)$ The valence of $A$ in $\mathcal T_1$ is finite and with the exception of $E$, each edge $E'$ of $\mathcal T_1$ containing $A$ has as other vertex $h'B$ for some $h'\in H\cap A$. Observe that for the number of $H\cap C$ cosets in  $H\cap A$ is the number of edges at the vertex $A$ in $\mathcal T_1$ with end point $hB$ for some $h\in H$. 

\medskip

This completes the proof of conclusion (4).

Since $h'B$ has valence 1 for all $h'\in H$, there is no such vertex in the geodesic between $A$ and $hA$. Statement $(II)$ implies there is no vertex $h'A$ (with $h'\in H$), between $A$ and $hA$. Together this implies conclusion conclusion (5). 

For each $h\in H$ let $M_h$ be the midpoint of the geodesic between $A$ and $hA$. Choose $h\in H$ such that the distance between $M_h$ and $Q$ is minimal over all $h\in H$. If $h'\in H$, then  
Conclusion $(2)$ implies that $M_h$ and $M_h'$ are on the geodesic from  $Q$ to $A$. By the minimality of $h$, $M_{h}$ is between $M_{h'}$ and $Q$. Conclusion $(2)$ implies that $h'$ stabilizes $M_{h'}$ and $Q$ and hence $h'$ stabilizes $M_h$ for all $h'\in H$. This implies that $M_h=Q$. 
\end{proof}

We wonder what can be said about inaccessible finitely generated groups. 

\begin{conjecture}
The subgroup $H$ of the finitely generated group $G$ is 1-ended in $G$ if and only if for each splitting $A\ast_CB$ of $G$ (where $A\ne C\ne B$ and $C$ is finite), $H$ is a subgroup of a conjugate of $A$ or $B$
\end{conjecture}
Certainly if $H$ is 1-ended in $G$ then $H$ is a subgroup of a conjugate of $A$ or $B$. 

\begin{definition}
Suppose $\Gamma$ is a Cayley graph of $G$ with respect to a finite generating set. A subgroup $H$ of $G$ is {\it locally finite at $\infty$} in $G$ if for any compact subset $C\subset \Gamma$, only finitely many cosets $gH$ are separated by $C$. This means that for all but finitely many cosets $gH$ of $H$ in $G$ any two points of $gH$ can be joined by a path in $\Gamma-C$. 
\end{definition}

\begin{theorem} 
If $H$ is finitely generated, then $H$ is locally finite at $\infty$ in $G$. 
\end{theorem} 

\begin{theorem}
Suppose $H$ and $K$ are infinite, locally finite at $\infty$ subgroups of $G$,  $H\cap K$ is infinite  and both $H$ and $K$ are 1-ended in $G$. Then $\langle H\cup K\rangle$ is 1-ended in $G$. 
\end{theorem}

\begin{proof}
Again, $\Gamma$ is a Cayley graph of $G$ with respect to a finite set. Let $C$ be compact in $\Gamma$.
Let $g_1H,\ldots, g_nH$ be the cosets of $H$ that are separated by $C$ and $a_1K,\ldots, a_mK$ be the cosets of $K$ that are separated by $C$. Choose $D_i$ compact for $g_iH$ and $C$ so that any two points of $g_iH-D_i$ can be joined by a path in $\Gamma-C$. 

Consider the compact set $g_i^{-1} C$ and choose $D_i'$ such that any two points of $H-D_i'$ can be joined in $\Gamma -g_i^{-1}(C)$. Then any two points of $g_iH-g_iD'$ can be joined in $\Gamma-C$. Now let $D_i=g_iD_i'$.
Similarly choose $E_i$ compact for $a_iK$ and $C$ so that any two points of $a_iK-D_i$ can be joined by a path in $\Gamma-C$. 
Let $D=(\cup _{i=1}^n D_i)\cup (\cup_{i=1^m}E_i)$. 

\medskip

$(\dagger )$ Observe that for any $g\in G$, any two points of $gH-D$ (or $gK-D$) can be joined by a path in $\Gamma-C$. 

\medskip

If $H$ has finite index in $\langle H\cup K\rangle$ we are finished by Theorem \ref{FI}.
Choose $a\in \langle H\cup K\rangle$ so that any two points of $aH$ can be joined in $\Gamma-C$. 
It is enough to show that any point of $\langle H\cup K\rangle-D$ can be joined to $aH$ by a path in $\Gamma-C$. If $b\in \langle H\cup K\rangle -D$, then there are elements $h_1,\ldots h_n$ of $H$ and $k_1,\ldots, k_n$ of $K$ such that $a=bh_1k_1\cdots h_nk_n$. 

Notice that $bh_1H=bH$ and $(\dagger)$ implies that there is a path in $\Gamma-C$ from $b$ to a point 
$$p_1\in bh_1\langle H\cap K\rangle (\subset bH\cap bh_1K).$$ 
Since $bh_1K=bh_1k_1K$, $(\dagger)$ implies that there is a path in $\Gamma-C$ from $p_1$ to a point 
$$q_1\in bh_1k_1 \langle H\cap K\rangle (\subset bh_1K\cap bh_1k_1H).$$ 

Inductively, we have a path in $\Gamma-C$ from $b$ to $aH$.
\end{proof}

\begin{corollary}
Suppose that $H$ and $K$ are finitely generated and 1-ended in $G$ and $H\cap K$ is infinite. Then $\langle H\cup K\rangle$ is 1-ended in $G$. 
\end{corollary}

\newpage
\subsection{One-Ended Groups}\label{1eg}

Dunwoody's accessibility result allows many questions about certain classes of finitely presented groups to be reduced to the same question about 1-ended groups in that class. It is important to understand algebraic and geometric conditions on a group that force the group to be 1-ended. In this section, we list a number of such results.

\subsubsection{Knot Groups and the Covering Conjecture}\label{KGs} 
This section is for historical perspective. 
In his thesis, E. Specker \cite{Sp49} (page 329) proved the following: 

\begin{theorem} [(page 329), \cite{Sp49}] 
The conjecture that  ``the knot space $\mathbb S^3-K$ is aspherical", is  equivalent to the conjecture  ``the knot group $\pi_1(\mathbb S^3-K)$ has one or two ends". 
\end{theorem}

It seemed reasonable that  $\pi_1(\mathbb S^3-K)$ should have 2-ends if and only if $K$ is {\it algebraically unknotted} i.e. when $\pi_1(\mathbb S^3-K)$ is $\mathbb Z$. 
In his early attempts to prove asphericity of knots C. D. Papakyriakopoulos proved (in the Annals of Mathematics paper \cite{Pap55}): 
\begin{theorem} [Corollary 3, \cite{Pap55}] \label{knot}  
If the asphericity of knots holds, then a knot group has one or two ends according as the knot is algebraically knotted or unknotted. 
\end{theorem}

In 1957 Papakyriakopoulos proved Dehn's Lemma and the Sphere Theorem. From this he can conclude the asphericity of knots. 

\begin{theorem} [Corollary 2, \cite{Pap57}, \cite{Pap57b}]
If $F$ is  a connected graph or knot, then $\mathbb S^3-F$ is aspherical.
\end{theorem} 

Combining with Specker's Theorem: 
\begin{theorem} [Theorem 3, \cite{Pap57}, \cite{Pap57b}] 
All knot groups have 1 or 2-ends.
\end{theorem} 

Also in 1957, Papakyriakopoulos considered ends of fundamental groups of $3$-manifolds with boundary (see \cite{Pap57c}).

\subsubsection{Normal and Commensurated Subgroups}\label{normalE}

\begin{theorem} [Proposition 2.1, \cite{Co72}] \label{cohen} 
If $N\leq A\leq G$ are groups with $N$ a non-locally finite normal subgroup of $G$ and $A$ a finitely generated subgroup of infinite index in $G$, then $G$ is 1-ended.
\end{theorem}

If $1\to N\to G\to K\to 1$ is a short exact sequence of finitely generated infinite groups then $G$ is 1-ended. Since non-trivial finitely generated free groups have more than 1-end, this implies the well known result that a free group contains no infinite finitely generated normal subgroup of infinite index: 

(Suppose $F$ is free of infinite rank and $N$ is finitely generated infinite and of infinite index in $F$. Let $B$ be a basis for $F$. Let $B_1$ be a finite subset of $B$ such that $N$ is a subgroup of the finitely generated free group $\langle B_1\rangle=F(B_1)$.  Let $b\in B-B_1$, then $N$ is normal in and has infinite index in $F(B_1\cup \{b\})$. But finitely generated free groups are not 1-ended.)

The Theorems \ref{superC} and \ref{subCend} give broad generalizations of this result.

\begin{proposition} [Proposition 4.1,\cite{CM2}] \label{superC}  Suppose $A$ is a finitely generated, infinite subgroup of infinite index in a finitely generated group $G$, and $gAg^{-1}\cap A$ is infinite for all $g\in G$ (in particular if $A$ is ``commensurated" in $G$ - see Definition \ref{Dcomm}). Then $G$ is 1-ended.
\end{proposition}

\begin{proof}
Suppose $S$ is a finite set of generators for $G$ containing a generating set $S_A$ for $A$. Let $\Gamma$ be the Cayley graph of $G$ with respect to $S$ and let $\Gamma_A$ be the Cayley graph of $A$ with respect to $S_A$. We consider $\Gamma_A$ to be a subset of $\Gamma$ containing $\ast$, the identity vertex. Let $C$ be a finite subcomplex of $\Gamma$. List elements $g_1,\ldots ,g_n$ of $G$ such that $g_iA\ne g_jA$ for $i\ne j$ and $g\Gamma_A\cap C\ne\emptyset$ if and only if $gA=g_iA$ for some $i\in\{1,\ldots , n\}$. Choose $g_0$ such that $g_0\Gamma_A\cap C=\emptyset$. Let $D$ be a finite subcomplex of $\Gamma$ containing $C$ and all bounded components of $g_i\Gamma_A-C$ for all $i\in \{1,\ldots ,n\}$. 

It suffices to show that any vertex of $\Gamma-D$ can be joined by a path in $\Gamma-C$ to $g_0$. Suppose $v$ is a vertex of $\Gamma-D$. 

First we consider the case $v\Gamma_A\cap C=\emptyset$. Choose $y$ (in the infinite set) $(v\Gamma_Av^{-1})\cap (g_0\Gamma_Ag_0^{-1})$ such that $d(y, C)> max\{ |v|, |g_0|\}$. Then there are paths from $y$ to $v\Gamma_A$ and from $y$ to $g_0\Gamma_A$ avoiding $C$. Hence there is a path from $v$ to $g_0$ avoiding $C$ and the first case is finished.

Next suppose $v\Gamma_A\cap C\ne \emptyset$. Then $v$ is in an unbounded component $K$ of $v\Gamma_A-C$. Let $N=max\{|g_0|,\ldots, |g_n|\}$. Choose $k$ a vertex of  $K$ such that $d(k,C)>N$. Then there are paths from $v$ to $k$ and from $k$ to $kg_i$ for each $i\in \{0,\ldots ,n\}$, all avoiding $C$. At least one of $kg_iA$ does not intersect $C$, so by the first case we can connect $v$ to $g_0$ avoiding $C$. 
\end{proof}

\begin{definition}\label{Dcomm} The {\it commensurator} of a subgroup $H$ of a group $G$ is the set of all $g\in G$ such that $gHg^{-1}\cap H$ has finite index in both $gHg^{-1}$ and $H$. The commensurator of $H$ in $G$ is denoted $Comm_H(G)$. 
 A subgroup $H$ of a group $G$ is {\it commensurated} in $G$, written $H\prec G$ if $Comm_H(G)=G$. 
 \end{definition}

\begin{theorem} [Theorem 2.3, \cite{CM13}] \label{T1} 
If $H$ is a subgroup of $K$ and $k\in K$ then $k\in Comm_H(K)$ iff there are finite subsets $A$ and $B$ of $K$ such that for each $h\in H$ there is an $a\in A$ and  $b\in B$ such that $ha\in kH$, and $khb\in H$. (Equivalently, there is $a\in A$ such that  $h(ak^{-1})\in kHg^{-1}$ and $b\in B$ such that $khk^{-1}(kb)\in H$.) 
\end{theorem}

As a direct corollary we have
\begin{corollary} \label{C1}
Suppose the group $H$ is commensurated in the group $K$ and $K$ is a subgroup of the finitely generated group $G$. Let $S$ be a finite generating set for $G$ and $\Gamma$ be the Cayley graph of $G$ with respect to $S$. Then for each element $k\in K$ there is an integer $N_k$ such that each point of $H$ in $\Gamma$ is within $N_k$ of a point of $kH$ and each point of $kH$ is within $N_k$ of $H$. In other words, the Hausdorff distance between $H$  and $kH$ is $\leq N_k$. 
\end{corollary}

We are most interested in Corollary \ref{Stabv} which follows easily from the next result. 
The first two parts of the next theorem are elementary and well known. We are more interested in part 3).

\begin{theorem} \label{StabP} 
Suppose $G$ is a group acting bijectively on the set of points $V$.  Let $G_v$ be the subgroup of $G$ that stabilizes $v\in V$. 

(1) If  $v, w\in V$ and $g\in G$ such that $v=gw$ then $G_w= g^{-1} G_vg$. 

(2) If $V=\{v_0,\ldots, v_n\}$ is finite, then $G_{v_i}$ has finite index in $G$.

(3) Suppose $v\in V$ and $g\in G$ are such that both $G_v(gv)$ and $G_v(g^{-1}v)$ are finite subsets of $V$. Then $G_v\cap gG_vg^{-1}$ has finite index in both $G_v$ and $gG_vg^{-1}$. In particular, if $G_v(gv)$ is finite for all $g\in G$, then $G_v$ is commensurated in $G$. (Note that it may be the case that $G_v$ is infinite and $G_vw$ is finite.)

\end{theorem}
\begin{proof} (1):  If $v=gw$ and  $h\in G_v$, then 
$$g^{-1}hgw=g^{-1}hv=g^{-1}v=w$$
So, if $v=gw$, then $g^{-1}G_vg<G_w$. This implies (since $w=g^{-1}v$) that  $gG_wg^{-1}<G_v$ and so $G_w<g^{-1}G_vg$.
Combining:
$$g^{-1}G_vg<G_w<g^{-1}G_vg$$
This completes the proof of Part (1)

\medskip

(2): Let $G_i=G_{v_i}$. We show $G_0$ has finite index in $G$. Choose $g_i$ such that $g_iv_0=v_i$. Let $g\in G$ and say $gv_0=v_j$. Then $g_j^{-1}gv_0=v_0$,  $g_j^{-1}g\in G_0$ and $g\in G_0g_j$. This means $G=\cup_{i=0}^n G_0g_i$ and $G_0$ has finite index in $G$. 

\medskip

(3): Let $gv=w$. The group $G_v$ acts bijectively on the finite set $G_vw$. Certainly $w\in G_vw$ and by (2), the stabilizer of $w$ in $G_v$ (which is $G_v\cap G_w$) has finite index in $G_v$. Since $v=g^{-1}w$, part (1) implies $G_w=gG_vg^{-1}$. We have $G_v\cap gG_vg^{-1}$ has finite index in $G_v$. 
Similarly if $g^{-1}v=w'$, we have $G_v\cap g^{-1}G_vg$ has finite index in $G_v$. Conjugating we have $gG_vg^{-1}\cap G_v$ has finite index in $gG_vg^{-1}$. 
\end{proof}

\begin{corollary}\label{Stabv}
Suppose $G$ is a group acting as graph isomorphisms on a connected locally finite graph $\Gamma$. If $v$ is a vertex of $\Gamma$ then  $G_v$ (the stabilizer of $v$ in $G$) is commensurated in $G$. 
\end{corollary}
\begin{proof}.  For $g\in G$, let $\alpha$ be an edge path from $v$ to $gv$. Then for each $h\in G_v$, the path $h\alpha$ begins at $v$ and ends at $hgv$. If $\alpha$ has length $n$ then each point of $G_v(gv)$ is within $n$ of $v$. As $\Gamma$ is locally finite, $G_v(gv)$ is finite for all $g\in G$. Apply Theorem \ref{T1}(3) to the vertex set of $\Gamma$.
\end{proof}

\begin{theorem}\label{subCend} 
Suppose  
$$C_1\prec C_2\prec \cdots \prec C_n=G$$
is a commensurated sequence of groups, $C_1$ and $G$ are  finitely generated, $C_1$ is infinite, and  $C_1$ has infinite index in $G$, then $G$ is 1-ended. In particular, no free group contains an infinite finitely generated subcommensurated subgroup of infinite index.
\end{theorem} 

\begin{proof}  
If $\Gamma(G,S)$ is the Cayley graph of $G$ with respect to the finite generating set $S$, then say the subgroup $H$ of $G$ is {\it 1-ended} in $\Gamma$ if for any compact set $C$ of $\Gamma$ there is a compact set $D$ in $\Gamma$ such that any two points of $H$ in $\Gamma-D$ can be connected by a path in $\Gamma-C$. If $C_i$ has finite index in $C_{i+1}$ then $C_i$ is commensurated in $C_{i+1}$ (Lemma 3.7, \cite{CM13}). Hence we may assume that $C_i$ has infinite index in $C_{i+1}$. 
We show that for $i\geq 2$, $C_i$ is 1-ended in $\Gamma$. (Note that if $G$ is 1-ended in $\Gamma$ then $G$ is 1-ended.)

Assume the generating set $S$ of $G$ contains generators for $C_1$. Let $C$ be compact in $\Gamma$. Choose distinct $C_1$-cosets of $C_2$, $h_1C_1, \ldots, h_nC_1$ such $h_i\in C_2\cap C$, and $C\cap C_2\subset \cup_{i=1}^n h_iC_1$. Since $C_1$ has infinite index in $C_2$, there is $c\in C_2$ such that $cC_1\cap C=\emptyset$.  For each $i\in \{1,\ldots, n\}$ choose an integer $N_i$ (see Corollary \ref{C1}) such that there is an edge path of length $\leq N_i$ from each point of $h_iC_1$ to a point of $cC_1$. Let $N=max\{N_1,\ldots, N_n\}$ and $D_2'=st^N(C)$. For each point $p\in \cup_{i=1}^nh_iC_1-D_2'$, there is an edge path in $\Gamma-C$ (of length $\leq N$) from $p$ to a point of $cC_1$. Since $cC_1$ avoids $C$, there is an edge path from $p$ to $c$ avoiding $C$. 

If $d\in C_2-\cup _{i-1}^nh_iC_1$, let $N_d$ be such that there is an edge path of length $\leq N_d$ from each point of $dC_1$ to $cC_1$. Choose any point $q$ of $dC_1-st^{N_d}(C)$. Then there is a path from $q$ to $cC_1$ avoiding $C$ and so a path from $q$ to $c$ avoiding $C$. As any point of $dC_1$ is connected to $q$ by a path in $dC_1$ (and so avoiding $C$)  each point of $dC_1$ is joined to $c$ by a path avoiding $C$. (In particular, if $C_2\cap C=\emptyset$, then any two points of $ C_2$ can be joined by a path in $\Gamma-C$.) We have shown $C_2$ is 1-ended in $\Gamma$. 

If $g\in G$ then we have shown that shows that there is a compact $D_g$ such that and two points of  $C_2$ in $\Gamma-D_g$ can be joined in $\Gamma-g^{-1}C$; and if $C_2\cap g^{-1}C=\emptyset$ then any two points of $C_2$ can be joined by a path in $\Gamma-g^{-1}C$. Translating by $G$, any two points of $gC_2$ in $\Gamma-gD_g$ can be joined in $\Gamma-C$; and if $gC_2\cap C=\emptyset$ then any two points of $gC_2$ can be joined by a path in $G-C$. Since only finitely many cosets $gC_2$ intersect $C$, there is a single compact set $D_2$ such that for every $g\in G$, any two points of $gC_2-D_2$ can be joined by a path in $\Gamma-C$.

Our inductive hypothesis is $H(i)$: For $i\geq 3$ there is a compact set $D_i$ such that for any $g\in G$, any two points of  $gC_i-D_i$ can be joined by a path in $\Gamma-C$ and if $gC_i\cap C=\emptyset$ then then any two points in $gC_i$ can be joined by a path in $\Gamma-C$. We need only show that $H(i)$ implies $H(i+1)$. 

Assume $H(i)$. We follow the above argument closely.  Let $k_1C_i, \ldots, k_mC_i$ be distinct cosets of $C_i$ in $C_{i+1}$ such that  $k_j\in C_{i+1}\cap C$, and $C\cap C_{i+1}\subset \cup_{i=1}^m k_iC_{i}$. Since $C_i$ has infinite index in $C_{i+1}$, there is $c\in C_{i+1}$ such that $cC_i\cap C=\emptyset$.  For each $j\in \{1,\ldots, m\}$ choose an integer $M_j$ (see Corollary \ref{C1}) such that there is an edge path of length $\leq M_j$ from each point of $k_jC_i$ to a point of $cC_i$. Let $M=max\{M_1,\ldots, M_m\}$ and $D_{i+1}'=st^M(C)$. For each point $p\in k_jC_i-D_{i+1}'$, there is an edge path in $\Gamma-C$ (of length $\leq M$) from $p$ to a point of $cC_i$. Since $cC_i$ avoids $C$, there is an edge path from $p$ to $c$ in $\Gamma-C$. We have shown that for any point $p$ of $\cup_{j=1}^mk_jC_i-D_{i+1}'$ there is a path from $p$ to $c$ avoiding $C$.

If $d\in C_{i+1}-\cup _{j=1}^mk_jC_i$, let $M_d$ be such that there is an edge path of length $\leq M_d$ from each point of $dC_i$ to $cC_i$. Choose any point $q$ of $dC_i-st^{M_d}(C)$. Then there is a path from $q$ to $cC_i$ avoiding $C$ and so a path from $q$ to $c$ avoiding $C$. As any point of $dC_i$ is connected to $q$ by a path in $\Gamma-C$  each point of $dC_i$ is joined to $c$ by a path avoiding $C$. We have shown $C_{i+1}$ is 1-ended in $\Gamma$. 

If $g\in G$ then we have shown that shows that there is a compact $D_g$ such that and two points of  $C_{i+1}$ in $\Gamma-D_g$ can be joined in $\Gamma-g^{-1}C$; and if $C_{i+1}\cap g^{-1}C=\emptyset$ then any two points of $C_{i+1}$ can be joined by a path in $\Gamma-g^{-1}C$. Translating by $g$, any two points of $gC_{i+1}$ in $\Gamma-gD_g$ can be joined in $\Gamma-C$; and if $gC_{i+1}\cap C=\emptyset$ then any two points of $gC_{i+1}$ can be joined by a path in $G-C$. Since only finitely many cosets $gC_{i+1}$ intersect $C$, there is a single compact set $D_{i+1}$ such that for every $g\in G$, any two points of $gC_{i+1}-D_{i+1}$ can be joined by a path in $\Gamma-C$.
\end{proof}

Note that if $C$ is a subgroup of finite index in the group $A$, then $C$ is commensurated in $A$. 

\begin{corollary}\label{FIcomm} 
Suppose $G=A\ast_C B$, where $A$ and $B$ are finitely generated and $C$ has finite index (at least 2) in both $A$ and $B$. Then (certainly $C$ is commensurated in both $A$ and $B$ and hence in $G$) $G$ has 1-end.
\end{corollary}

\subsubsection{Combination Results and Graph Products}\label{compGP} 
  
 \begin{theorem} [Theorem 3, \cite{M4}] \label{combE}  
Suppose $G$ is a finitely generated group, $A$ and $B$ are finitely generated 1 or 2-ended subgroups of $G$, $A\cup B$ generates $G$ and $A\cap B$ is infinite. Then $G$ is 1 or 2-ended.
\end{theorem}
\begin{proof}
If $A$ and $B$ are 2-ended, then say $a$ ($b$) is an element of infinite order in $A$ ($B$) which generates a normal subgroup of finite index in $A$ ($B$). As $\langle a\rangle$ has finite index in $A$, there is $m$ such that $a^m\in C$. Similarly there is $n$ such that $b^n\in C$. Then there is $k,l$ such that $c\equiv a^{mk}= b^{ln}\in C$ and $\langle c\rangle$ is normal in $A$ and $B$. Then $\langle c\rangle$ is normal in $G$. If $\langle c\rangle$ has finite index in $G$, then $G$ is 2-ended. Otherwise, Theorem \ref{subCend} implies $G$ is 1-ended. 

From this point on, assume that $A$ is 1-ended, $B$ is 1 or 2-ended and $G$ has more than one end. Then $G$ splits as $D\ast_CE$ where $C$ is finite. Let $T$ be the Bass-Serre tree for this splitting. Then $A$ and $B$ act on $T$. Since $A$ is 1-ended, it must stabilize a vertex $gF$ of $T$ (where $g\in G$ and $F\in \{D, E\}$). Then 

$$A <gFg^{-1}$$ 

If $B$ is 1-ended, then $B$ stabilizes $g'F'$ for some $g'\in G$ and $F'\in \{D,E\}$. Note that $A\cap B$ stabilizes the geodesic between $gF$ and $g'F'$ in $T$. If $gD\ne g'D'$ then $A\cap B$ stabilizes an edge of $T$, but the stabilizer of any edge of $T$ is finite. Hence $G=\langle A\cup B\rangle$ stabilizes $gD$ and so is a subgroup of $gDg^{-1}$. But then $G$ is a subgroup of $hDh^{-1}$ for every $h\in G$ and so $G$ stabilizes $hD$ for every $h\in G$. This implies $G$ stabilizes edges of $T$ and so is finite (which it is not).

Finally, assume $A$ is 1-ended and $B$ is 2-ended. Say $b\in B$ and $\langle b\rangle$ generates a normal infinite cyclic subgroup of $B$. Choose $m$ such that $b^m\in A\cap B$.  Let $\hat g\in G$ and write $\hat g$ as $a_nb_n\cdots a_1b_1$ where $a_i\in A$ and $b_i\in B$ for all $i$. The group $b_1Ab_1^{-1}$ is 1-ended and 
$$b_1\langle b^m\rangle b_1^{-1}=\langle b^m\rangle<b_1Ab_1^{-1}\cap A$$ 
By the previous argument, $A_1\equiv \langle A, b_1Ab_1^{-1}\rangle$ is 1-ended (contains $b^m$ and $b_1Ab_1^{-1}$) and is a subgroup of $gFg^{-1}$ (recall $g\in G$ and $F\in \{A,B\}$). The group $a_1A_1a_1^{-1}$ is 1-ended and $A<a_1A_1a_1^{-1}\cap A_1$,  so $A_1'\equiv \langle A_1, a_1A_1a_1^{-1}\rangle$ is 1-ended, contains both $b^m$ and $a_1b_1Ab_1^{-1}a_1^{-1}$, and is a subgroup of $gFg^{-1}$. Inductively, $A_n'\equiv \langle A_n,a_nA_na_n^{-1}\rangle$ is 1-ended contains both $b^m$ and $\hat gA\hat g^{-1}$, and is a subgroup of $gFg^{-1}$. In particular, $N_G(A)$ (the normal closure of $A$ in $G$) is a subgroup of $gFg^{-1}$. But then $N_G(A)<hFh^{-1}$ for every $h\in G$ and $N_G(A)$ stabilizes edges of $T$. This is impossible since $N_G(A)$ is infinite.
\end{proof}

\begin{definition} 
Given a graph $\Gamma$ with vertex set $V(\Gamma)$, and a group $G_v$ for each $v\in V(\Gamma)$, the {\it graph product} for $(\Gamma, \{G_v\}_{v\in V(\Gamma)})$
is the quotient of the free product of the $G_v$ by the normal closure of the set of all commutators $[g,h]$ where $g$ and $h$ are elements of adjacent vertex groups. This group is denoted by $G_\Gamma$.
\end{definition}
Every right angled Coxeter group and right Angled Artin group is a graph product where vertex groups are copies of $\mathbb Z_2$ and $\mathbb Z$ respectively.

Olga Varghese completely determines the number of ends of a graph product of finitely generated groups.

\begin{theorem} [(O. Varghese, \cite{OV}] \label{OV}   Suppose $G$ is a finitely generated graph product group on the graph $\Gamma$ (so that each $G_v$ is finitely generated and $\Gamma$ is finite) and $G$ has more than one end, then either:

(i) $\Gamma$ is a complete graph such that one vertex group has more than one end and all others are finite, or

(ii) $G$ visually splits over a finite group. This means that there are non-empty full subgraphs $\Gamma_1$  and $\Gamma_2$ of $\Gamma$  (neither containing the other) such that $\Gamma=\Gamma_1\cup \Gamma_2$, (so $\Gamma_1\cap \Gamma_2$ separates $\Gamma$) and $\langle \Gamma_1\cap \Gamma_2\rangle$ is a finite group. In particular, $G$ visually decomposes as the (non-trivial) amalgamated product
$$G \cong \langle \Gamma_1\rangle\ast _{\langle\Gamma_1\cap \Gamma_2\rangle} \langle\Gamma_2\rangle$$  
\end{theorem}

\begin{theorem} [(O. Varghese, \cite{OV}] \label{OV2}   Let $G_\Gamma$ be a graph product of non-trivial finitely generated groups. Decompose $\Gamma$ as a join $\Gamma_1\ast  \Gamma_2$ where $\Gamma_1$ is a subgraph generated by the vertices which are adjacent to all the vertices of $\Gamma$. The graph product $G_\Gamma$ is 2-ended if and only if either:

(i) The graph $\Gamma$ is complete with exactly one 2-ended vertex group and all other vertex groups being finite, or

(ii) The subgraph $\Gamma_1$ is complete with finite vertex groups (and so $G_{\Gamma_1}$ is finite) and $G_{\Gamma_2} = \mathbb Z_2\ast \mathbb Z_2$.
\end{theorem}

\subsubsection{Artin and Coxeter Groups}\label{ACGE}
\begin{definition} \label{Artin}  Let $\Lambda$ be a  graph with vertex set $S=\{s_1,\ldots, s_n\}$.  If there is an edge between distinct elements $s_i$ and $s_j$  (for $i<j$) it is labeled by an integer $m_{ij}\geq 2$. Let $(s_i,s_j)_{m_{ij}}$ be the word of length $m_{ij}$, beginning with $s_i$ and alternating letters between $s_i$ and $s_j$. The {\it Artin group} on the (Artin) diagram $\Lambda$ is a group with presentation $\langle S\ |\ R\rangle$ where $R$ is the set of all words $(s_i,s_j)_{m_{ij}}=(s_j,s_i)_{m_{ij}}$ such that there is an edge between $s_i$ and $s_j$ for $i<j$.
\end{definition} 

\begin{theorem}\label{ArtinE} 
If $A$ is a finitely generated Artin group on a connected Artin diagram with at least 2 vertices, then $G$ is 1-ended.
\end{theorem}
\begin{proof}
Let $\Lambda$ be the Artin diagram with vertex set $S=\{s_1,\ldots, s_n\}$. If there is an edge between distinct element $s_i$ and $s_j$, the subgroup $A_{ij}=\langle s_i,s_j\rangle$ has presentation $\langle s_i,s_j\ |\  (s_i,s_j)_{m_{ij}}=(s_j,s_i)_{m_{ij}}\rangle$. The group $A_{ij}$ is a 1-relator group and the relation is not a proper power, so $A_{ij}$ is torsion free. If $A_{ij}$ had more than 1-end, Theorem \ref{Stall} would imply $A_{ij}$ splits as a non-trivial free product. Since it is a 2-generator group, it would be free on two generators (which it is not). Instead, $A_{ij}$ has 1-end. Now inductively apply Theorem \ref{combE}
\end{proof}

\begin{definition}\label{Coxeter}  Let $\Lambda$ be a  graph with vertex set $S$.  If there is an edge between distinct elements $s_i$ and $s_j$  of $S$ it is labeled by an integer $m_{ij}\geq 2$. The {\it Coxeter group} $W$ with {\it presentation diagram} $\Lambda$ has presentation: 
$$\langle S \ | \ s^2=1 \hbox{ for all } s\in S, (s_is_j)^{m_{ij}}\rangle$$
The pair $(W,S)$ is called a {\it Coxeter system}. 
\end{definition}

The finite Coxeter groups are well understood and so the following results explain how to determine the number of ends of a finitely generated Coxeter group from a simple examination of a presentation graph. 

\begin{theorem} [Corollary 16, \cite{MT09}] \label{CoxE}   Suppose $W$ is  a finitely generated Coxeter group with presentation diagram $\Lambda$. Then $W$ has more than one end if and only if $\Lambda$ contains a complete separating subgraph, the vertices of which generate a finite subgroup of $W$.
\end{theorem}

\begin{theorem} [Corollary 17, \cite{MT09}] \label{Cox2E}  A finitely generated Coxeter group with system $(W,S)$ and corresponding diagram $\Lambda$ is 2-ended if and only if $\Lambda$ contains a separating subdiagram $\Lambda_0$ whose vertices generate a finite group, and $\Lambda-\Lambda_0$ consists of two vertices each of which is connected to each vertex of $\Lambda_0$ by edges labeled 2 (but not connected to each other). Equivalently, $W$ is isomorphic to $\langle x,y\rangle\times \langle H\rangle$ where $\{x,y\}\cup H=S$, $x$ and $y$ are unrelated and $\langle H\rangle$ is finite. 
\end{theorem}

For ``Dunwoody" decompositions of Coxeter groups see \S 5 of \cite{MT09}.
\subsubsection{Groups with No Free Subgroups of Rank 2}\label{noF2}
\begin{theorem} \label{free} 
Every finitely generated infinite ended group contains a free group on 2-generators. So every solvable group is either finite, 1-ended or 2-ended. Similarly for finitely generated amenable groups. Finitely generated torsion groups are finite or 1-ended.
\end{theorem}
\begin{proof}
Suppose $G$ is infinite ended and $G=A\ast_C B$, where $C$ is finite and of index at least 2 in $B$ and index at least 3 in $A$. (Recall that if $C$ has index 2 in both $A$ and $B$, then $G$ is 2-ended.) Let $a_1C\ne C$ and $a_2C\ne C$ be distinct cosets in $A$. Let $bC\ne C$ for some $b\in B$. An elementary argument shows $u=a_2ba_1^{-1}$ and $v=ba_1ba_2^{-1}b^{-1}$ generate a free group of rank 2 in $G$. 

Suppose $w$ is a reduced word of length $n$ in $u^{\pm 1}$, $v^{\pm 1}$. Let $|w|$ be the length of the corresponding element of $G$ in the $A\ast_CB$ word length structure. We show by induction that $|w|\geq 3n$. Furthermore, if the last letter of $w$ is $u$, $v$, $u^{-1}$ or $v^{-1}$ then the last syllable of $w$ in any reduced $A\ast_CB$ representation is $a_1^{-1}$, $b^{-1}$, $a_2^{-1}$ or $b^{-1}$ respectively. Certainly if $n=1$ our inductive statement is true since $u^{\pm1}$ and $v^{\pm 1}$ are $A\ast_C B$ reduced words of length 3 and 5 respectively. Assume our inductive hypothesis is true for $n-1$ ($n>1$) and let $w=w_1\cdots w_n$ be a reduced word of length $n$ in $u^{\pm 1}$, $v^{\pm 1}$. If $w_{n-1}$ is $u$ then $|w_1\cdots w_{n-1}|\geq 3(n-1)$. If $w_{n-1}=u$,  then $w$ has a reduced $A\ast_CB$ form with last syllable $a_1^{-1}$. If $w_n=u,v$ or $v^{-1}$ then $|w|\geq 3n$ and $w$ has reduced $A\ast_CB$ form with last syllable $a_1^{-1}$, $b^{-1}$ or $b^{-1}$ respectively. All other cases are completely analogous. 

For the HNN extension $B\ast_C$ with stable letter $t$, let $b\in B$ such that $bC\ne C$. The elements $btb^{-1}$ and $t$ generate a free group of rank 2. 
\end{proof}

\subsubsection{Ascending HNN-Extensions} \label{AHNN}

Suppose $B$ is a group with presentation $\langle S  \ |\ R\rangle$ and $\phi:B\to B$ is a monomorphism. Let $p:F(S)\to B$ be the map that takes $s\in S\subset F(S)$ to $s\in S\subset B$ and $\phi' :F(S)\to F(S)$ be the homomorphism such that $P \phi'=\phi P$. The ascending HNN-extension $H=B\ast_B$ (also written $B\ast_\phi$) with stable letter $t$ has presentation $\langle t, S\ |\  R, t^{-1}st=\phi'(s) \hbox{ for } s\in S\rangle$. The relations $t^{-1}st=\phi'(s) \hbox{ for } s\in S$ are called conjugation relations. 

First a few elementary observations. Using conjugation relations, it is straightforward to show elements of $H$ have a normal form:

\noindent (1) Each element of $H$ can be written as $t^abt^{-c}$ for some $a,c\geq 0$ and some $b\in B$.

\noindent (2) If $h=t^abt^{-c}$, then for $N\geq 0$, $t^{-(N+a)}(t^abt^{-c}) t^{N+c}=t^{-N}bt^N=\phi^N(b)\in B$.

Next, we consider a generalization of ascending HNN-extensions. 
\begin{theorem} [(Theorem 3.1, \cite{Mih22}] \label{MMFIE}  
Suppose $H_0$ is an infinite finitely generated group, $H_1$ is a subgroup of finite index in $H_0$, $\phi:H_1\to H_0$ is a monomorphism and $G=H_0\ast_\phi$ is the resulting HNN extension. Then $G$ is 1-ended.
\end{theorem}

\subsubsection{Touikan's Theorem on Ends of Graphs of Groups} \label{Touik} 

The main theorem of \cite{Tou15} can be used to prove many of the 1-ended splitting results of this chapter.  The group labeled $G$ in this section will always be  finitely generated. First some definitions.

A group $G$ is {\it 1-ended relative to a collection $\mathcal H$ of subgroups} if for any minimal nontrivial $G$-tree with finite edge stabilizers, there exists a subgroup $H\in \mathcal H$ that acts without a global fixed point. Otherwise $G$ is {\it many-ended relative to $\mathcal H$}. 

A $G$-equivariant map $S\to T$ of simplicial $G$-trees is called a {\it collapse} if $T$ is obtained by identifying some edge orbits of $S$ to points. In this case we also say that $S$ is obtained from $T$ by a {\it blow up}. We call the pre-image $\hat T_v\subset S$ of a vertex $v\in T$ its {\it blowup}.

Say  $H\preceq G$ to signify that $G$ splits essentially as a graph of groups with finite edge groups and $H$ is a vertex group. If a group $G$ acts on a set $A$ and $S\subset A$ then $G_S=\{g\in G : gS=S\}$.

A $G$-tree $T$ is {\it essential} if every edge of $T$ divides it into two infinite components. The tree $T$ is {\it minimal} if there are no proper subtrees $S\subset T$ with $G_S = G$. The tree $T$ is cocompact if $G/T$ is compact. An element $g$ or a subgroup $H$ of $G$ are said to {\it act  elliptically} on $T$ if the groups $\langle g\rangle$  or $H$ fix some $v\in  Vertices(T )$.

Fix a collection $\mathcal H$ of subgroups of the finitely generated group $G$. We let $T_\infty$  and $T_{ \mathcal F}$ be cocompact, minimal $G$-trees in which the subgroups in $\mathcal H$ act elliptically. The edge groups of $T_\infty$  are infinite and finitely generated and the edge groups of $T_{\mathcal F}$  are finite. Note that any nontrivial tree obtained by a collapse of $T_\infty$ has infinite edge groups whereas any collapse of $T_{\mathcal F}$ has finite edge groups. It follows that $\mathcal T_\infty$  and $T_{\mathcal F}$ have no nontrivial common collapses.

\begin{theorem} [Theorem 3.1, \cite{Tou15}] \label{Toui} Let $\mathcal H$  be a collection of subgroups of $G$ and let $T_\infty$ and $T_{\mathcal F}$ be cocompact, minimal $G$-trees in which the subgroups in $\mathcal H$ act elliptically.
Suppose furthermore that the edge groups of $T_{\mathcal F}$ are finite and that the edge groups of $T_\infty$ are infinite. Then there exists a vertex $v\in Vertices(T_\infty)$ and a nontrivial, cocompact, minimal $G_v$-tree $\hat T_v$ such that

\item (i) for every $f\in Edges(T_\infty)$ incident to $v$ the subgroups $G_f\leq G_v$ act elliptically
on $\hat T_v$, and

\item (ii) for every $H\in \mathcal H$ and $g\in G$, the subgroup $H^g \cap G_v\leq G_v$ acts elliptically on $\hat T_v$.

Moreover, either 
\begin{enumerate}

\item  every edge of $\hat T_v$ is finite or

\item  there is some edge $e\in Edges(T_\infty)$, incident to $v$, that not only satisfies condition (i), but
also satisfies the following:

\subitem (a) $G_e$ splits essentially as an amalgamated free product or an HNN extension
with finite edge group.

\subitem (b) $G_e = G_{v_e}$ for some vertex $v_e \in Vertices (\hat T_v)$.

\subitem (c) The edge stabilizers of $\hat T_v$ are conjugate in $G_v$ to the vertex group(s) of the
splitting of $G_e$ found in (a); in particular, the edge groups of $\hat T_v$ are $\prec G_e$. 

\subitem (d) The vertex groups of $\hat T_v$  that are not conjugate in $G_v$ to $G_e$ are also vertex groups of a one-edge splitting of $G_v$ with a finite edge group; in particular these vertex groups of $\hat T_v$ are $\prec G_v$.
\end{enumerate}
\end{theorem}

As an application of this theorem we prove Theorem \ref{MMFIE}.
First we need an elementary fact for group splittings.
\begin{lemma} \label{noFI}
Suppose a group $G$ splits nontrivially then no vertex group of the splitting contains a subgroup of finite index in $G$.
\end{lemma} 
\begin{proof}
If $G=A\ast_CB$ with $A\ne C\ne B$ then let $a\in A-C$ and $b\in B-C$. The element $ab$ has infinite order in $G$ and $(ab)^n\not\in A$ for any $n>0$. If $G=A\ast_\rho$ and $t$ is the stable letter, then $t^n\not\in A$ for all $n>0$. Hence $A$ cannot contain a subgroup of finite index in $G$. 
\end{proof}

\begin{proof} (of Theorem \ref{MMFIE}) Here $H_0$ is an infinite finitely generated group, $H_1$ is a subgroup of finite index in $H_0$, $\phi:H_1\to H_0$ is a monomorphism and $G=H_0\ast_\phi$ is the resulting HNN extension. We want to show that $G$ is 1-ended. 
We apply Theorem \ref{Toui} with $\mathcal H=\emptyset$. Let $T_\infty$ be the Bass-Serre tree for $H_0\ast_\phi$. The stabilizer of each vertex $v$ of $T_\infty$ is conjugate to $H_0$ and the stabilizer of each edge ``exiting" $v$ is a conjugate of $H_1$ (that has finite index in the stabilizer of $v$). Suppose that $G$ has more than one end, so that $G$ splits over a finite group and $T_\mathcal F$ is the Bass-Serre tree for that splitting. Theorem \ref{Toui} implies there is a vertex $v$ of $T_\infty$ and a non trivial, cocompact, minimal $G_v$ tree $\hat T_v$ satisfying (i). Let $f$ be an edge incident to $v$ such that $G_f$ (a conjugate of $H_1$) has finite index in $G_v$ (a conjugate of $H_0$).
Theorem \ref{Toui} (i) implies $G_f$ acts elliptically on $\hat T_v$. But then the resulting non-trivial splitting of $G_v$ has a vertex group that contains a subgroup of finite index in $G_v$, contradicting Lemma \ref{noFI}.
\end{proof}

Next we generalize the 1-ended part of Lemma 5.1 of \cite{Mih22} from finitely presented groups to finitely generated groups.

\begin{theorem} \label{SS3} 
Suppose $G=A\ast_CB$ or $G=A\ast_C$, where $A$ and $B$ are finitely generated, $C$ is infinite, finitely generated  and in the first case, of finite index in $B$. If $A$ is 1-ended, then $G$ is 1-ended.
\end{theorem}
\begin{proof}
Suppose $G$ has more than one end (in either case). Let $T_{\mathcal F}$ be the Bass-Serre tree for a nontrivial splitting of $G$ over a finite group. Let $T_\infty$ be the Bass-Serre tree for $A\ast_CB$ or $A\ast_C$.
Apply Theorem \ref{Toui} with $\mathcal H=\{A\}$. Let $v\in Vertices(T_\infty)$ and $\hat T_v$ be the vertex and tree guaranteed by Theorem \ref{Toui}. First consider the HNN case. In this case, $G_v=gAg^{-1}$ for some $g\in G$.
Since $A\in \mathcal H$,  (ii) implies that $G_v$ acts elliptically on $\hat T_v$, contradicting Lemma \ref{noFI}. 

Next assume $G=A\ast_CB$. Again,  let $v\in Vertices(T_\infty)$ and $\hat T_v$ be the vertex and tree guaranteed by Theorem \ref{Toui}. Either $G_v=gAg^{-1}$ or $G_v=gBg^{-1}$ for some $g\in G$. If $G_v=gAg^{-1}$, then as above (ii) implies  that $G_v$ acts elliptically on $\hat T_v$, contradicting that fact that $\hat T_v$ induces a non-trivial splitting of $G_v$. If $G_v=gBg^{-1}$, Then (ii) implies $gAg^{-1}\cap gBg^{-1} (=gCg^{-1})$ acts elliptically on $\hat T_v$. But $gCg^{-1}$ has finite index in $gBg^{-1}$ contradicting Lemma \ref{noFI}. 
\end{proof}

\subsubsection{Halfspaces of Splittings of Finitely Generated Groups} \label{Hspace} 

In this section, we are interested in when a ``halfspace" of a splitting of a finitely generated 1-ended group is 1-ended. Three equivalent definitions of a half space of a graph of groups splitting of a finitely generated group with finitely generated edge groups are given in \cite{MS24}. Suppose $G=A\ast_CB$ is a non-trivial splitting of a finitely generated group with $A$, $B$ and $C$ finitely generated. Let $\mathcal A$ and $\mathcal B$ be finite generating sets of $A$ and $B$ with $\mathcal C=\mathcal A\cap \mathcal B$ a generating set for $C$. The Cayley graphs $\Gamma(A,\mathcal A)$, $\Gamma (B, \mathcal B)$ and $\Gamma(\mathcal C)$ sit naturally in $\Gamma (G, \mathcal A\cup \mathcal B)$ with $\Gamma(A,\mathcal A)\cap \Gamma (B, \mathcal B)=\Gamma(\mathcal C)$. 
The graph $\Gamma( C, \mathcal C)$ separates $\Gamma(G,\mathcal A\cup \mathcal B)$ into two {\it halfspaces}, one containing $\Gamma(A,\mathcal A)$ and the other containing $\Gamma(B,\mathcal B)$. One of the main theorems of \cite{MS24} is the following:

\begin{theorem}[Theorem 1.1, \cite{MS24}]\label{Half1}
	Suppose the finitely generated 1-ended group $G$ has a non-trivial graph of groups decomposition $\mathcal G$. Suppose the edge stabilizers are finitely generated and accessible.
	Then there is a non-trivial graph of groups decomposition $\mathcal G'$ of $G$ with minimal action such that:
	\begin{enumerate}
		\item The halfspaces for $\mathcal G'$ are 1-ended.
		\item Edge stabilizers (resp. vertex stabilizers) of $\mathcal G'$ are finitely generated and are subgroups of the edge stabilizers (resp. vertex stabilizers) of $\mathcal G$.
		\item For each edge $e'$ in $\mathcal G'$ there exists an edge $e$ in $\mathcal G$ such that $G_{e'}$ is a vertex stabilizer in some finite splitting of $G_e$ over finite subgroups.
	\end{enumerate}
\end{theorem}

A corollary of this theorem is:

\begin{corollary}[Corollary 2.1, \cite{MS24}]\label{Half2}
	If a 1-ended group splits non-trivially over a 2-ended group then the halfspaces of this splitting are 1-ended.
\end{corollary}

\newpage
\section{Semistability and Simple Connectivity at $\infty$ of Groups and Spaces and $H^2(G,\mathbb ZG)$}\label{SSscH}

The topological spaces we are most interested in are connected locally finite CW complexes. Occasionally we will consider connected locally compact ANR's. (The class of ANR's contains topological manifolds and locally finite CW-complexes.)The question of whether or not all finitely presented groups are semistable at $\infty$ is over 40 years old. Finite and 2-ended groups are all semistable (in fact simply connected) at $\infty$. An infinite ended finitely presented group is semistable at $\infty$ if and only if each 1-ended group in a ``Dunwoody" decomposition of $G$ is semistable at $\infty$ \cite{M87}. Hence we concentrate mostly on 1-ended groups. Many classes of groups are known to contain only semistable at $\infty$ groups. If a 1-ended group is semistable at $\infty$ then it has a well defined fundamental group at $\infty$. The notions of  semistability and simple connectivity at $\infty$ have been extended from finitely presented groups to finitely generated groups - for the most part to assist in proving certain finitely presented groups are semistable or simply connected at $\infty$. Examples of 1-ended finitely generated groups that are NOT semistable at $\infty$ are provided in \S \ref{NonSS}. The theory of simple connectivity at $\infty$ for groups and space is older than that of semistability. The question of whether or not $H^2(G,\mathbb ZG)$ is free abelian for all finitely presented groups $G$ dates back to H. Hopf. 

Corollary \ref{GM2} connects semistability and simple connectivity at $\infty$ of a group $G$ to $H^2(G,\mathbb ZG)$: If the finitely presented group $G$ is semistable at $\infty$, then $H^2(G,\mathbb ZG)$ is free abelian and if $G$ is simply connected at $\infty$, then $H^2(G,\mathbb ZG)$ is trivial. Section \ref{SSHom} contains our homological results.

\subsection{Semistability, and Simple Connectivity at $\infty$ of a space} \label{DefSSsp}
For the most part the spaces we are concerned with are connected, locally finite CW complexes.  The semistability and simple connectivity of groups is determined in such spaces. Occasionally, we consider locally compact ANRs. Many of the basic definitions in this section make sense for more general topological spaces. Recall, proper rays $r,s:[0,\infty)\to X$ converge to the same end of $X$ if for any compact set $C$ in $X$, there is an integer $N(C)$ such that $r([N(C),\infty)$ and $s([N(C),\infty)$ belong to the same component of $X-C$. 

\begin{definition} 
The topological space $X$ {\it has semistable fundamental group at $\infty$} if any two proper rays $r,s:[0,\infty)\to X$ converging to the same end of $X$, are properly homotopic. This means there is a proper map: 

\noindent $H:[0,\infty)\times [0,1]\to X$ such that $H(t,0)=r(t)$  and $H(t,1)=s(t).$ 
\end{definition}

\begin{remark} \label{basess} 
Suppose $X$ is path connected, $v\in X$ and any two proper rays at $v$ that converge to the same end of $X$ are properly homotopic. Suppose $r$ and $s$ be proper rays that converge to the same end of $X$. Let $\alpha$ be a path from $v$ to $r(0)$ and $\beta$ a path from $v$ to $s(0)$. Then since $(\alpha, r)$ is properly homotopic to $(\beta,s)$, it is easy to see that $r$ is properly homotopic to $s$. Hence $X$ is semistable at $\infty$. 
\end{remark}

\begin{definition} 
The space $X$ is {\it simply connected at $\infty$} if for any compact set $C$ there is a compact set $D$ such that loops $\alpha$ in $X-D$ are homotopically trivial in $X-C$. This means there is a homotopy $H:[0,1]\times [0,1]\to X-C$ such that $H(t,0)=\alpha(t)$, $H(0,t)=H(1,t)=H(t,1)=H(0,0)$ for all $t\in [0,1]$.
\end{definition}

It is an elementary exercise to show:
\begin{theorem}\label{SCtoSS} 
Suppose $X$ is a locally compact, locally connected, path connected, Hausdorff space that is $\sigma$-compact. If $X$ is simply connected at $\infty$ then $X$ is semistable at $\infty$. 
\end{theorem}

For brevity we often simply say semistable at $\infty$, instead of semistable fundamental group at $\infty$.

\begin{remark} 
 The simple connectivity at $\infty$ of a locally
finite CW-complex only depends on the 2-skeleton of the complex (see for
example, Lemma 3 \cite{LR75} or Proposition 16.2.2, \cite{G}). 
\end{remark}

\begin{example} [Theorem 2.2 ,\cite{M1}] \label{prodss}  
If $A$ and $B$ are locally finite and connected infinite CW-complexes, then $A\times B$ is semistable at $\infty$.
If additionally $A$ and $B$ are simply connected and either is 1-ended, then $A\times B$ is simply connected at $\infty$.
Let $A=[0,\infty)$ (a 1-ended simply connected space that is simply connected at $\infty$) and $B$ be the (1-ended) lattice in the plane. The space $A\times B$ is not simply connected at $\infty$.
\end{example}

\begin{theorem}\label{NotScI} [Theorem 16.18.1, \cite{G} \label{NotSc1}
If $A_1$ and $A_2$ are locally finite  1-connected infinite CW-complexes and neither $A_1$ nor $A_2$ is 1-ended then $A_1\times A_2$ is not simply connected at $\infty$.
\end{theorem}

\newpage

\subsection{The Fundamental Group and Pro-Group of a Space and the Equivalence of Several Definitions}\label{prosp}
If $K$ has semistable fundamental group at $\infty$, then we can define the fundamental pro-group of $K$ and the fundamental group at $\infty$ of $K$, but first we need to develop an equivalent notion of semistability defined in terms of inverse sequences of groups.

\begin{definition}\label{DProI}
Suppose $G_1{\buildrel p_1\over \longleftarrow}G_2 {\buildrel p_2\over \longleftarrow}\cdots$ and 
$H_1{\buildrel q_1\over \longleftarrow}G_2 {\buildrel q_2\over \longleftarrow}\cdots$ are inverse sequences of groups with bonding maps homomorphisms. These sequences are {\it pro-isomorphic} if there are sequences of integers $0< m_1<m_2<\cdots$ and $0<n_1<n_2<\cdots$ and homomorphisms $f_{n_{k}}:G_{m_{k}}\to H_{n_{k}}$ and $g_{m_k}:H_{n_{k+1}}\to G_{m_k}$ making the following diagram commute for all $k$. (Horizontal arrows are compositions of bonding maps.)
\end{definition}

\vspace {-.7in}
\vbox to 2in{\vspace {-1in} \hspace {-.2in}
\includegraphics[scale=1]{FigProIso}
\vss }
\vspace{-.3in}

\begin{definition} 
An inverse sequence of groups 
$$G_1{\buildrel p_1\over \longleftarrow}G_2 {\buildrel p_2\over \longleftarrow}\cdots$$
is {\it semistable} (or {\it Mettag-Leffler}) if for each $m$ there exists $\phi(m)\geq m$ such that for all $k\geq \phi(m)$  the image of $G_{\phi(m)}$ in $G_m$ is equal to the image of $G_k$ in $G_m$. 
\end{definition}
The following result is an elementary consequence of the definitions.

\begin{theorem} [Proposition 11.3.1, \cite{G}] \label{MLSS}   An inverse sequence of groups is semistable if and only if it is pro-isomorphic to an inverse sequence of groups with epimorphisms as bonding maps.
\end{theorem}

\begin{definition} 
If $\{G_n\}$ is an inverse sequence of groups with bonding maps $f_n:G_{n+1}\to G_{n}$, then 
$\varprojlim ^1 \{G_n\}$ is the quotient set of the set $\prod_{n=1}^\infty G_n$ under the equivalence
relation generated by: $(x_n )\sim (y_n )$ if there exists $(z_n )\in G_n$ such that
$y_n = z_nx_nf_n(z_{n+1}^{-1} )$ for all $n$. We say that  $\varprojlim ^1 \{G_n\}$ is trivial if it is the set with one element $(1,1,\ldots)$.
\end{definition}

\begin{theorem}  [Theorem 11.3.2, \cite{G}] \label{lim1} 
If the inverse sequence of groups $\{G_n\}$ is semistable then  $\varprojlim ^1\{G_n\}
$  is trivial. If $\varprojlim ^1\{G_n\}$ is trivial and each $G_n$ is countable, then $\{G_n\}$ is semistable. 
\end{theorem}

\begin{theorem} \label{pi1} 
Let $K$ be a locally finite, connected CW-complex. Suppose $r$ and $s$ are proper rays in $K$ that are properly homotopic in  $K$. Suppose $\{C_i\}_{i=0}^\infty$ and $\{D_i\}_{i=0}^\infty$ are cofinal sequences of compact sets, and  $0=x_0<x_1<\cdots$ and $0=y_0<y_1<\cdots$ are sequences of real numbers such that for $i\geq 1$, $r([x_i,\infty))\subset K-C_{i-1}$ and $s([y_i,\infty))\subset K-D_{i-1}$. Then the following inverse sequences of groups are pro-isomorphic:
$$\pi_1(K-C_0,r(x_1)){\buildrel p_1\over \longleftarrow}\pi_1(K-C_1,r(x_2)) {\buildrel p_2\over \longleftarrow}\cdots$$
$$\pi_1(K-D_0,s(y_1)){\buildrel q_1\over \longleftarrow}\pi_1(K-D_1,s(y_2)) {\buildrel q_2\over \longleftarrow}\cdots$$
Here the map $p_i:\pi_1(K-C_i,r(x_{i+1})) \to \pi_1(K-C_{i-1},r(x_i))$ takes the class $[\alpha]$ to $[\tau_i \alpha \tau_i^{-1}]$ where $\tau_i$ is the restriction of $r$ to $[x_{i}, x_{i+1}]$ and similarly for $q_i$. 
\end{theorem}

\begin{proof} 
Let $H:[0,\infty)\times [0,1] \to K$ be a proper map such that $H(t,0)=r(t)$ and $H(t,1)=s(t)$. Any subsequence of an inverse sequence of groups and the inverse sequence itself are  pro-isomorphic. Combining that fact with a reparametrization of $[0,\infty)\times [0,1]$ we can assume that $x_k=y_k=k$ and:
$$H^{-1} (C_k\cup D_k)\subset [0,k)\times [0,1]$$
Let $\delta_k$ be the restriction of $s$ to $[k,k+1]$ and note that $\tau_k$ is now the restriction of $r$ to $[k,k+1]$. Let $\gamma_k:[0,1]\to K$ by $\gamma_k(t)=H(k,t)$ (see Figure \ref{Fpi1}).

\begin{figure}
\vbox to 3in{\vspace {-2in} \hspace {-.5in}
\hspace{-1 in}
\includegraphics[scale=1]{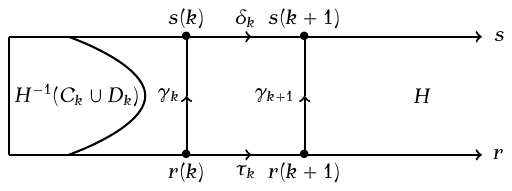}
\vss }
\vspace{-2.2in}
\caption{$r$ is properly homotopic to $s$} 
\label{Fpi1}
\end{figure}

\medskip

\medskip

\noindent ($\ast$) Note that $H$ restricted to $[k,k+1]\times [0,1]$ is a homotopy killing the loop $(\delta_{k}^{-1},\gamma_k^{-1},\tau_{k},\gamma_{k+1})$ in $K-(C_k\cup D_k)$ so that $\delta_k$ is homotopic rel$\{0,1\}$ to $(\gamma_k^{-1},\tau_{k},\gamma_{k+1})$ in $K-(C_k\cup D_k)$.

\medskip

\medskip


\noindent Let $f_k:\pi_1(K-C_k, r(k))\to \pi_1(K-D_k, s(k))$ and 

\noindent $g_k:\pi_1(K-D_{k+1},s(k+1)) \to \pi_1(K-C_k,r(k))$ be defined by: 
$$f_k([\alpha])=[(\gamma_k^{-1}, \alpha, \gamma_k)]\hbox{ and } g_k([\beta])=[(\tau_{k}, \gamma_{k+1}, \beta, \gamma_{k+1}^{-1},\tau_{k}^{-1})]$$

According to Definition \ref{DProI}, it is enough to show that the following diagram commutes for all $k$:

\vspace {-.5in}
\vbox to 2in{\vspace {-1in} \hspace {-1in}
\includegraphics[scale=1]{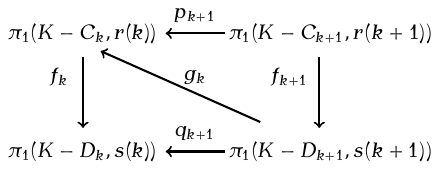}
\vss }
\vspace{-.1in}

Let $\alpha$ be a loop at $r(k+1)$ and $\beta$ be a loop at $s(k+1)$ then:
$$g_k(f_{k+1}([\alpha]))=g_k[(\gamma_{k+1}^{-1},\alpha, \gamma_{k+1})]=(\tau_{k},\alpha,\tau_{k}^{-1})=p_{k+1}([\alpha])$$

$$f_k(g_k([\beta]))=f_k([(\tau_{k},\gamma_{k+1},\beta,\gamma_{k+1}^{-1},\tau_{k}^{-1})] =$$
$$[(\gamma_k^{-1},\tau_{k},\gamma_{k+1},\beta,\gamma_{k+1}^{-1},\tau_{k}^{-1},\gamma_k]\in \pi_1(K-D_k,s(k))$$
By ($\ast$), $\delta_{k}$ is homotopic rel$\{0,1\}$ to $(\gamma_k^{-1}, \tau_{k},\gamma_{k+1})$ in $K-D_k$. Combining:
$$f_k(g_k([\beta]))=[(\delta_{k}, \beta, \delta_{k}^{-1})]=q_{k+1}([\beta])$$
\end{proof}

Theorem \ref{pi1} implies the following definition is independent (up to pro-isomorphism) of choice of cofinal sequence of compact sets.

\begin{definition}
Suppose $K$ is a locally finite, connected CW-complex, and $r$ is a proper ray in $X$. Let $\{C_i\}_{i=0}^\infty$ be a cofinal sequence of compact sets in $K$ and $0=x_0<x_1<\cdots$ a sequence of real numbers such that for $i\geq 1$, $r([x_i,\infty))\subset K-C_{i-1}$. Then the {\it fundamental pro-group of $K$ based at $r$} (denoted $\pi_1(\varepsilon K,r)$) is (up to pro-isomorphism) the inverse sequence of groups:
$$\pi_1(K-C_0,r(x_1)){\buildrel p_1\over \longleftarrow}\pi_1(K-C_1,r(x_2)) {\buildrel p_2\over \longleftarrow}\cdots$$ 
Here the map $p_i:\pi_1(K-C_i,r(x_{i+1})) \to \pi_1(K-C_{i-1},r(x_i))$ takes the class $[\alpha]$ to $[\tau_i \alpha \tau_i^{-1}]$ where $\tau_i$ is the restriction of $r$ to $[x_{i}, x_{i+1}]$. The inverse limit of this inverse sequence is called the {\it fundamental group at $\infty$ of $K$ based at $r$} and is denoted $\pi_1^e(K,r)$. 
\end{definition}
The next result is a change of base ray result that follows directly from Theorem \ref{pi1}. 
\begin{theorem} \label{pro1} 
 Suppose $K$ is a locally finite, connected, CW-complex. If  $r$ and $s$ are proper rays that are properly homotopic in $K$, then $\pi_1(\varepsilon K,r)$ is pro-isomorphic to $\pi_1(\varepsilon K,s)$. 
 \end{theorem} 
Theorem \ref{pro1} implies the following definition is independent (up to pro-isomorphism) of choice of base ray.

\begin{definition} 
Suppose $K$ is a locally finite CW-complex that is connected, 1-ended and semistable at $\infty$. Then {\it the fundamental pro-group of $K$} (denoted $\pi_1(\varepsilon K)$)  is $\pi_1(\varepsilon K,r)$, for any proper ray $r$ in $K$. The {\it fundamental group at $\infty$} of $K$ is $\pi_1^e(K,r)$ (the inverse limit of $\pi_1(\varepsilon K,r)$) and is denoted $\pi_1^e(K)$. 
\end{definition}

\begin{theorem} \label{SSloop}
Let $K$ be a connected locally finite CW-complex. Suppose $r$ is a proper ray in $K$ and $\pi_1(\varepsilon K, r)$ is semistable.  Then for any compact set $C\subset K$ there is a compact set $D(C)\subset K$ such that for any 
loop  $\alpha$ based at $r(x)$ and with image in $K-D$, the proper rays $r|_{[x,\infty)}$ and $(\alpha, r|_{[x,\infty)})$ are properly homotopic rel$\{r(x)\}$ 
in $K-C$. 
\end{theorem}

\noindent Geometrically this means if $\pi_1(\varepsilon K,r)$ is semistable, then loops in $K-D$ that are based on $r$  can be properly ``slid off to infinity" along $r$ by a homotopy in  $K-C$.

\begin{proof} 
Certainly if the theorem is true for some $x$ with $r(x)\in K-D$ then it is true for all $x$ such that $r(x)\in K-D$. 
Let $\{C_i\}_{i=0}^\infty$ be a cofinal sequence of compact subsets of $K$ and $x_0<x_1<\cdots$ a sequence of integers such that $r([x_i,\infty))\subset K-C_i$ for all $i$. If $\tau:([0,1],\{0,1\})\to (K-C_n, r(x_n))$, then let $[\tau]$ be the corresponding element of $\pi_1(K-C_n,r(x_n))$. Without loss, assume that $C\subset C_0$. We are assuming that
$$\pi_1(K-C_0,r(x_0)){\buildrel p_1\over \longleftarrow}\pi_1(K-C_1,r(x_1)) {\buildrel p_2\over \longleftarrow}\cdots$$
is semistable.  So there is an integer $N$ such that for all $n\geq N$ the image of $\pi_1(K-C_n, r(x_n))$ in $\pi_1(K-C_0, r(x_0))$ is equal to the image of $\pi_1(K-C_N,r(x_N))$ in $\pi_1(K-C_0, r(x_0))$. Then for any loop $\alpha$ based at $r(x_N)$ and with image in $K-C_N$, and any integer $n>N$, there is a loop $\beta_n$ based at $r(x_n)$ such that the image of the elements  $[\alpha]\in \pi_1(K-C_N,r(x_n))$ and $[\beta_n]\in \pi_1(K-C_n, r(x_n))$ are the same in $\pi_1(K-C_0,r(x_0))$. This means that $\alpha$ and $\beta_n$ are homotopic rel$\{r\}$ by a homotopy in $K-C$. Then any loop based at $r(x_N)$, with image in $X-C_N$ is homotopic rel$\{r\}$ and by a homotopy in $K-C_0$ to a loop in $K-C_n$ for any $n>N$.  

\begin{figure}
\vbox to 3in{\vspace {-2in} \hspace {-.5in}
\hspace{-1 in}
\includegraphics[scale=1]{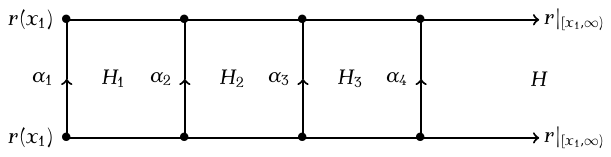}
\vss }
\vspace{-2.4in}
\caption{Stacking Homotopies} 
\label{FigSSloop}
\end{figure}

Without loss assume $C_N=C_1$. Inductively we may assume that if $n>0$ and $\alpha$ is a loop in $K-C_n$ based at $r(x_n)$ then $\alpha$ is homotopic rel$\{r\}$ to a loop in $K-C_{n+1}$ based at $r(x_{n+1})$ by a homotopy in $K-C_{n-1}$. Let $\alpha_1$ be a loop in $K-C_1$ based at $r(x_1)$. Then $\alpha_1$ is homotopic rel$\{r\}$ to a loop $\alpha_2$ with image in $K-C_2$ and based at $r(x_2)$ by a homotopy $H_1$ in $K-C_0$. Inductively $\alpha_n$ is a loop in $K-C_n$ based at $r(x_n)$ which is homotopic rel$\{r\}$ to a loop $\alpha_{n+1}$ with image in $K-C_{n+1}$ and based at $r(x_{n+1})$ by a homotopy $H_n$ in $K-C_{n-1}$. Stacking the homotopies $H_1, H_2,\ldots$ as in Figure \ref{FigSSloop} produces the desired proper homotopy of $r|_{[x_1,\infty)}$ to $(\alpha_1, r|_{[x_1,\infty)})$ rel$\{r(x_1)\}$
 in $K-C_0$. Here $C_1$ plays the role of $D$ in the statement of our theorem. 
\end{proof}

(Theorem 2.1, \cite{M1}), and Lemma 9 of \cite{M86},  provide several equivalent notions of semistability. A slight modification of proofs give the following result.

\begin{theorem}\label{ssequiv} 
Suppose $K$ is a locally finite, connected CW-complex. Then the following are equivalent:
\begin{enumerate}
\item $K$ has  semistable fundamental group at $\infty$.
\item Let $\mathcal E$ be an end of $K$ and $r$ some 
(equivalently any) proper ray in $\mathcal E$, then the fundamental pro-group of $K$ based at $r$ ($\pi_1(\varepsilon K,r)$) is semistable (equivalently $\pi_1(\varepsilon K,r)$ is pro-isomorphic to an inverse sequence of groups with epimorphic bonding maps - see Theorem \ref{MLSS}). 
\item Let $\mathcal E$ be an end of $K$, $r$ some 
(equivalently any) proper ray in $\mathcal E$, and $C$ a compact subset of $K$, then there is a compact set $D$ in $K$ such that for any third compact set $E$ in $K$ and loop $\alpha$ based on $r$ and with image in $K-D$, $\alpha$ is homotopic rel$\{r\}$ to a loop in $K-E$, by a homotopy with image in $K-C$. 
\item For any compact subset $C$ of $K$ there is a compact subset $D$ of $K$ such that if $r$ and $s$ are proper rays based at $v$, converging to the same end of $K$ and with image in $K-D$, then $r$ and $s$ are properly homotopic rel$\{v\}$, by a proper homotopy in $K-C$. 
\item If $C$ is compact in $K$ then there is a compact set $D$ in $K$ such that for
any third compact set $E$ in $K$ and proper rays $r$ and $s$ based at a vertex $v$, converging to the same end of $K$ and with
image in $K-D$, then there is a path $\alpha$ in $K-E$ connecting points of $r$ and
$s$ such that the loop determined by $\alpha$ and the initial segments of $r$
and $s$ is homotopically trivial in $K-C$.

\medskip

\hbox{If in addition $K$ is simply connected, then:}

\item If $r$ and $s$ are proper rays based at $v$ and converging to the same end of $K$, then
$r$ and $s$ are properly homotopic $rel\{v\}$. (Also see Lemma \ref{FGSS6}.)
\end{enumerate}
\end{theorem}


\begin{proof} ({\it 2} implies {\it 4}) 
Suppose  $r$ is a proper ray in the end $\mathcal E$ of $K$ and the inverse sequence of groups:
 $$(\ast) \ \ \ \ \ \ \ \pi_1(K-C_0,r(x_0)){\buildrel p_1\over \longleftarrow}\pi_1(K-C_1,r(x_1)) {\buildrel p_2\over \longleftarrow}\cdots$$ 
 is semistable. Here $r([x_i,\infty))\subset K-C_{i}$. Theorem \ref{SSloop} implies that  
for  any  compact set  $C\subset K$ there is a compact set $D(C)\subset K$, such that  for any 
loop, $\alpha$,  in $K - D$ (with $\alpha$  based at  $r(x)$),  the proper rays $r|_{[x,\infty)}$ and $(\alpha,r|_{[x,\infty)})$ are properly homotopic rel$\{r(x)\}$ 
by  a homotopy whose image lies in $K - C$. 
Without loss  assume that  $D(C_i)\subset C_{i+1}$ for all $i$. So that: 

\medskip

\noindent ($\ast$) If $\alpha$ is a loop in  $K - C_i$, based at  $r(x)$,  then $r|_{[x,\infty)}$ is properly homotopic rel$\{r(x)\}$ to $(\alpha, r|_{[x,\infty)})$, by a homotopy in $K- C_{i-1}$. 

\medskip

Theorem \ref{ES} allows us to assume that $K-C_i$ is a union of unbounded path components. Let $C$ be any compact subset of $X$. We assume $C\subset C_0$ and show any two proper rays in $K-C_1$ that are based at a vertex $v$ and converge to $\mathcal E$, are properly homotopic rel$\{v\}$ in $K-C_0$. First we let $v=r(x_1)$.  Let  $s:  ([0, \infty),\{0\}) \to (K-C_1, \{r(x_1)\})$ be  a  proper ray in $\mathcal E$. 
  
We show $s$ and  $r|_{[x_1,\infty)}$ are properly homotopic rel$\{r(x_1)\}$ in $K-C_0$. Since $s$ is arbitrary, this implies any two $\mathcal E$ proper rays based at $r(x_1)$ and with image in $K-C_1$ are properly homotopic rel$\{r(x_1)\}$ in $K-C_0$. For $i\geq 1$, let $\beta_i=r|_{[x_i,x_{i+1}]}$. Let $a_1=0$ (so that $r(x_1)=s(0)=s(a_1)$) and for $i>1$ choose $a_i\in [0, \infty)$ such that  $s([a_i, \infty)) \subset  K-C_i$. Without loss assume that  $a_i <  a_{i+1}$ for  all $i$. Let $\alpha_i=s|_{[a_{i},a_{i+1}]}$.   Let  $\gamma_1$ be the constant path at $r(x_1)$. For $i>1$, let
$\gamma_i: [0,1] \to  K-C_i$, such that  $\gamma_i(0) = s(a_i)$ and  $\gamma_i(1) = r(x_i)$ (See Figure \ref{Fig231a}).   
  
\begin{figure}
\vbox to 3in{\vspace {-2in} \hspace {.5in}
\hspace{-1 in}
\includegraphics[scale=1]{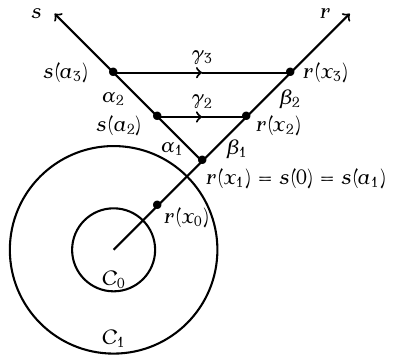}
\vss }
\vspace{-1in}
\caption{The case $v=r(x_1)$} 
\label{Fig231a}
\end{figure}

\medskip

For $i\geq 1$, let $\delta_i$  be the loop $(\gamma_{i}^{-1},\alpha_i,\gamma_{i+1},\beta_i^{-1})$. It is easy to  show that  $s$, which we represent by $(\alpha_1,\alpha_2,\ldots)$, is properly homotopic to the map of $[0, \infty) \to K$ represented by  $(\delta_1,\beta_1,\delta_2,\beta_2,\ldots)$ (simply eliminate backtracking). For $i\geq 1$, ($\ast$) implies that $r|_{[x_i,\infty)}$ is properly homotopic to $(\delta_i,r|_{[x_i,\infty)})$ rel$\{r(x_i)\}$ in $K-C_{i-1}$. 

This means there is a proper map
$H_i:[0,1]\times [0,\infty)\to K-C_{i-1}$  such that on $\{0\}\times [0,\infty)$ and on $\{1\}\times [0,\infty)$, $H_i$ is $(\beta_i, \beta_{i+1}, \ldots)=r|_{[x_i,\infty)}$ and on $[0,1]\times \{0\}$, $H_i$ is $\delta_i$. Combining the $H_i$ as in Figure \ref{Fig231b}, produces a homotopy rel$\{r(x_1)\}$ of $r|_{[x_1,\infty)}=(\beta_1,\beta_2,\ldots )$ to $(\delta_1,\beta_1,\delta_2,\beta_2, \ldots)$ in $K-C_0$.

\begin{figure}
\vbox to 3in{\vspace {-2in} \hspace {.5in}
\hspace{-1 in}
\includegraphics[scale=1]{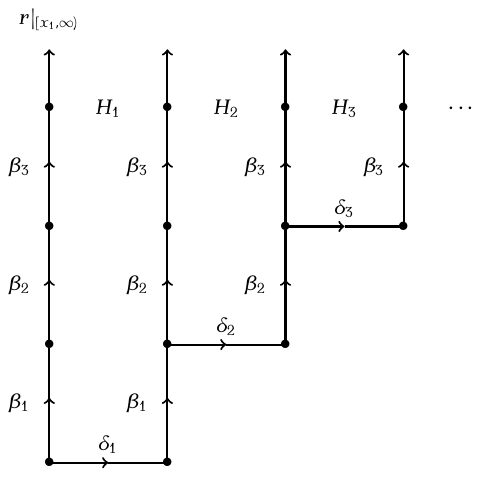}
\vss }
\vspace{-.1in}
\caption{Combining the $H_i$} 
\label{Fig231b}
\end{figure}

\medskip

Next we show this homotopy (call it $H$) is proper. Suppose $A$ is compact in $K$. Choose $n$ such that $A\subset C_n$. Then for $i\geq n+1$, the image of $H_i$ avoids $C_n$ (and hence avoids $A$). So $H^{-1}(A)=\cup_{i=1}^nH_i^{-1}(A)$ (a finite union of compact sets). So $H$ is proper. 

We have $r|_{[x_1,\infty)}$ is properly homotopic rel$\{r(x_1)\}$ to $(\delta_1, \beta_1,\delta_2,\beta_2,\ldots)$ in $K-C_0$ which is properly homotopic rel$\{r(x_1)\}$ to $s$ in $K-C_0$. As $s$ is an arbitrary $\mathcal E$ proper ray in $K-C_1$ we have any two $\mathcal E$ proper rays based at $r(x_1)$ and with image in $K-C_1$ are properly homotopic in $K-C_0$. 

Next suppose $s_1$ and $s_2$ are $\mathcal E$ proper rays based at the vertex $v$ and with image in $K-C_1$. Let $\tau$ be a path in $X-C_1$ from $r(x_1)$ to $v$. Then the proper rays $(\tau, s_1)$ and $(\tau,s_2)$ are properly homotopic rel$\{x_1)\}$ in $X-C_0$. This implies $s_1$ is properly homotopic rel$\{v\}$ to $s_2$ by a homotopy in $K-C_0$. 
\end{proof} 

\begin{proof} ({\it 4} implies {\it 1})
Let $C=\emptyset$. Then {\it 4} implies there is a compact set $D$ such that any two proper rays converging to the same end in $K-D$ are properly homotopic. Say $r$ and $s$ are proper rays converging to the same end of $K$. Choose an integer $n$ such that $r([n,\infty)$ and $s([n,\infty))$ have image in $K-D$. Let $\tau$ be a path in $K-D$ from $r(n)$ to $s(n)$. Then there is a proper homotopy $H$ of $r|_{[n,\infty)}$ to $(\tau, s|_{[n,\infty)})$ rel$\{r(n)\}$. This gives a proper homotopy $H':[0,\infty)\times [0,1]\to K$ such that $H'(t,0)=r(t)$, $H'(t,1)=s(t)$ and $H'|_{\{0\}\times [0,1]} =(r|_{[0,n]}, \tau, s|_{[0,n]}^{-1})$
\end{proof}

\begin{proof} ({\it 1} implies {\it 2})
 Let  $G_1  {\buildrel \phi_1 \over \leftarrow} G_2{\buildrel \phi_2\over \leftarrow}\cdots$ be  an  inverse sequence of  
groups. Recall $\varprojlim ^1 \{G_n\}$ is  the  pointed set  of  equivalent classes under the  equivalence relation on $\Pi_{n>0}G_n$  defined by  $\langle x_n\rangle \sim \langle y_n\rangle$ if  there is  $\langle g_n\rangle$ such that  $\langle y_n\rangle= \langle g_nx_n\phi_n(g_{n+1}^{-1})\rangle$. Theorem \ref{lim1} implies that if each $G_n$ is countable, then $\varprojlim^1\{G_n\}$ is trivial if and only if $\{G_n\}$ is semistable. 

\begin{figure}
\vbox to 3in{\vspace {-2in} \hspace {.5in}
\hspace{-1 in}
\includegraphics[scale=1]{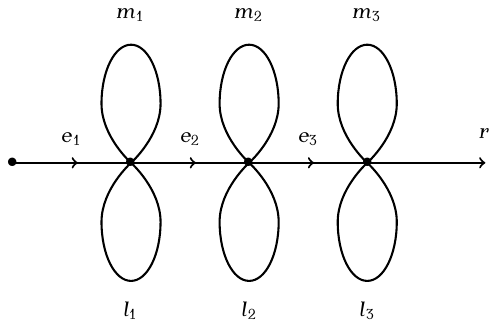}
\vss }
\vspace{-1in}
\caption{Equivalent elements} 
\label{Fig231c}
\vspace{.3in}
\end{figure}

Let $\{C_i\}_{i=1}^\infty$ be a cofinal sequence of compact sets in $K$. Our goal is to show that if any two proper maps $[0, \infty) \to  K$ converging to the same end of $K$  are  properly homotopic, then $\varprojlim \pi_1(\varepsilon K, r)$ is  trivial for $r:[0,\infty)\to K$ a proper map. For simplicity assume we have reparametrized $r$ so that $r([n,\infty))\subset K-C_n$ for all $n\geq 1$.  
Represent $r$ by $(e_1,  e_2,\ldots)$  where  $e_n(t) =  r(n-1  +  t)$ for $t\in [0,1]$.   Assume that  $r(n)$ is the  base point of $\pi_1(K - C_n, r(n))$ for $n\geq 1$. 
Choose $\langle [m_n]\rangle$ and  $\langle [l_n]\rangle$  in  $\Pi_{n>0}\pi_1(K-C_n,r(n))$ (see  Figure \ref{Fig231c}).  We  show  $\langle [m_n]\rangle \sim\langle [l_n]\rangle$. 

Since all  proper maps $[0, \infty)\to K$ converging to the same end are  properly homotopic there is a  proper homotopy $F:  [0,\infty) \times [0,1] \to K$  of  $(e_1,  m_1,  e_2,m_2,\ldots)$ to  $(e_1, l_1, e_2, l_2,\ldots)$. Choose $n(i)$  so that  $F([2(n(i)), \infty)\times [0,1])\subset K- C_i$. (Note that the length of $(e_1,m_1,\ldots, e_{n(i)},m_{n(i)})$ is $2(n(i))$.)  Let $g_i:  [0,1] \to K$  be  defined by $g_i(t) =  F(2n(i),t)$ and let:
$$h_i =(m_i,e_{i+1},m_{i+1},e_{i+2},\cdots ,m_{n(i)},g_i,l_{n(i)}^{-1},\cdots ,e_{i+2}^{-1},l_{i+1}^{-1},e_{i+1}^{-1}, l_i^{-1}).$$

See  Figure \ref{Fig231d}  representing $[0,  \infty) \times [0,1]$. By definition, the  bonding maps $\phi_i$ of $\pi_1(\varepsilon K,r)$ are  such that  $\phi_i([h_{i+1}])=[(e_{i+1},h_{i+1},e_{i+1}^{-1})]$.  It remains to show that  $m_i$, is homotopic rel$\{0,1\}$ to $(h_i,l_i,\phi_i(h_{i+1}))^{-1})$  in $K-C_i$.  We  show $m_1$ is homotopic to $(h_1, l_1, \phi_1(h_2)^{-1})$ rel$\{0,1\}$ in  $K -  C_1$.  The general case is completely analogous.

\begin{figure}
\vbox to 3in{\vspace {-2in} \hspace {-1in}
\hspace{-1 in}
\includegraphics[scale=1]{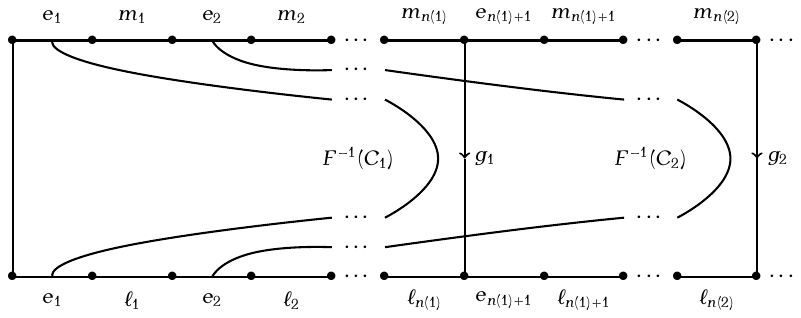}
\vss }
\vspace{-1.2in}
\caption{Avoiding compact sets} 
\label{Fig231d}
\end{figure}

$$(h_1,l_1, \phi_1(h_2)^{-1})=(m_1, e_2, m_2,\cdots, m_{n(1)}, g_1, l_{n(1)}^{-1},\cdots, l_2^{-1},e_2^{-1},l_1^{-1},$$
$$l_1,e_2, l_2, e_3, l_3, \cdots l_{n(2)}, g_2^{-1}, m_{n(2)}^{-1},\cdots, m_3^{-1},e_3^{-1},m_2^{-1},e_2^{-1})$$

Eliminate edges followed by  their inverses, and  the  loop:
$$(g_1,e_{n(1)+1}, l_{n(1)+1},\cdots, l_{n(2)}, g_2^{-1}, m_{n(2)}^{-1},\cdots, m_{n(1)+1}^{-1},e_{n(1)+1}^{-1})$$
 which is homotopically trivial in  $K-C_1$   (see Figure \ref{Fig231d}).  What remains is $m_1$, and  Figure \ref{Fig231d} shows the induced homotopy takes place in $K-C_1$
\end{proof}

It is straightforward to see that
{\it 2} is equivalent to {\it 3} and {\it 4} is equivalent to {\it 5}.

\begin{proof} (If $K$ is simply connected, then {\it 1} is equivalent to {\it 6}.) 
Certainly {\it 6} implies {\it 1}. 

Suppose {\it 1}. Let $r$ and $s$ be proper rays at $v$ converging to the same end of $K$. By hypothesis, there is a proper homotopy $H:[0,\infty)\times [0,1]\to K$ such that $H(t,0)=r(t)$, $H(t,1)=s(t)$. Since $K$ is simply connected the loop $H|_{\{0\}\times [0,1]}$ is homotopically trivial by a homotopy $H'$. The homotopies  $H$ and $H'$ can be combined to show $r$ is properly homotopic to $s$ rel$\{v\}$. 
\end{proof}

\begin{remark} 
 The semistability of the fundamental group at $\infty$ of a locally
finite CW-complex $K$ only depends on the 2-skeleton of the complex (see for
example, (Lemma 3, \cite{LR75}) or (Proposition 16.2.2, \cite{G}). If $r$ is a proper ray in $K$, then $r$ is properly homotopic to a proper edge path ray in $K$. 
\end{remark}

\begin{example}
Note that the 1-ended CW-complex obtained by attaching a loop at $0$ to the interval $[0,\infty)$ is  semistable at $\infty$. Consider a proper ray $r$ which maps $[0,\infty)$ homeomorphically to $[0,\infty) $ and a proper ray $s$ which maps $[0,1]$ once around the loop and then maps  $[1,\infty)$ homeomorphically to $[0,\infty)$. Clearly $r$ and $s$ are properly homotopic, but not by a proper homotopy rel$\{0\}$. 
\end{example}

\newpage

\subsection{Proper $n$-Equivalence, Semistability, Simple Connectivity, Fundamental Pro-Group and Fundamental Group at $\infty$ - for a Group}\label{SSproGp}

We begin this section by defining what it means for finitely presented groups to be proper $n$-equivalent and then connect this idea to quasi-isometries. 

\subsubsection{The Definitions} \label{Pr2eq} 

Let $D^k$ be the {\it $k$-disk} in $\mathbb R^{k}$ (all points of distance $\leq 1$ from the origin) and $S^{k-1}$ the {\it $(k-1)$-sphere} in $\mathbb R^k$ (all points of distance $1$ from the origin). Let $n$ be an integer $\geq 1$. A path connected space $X$ is {\it $n$-connected}  ({\it $n$-aspherical}) if for every $1\leq k\leq n$ (respectively $2\leq k\leq n$), every map $S^k\to X$ extends to a map of $D^k\to X$. A CW-complex $X$ is {\it $n$-connected} ($n$-aspherical) if and only if $X^{n+1}$ (its $(n+1)^{st}$ skeleton) is $n$-connected ($n$-aspherical). A space $X$ is {\it aspherical} if $X$ is $n$-aspherical for all $n$. 
\begin{proposition} [Proposition 7.1.3,\cite{G}] 
A path connected CW complex $X$ is $n$-aspherical if and only if its universal cover $\tilde X$ is $n$-connected. The complex $X$ is aspherical if and only if $\tilde X$ is contractible. 
\end{proposition}
Let $G$ be a group. A {\it $K(G,1)$-complex} is a path connected aspherical CW-complex $X$ with $\pi_1(X)$ isomorphic to $G$. The space $X$ is also called an {\it Eilengerg-MacLane complex of type $(G,1)$} or a {\it classifying space for $G$}.  A group $G$ has type $F_n$ if there exists a $K(G,1)$-complex having finite $n$-skeleton. A group has type $FP_n$ (over $\mathbb Z$) if there is a projective $\mathbb ZG$ resolution over $\mathbb Z$ which is finitely generated in dimensions $\leq n$. When $G$ is finitely presented the notions of $F_n$ and $FP_n$ are the same (see [Section8.2,\cite{G}] for more on this subject). 

Definition \ref{n-equiv} describes a proper $n$-equivalence. 

\begin{theorem} [See the first paragraph of \S 16.5,\cite{G}] \label{P1E} 
Suppose $X$ and $Y$ are finite connected CW-complexes with $\pi_1(X)$ isomorphic to $\pi_1(Y)$. Then there is a 2-equivalence between $X$ and $Y$, and hence (by lifting) a proper 2-equivalence between the 2-skeletons  of the universal covers of $X$ and $Y$.
\end{theorem}

More generally: 

\begin{theorem} [Theorem 18.2.11, \cite{G})] \label{pnequiv-qi} 
Let $G$ and $H$ be finitely generated groups of type $F_n$. Let $X$ be a $K(G,1)$ with finite $n$-skeleton, and let $Y$ be a $K(H,1)$ with finite $n$-skeleton. If $G$ and $H$ are quasi-isometric then there is a proper $n$-equivalence between the universal covers $\tilde X^n \to \tilde Y^ n$.
\end{theorem}

\begin{definition}
The $FP_n$ groups $G$ and $H$ are {\it proper $n$-equivalent} if for some (any) $K(G,1)$with finite $n$-skeleton and some (any) $K(H,1)$ with finite $n$-skeleton, the universal covers $\tilde X^n$ and $\tilde Y^n$ are proper $n$-equivalent. 
\end{definition}
Theorem \ref{pnequiv-qi} implies:
\begin{corollary} \label{qi-pnequiv} 
If $G$ and $H$ are quasi-isometric $FP_n$ groups, then they are proper $n$-equivalent.
\end{corollary}

\begin{theorem} [(Proposition 16.2.3, \cite{G}] \label{2-equiv}
If $X$ and $Y$ are locally finite, connected CW-complexes and there is a proper 2-equivalence $f:X\to Y$, then 
\begin{enumerate} 
\item $X$ has semistable fundamental group if and only if $Y$ has semistable fundamental group at $\infty$;
\item  $X$ is simply connected at $\infty$  if and only if $Y$ is simply connected at $\infty$;
\item the fundamental pro-group $\pi_1(\varepsilon X,r)$ is pro-isomorphic to $\pi_1(\varepsilon Y,f(r))$ for any proper ray $r$ in $X$;
\item the fundamental group at $\infty$,  $\pi_1^e(X, r)$ is isomorphic (as a topological group) to  $\pi_1^e(Y,f(r))$ for any proper ray $r$ in $X$. 
\end{enumerate}
\end{theorem}

\begin{corollary}\label{pro2}
Suppose $X$ and $Y$ are finite connected CW-complexes with respective universal covers $\tilde X$ and $\tilde Y$ and $\pi_1(X)$ is isomorphic to $\pi_1(Y)$. Then  
\begin{enumerate} 
\item $\tilde X$ has semistable fundamental group if and only if $\tilde Y$ has semistable fundamental group at $\infty$;
\item  $\tilde X$ is simply connected at $\infty$  if and only if $\tilde Y$ is simply connected at $\infty$;
\item the fundamental pro-group $\pi_1(\varepsilon \tilde X,r)$ is pro-isomorphic to $\pi_1(\varepsilon \tilde Y,f(r))$ for any  proper ray $r$ in $X$;
\item the fundamental group at $\infty$, $\pi_1^e(\tilde X,r)$ is isomorphic (as a topological group) to $\pi_1^e(\tilde Y,f(r))$ for any proper ray $r$ in $X$. 
\end{enumerate}
\end{corollary}

\begin{definition} \label{defss2}
If $G$ is a 1-ended finitely presented group and $X$ is some (equivalently any by Corollary \ref{pro2}) finite CW-complex with fundamental group $G$ and universal cover $\tilde X$, then we say: 
\begin{enumerate} 
\item $G$ is {\it semistable at $\infty$} if  $\tilde X$ has semistable fundamental group at $\infty$;
\item  $G$ is {\it simply connected at $\infty$}  if $\tilde X$ is simply connected at $\infty$;
\item if $G$ is semistable at $\infty$, then the fundamental pro-group $\pi_1(\varepsilon G)$ is (up to  pro-isomorphism) the fundamental pro-group $\pi_1(\varepsilon \tilde X,r)$ for any proper ray $r$;
\item if $G$ is semistable at $\infty$, then the fundamental group at $\infty$ of $G$,  denoted $\pi_1^e(G)$ is (up to isomorphism as a topological group) the fundamental group at $\infty$, $\pi_1^e(\tilde X,r)$ for any proper ray $r$. 
\end{enumerate}
\end{definition}

\begin{theorem} If the finitely presented group $G$ is simply connected at $\infty$, then $G$ is semistable at $\infty$, the fundamental pro-group  $\pi_1(\varepsilon G)$ is pro-trivial and the fundamental group at $\infty$,  $\pi_1^e(G)$ is the trivial group.
\end{theorem}

\begin{theorem}\label{QIsp}   Semistability at $\infty$, and simple connectivity at $\infty$, are quasi-isometry invariants of finitely presented groups (\cite{G} and \cite{BR93}). 
\end{theorem}

We wonder if this result can be extend this to say: {\it Semistability at $\infty$, simple connectivity at $\infty$, the fundamental group at $\infty$  and the fundamental pro-group are all quasi-isometry and proper 2-equivalence invariants of simply connected locally finite CW-complexes.}

Combining results we have:

\begin{theorem}\label{ssqi} 
Semistability at $\infty$, and simple connectivity at $\infty$, 
are quasi-isometry and proper 2-equivalence invariants for finitely presented groups.
\end{theorem}

\begin{theorem}\label{Fin2E}
All finite and 2-ended groups are semistable (in fact, simply connected) at $\infty$. 
\end{theorem}
\begin{proof}
The result is trivial for finite groups. If $G$ is 2-ended, then it has an infinite cyclic subgroup $J$ of finite index. As the real line covers the circle, $J$ is simply connected at $\infty$. Theorem \ref{QIsp} implies $G$ is simply connected at $\infty$. 
\end{proof}

Proper 2-equivalence is more of a topological invariant than a geometric invariant.  As a simple example $\mathbb Z\oplus \mathbb Z$ is not quasi-isometric to a hyperbolic surface group, but all (non-sphere) surfaces have universal covers homeomorphic to $\mathbb R^2$ and so their fundamental groups are proper 2-equivalent. Gromov's idea to classify finitely presented groups up to quasi-isometry fits well with idea of classifying finitely presented groups up to proper 2-equivalence. If one could classify the finitely presented groups in a class of groups up proper 2-equivalence, then one could attempt to classify the groups in a proper 2-equivalence class up to quasi-isometry. This reduction idea is supported by the following result.

\begin{theorem} [(Proposition 2.9, \cite{CLQR20}] \label{fgi-2eq}
Finitely presented 1-ended groups that are semistable at $\infty$ are proper 2-equivalent if and only if their fundamental groups at $\infty$ are pro-isomorphic. 
\end{theorem}

If $1\to H\to G\to K\to 1$ is a short exact sequence of infinite finitely presented groups, then (see Theorem \ref{P2E}) $G$ is proper 2-equivalent to either $\mathbb Z\times \mathbb Z\times \mathbb Z$, $\mathbb Z\times \mathbb Z$ or $\mathbb F_2\times \mathbb Z$ (a rather limited collection of possibilities). But if $G$ is a recursively presented group then there is a right angled Artin group $A$ whose fundamental group at $\infty$ contains a subgroup isomorphic to $C$ (see Theorem \ref{RecRA}). 

\subsubsection{Pro-mono and Coaxial Actions}\label{Coax}
In this section the space $Y$ is a 1-ended, simply connected, locally compact, absolute neighborhood retract (ANR). (The class of ANR's includes topological manifolds and locally finite CW-complexes.) Certain types of group actions on $Y$ can insure that $Y$ has fundamental pro-group that is both semistable and pro-free. 

\begin{definition} 
Suppose the infinite cyclic group $J$ acts as covering transformations on $Y$ and $j$ is a generator of $J$. Then $j$ is {\it coaxial} if given any compact set $C$ in $Y$, there is a larger compact set $D$ in  $Y$ such that any loop in $Y-J\cdot D$ is homotopically trivial in $Y-C$ (here $J\cdot D=\cup _{m\in \mathbb Z}(j^m(D))$).
\end{definition}

\begin{theorem} [Theorem 0.1, \cite{GGM16}] \label{GGm}  
If there exists an infinite cyclic group $J$ acting as covering transformations on $Y$ and generated by a coaxial homeomorphism then there is an infinite tree $\mathbb T$ and a proper 2-equivalence $\tilde f:Y\to \mathbb T\times \mathbb R$. 
\end{theorem} 
Since a proper 2-equivalence preserves the pro-isomorphism type of the fundamental pro-group of spaces we have:

\begin{corollary}[Corollary 0.2, \cite{GGM16}] The existence of such a coaxial $j$ ensures that $Y$ has semistable and pro-free fundamental pro-group. 
\end{corollary}

For comparison, F. E. A. Johnson \cite{Joh75} considered short exact sequences of finitely presented groups $1\to H\to G\to K\to 1$ and in certain cases proved that $G$ was 2-equivalent to $H\times K$. 

An inverse sequence of groups is {\it pro-mono} if it is pro-isomorphic to an inverse sequence of groups with monomorphic bonding maps. This notion is dual to semistability (pro-epi).  If an inverse sequence of groups is both pro-mono and pro-epi, then it is stable (pro-isomorphic to an inverse system with bonding maps all isomorphisms). 
Building on earlier work by D. Wright, \cite{W92}, R. Geoghegan and C. Guilbault \cite{GG12} proved:

\begin{theorem} [Theorem 1.2, \cite{GG12}] \label{Gg12} 
If the fundamental pro-group at infinity of $Y$ is pro-mono and there is an infinite cyclic group $J$ acting as covering transformations on $Y$ then the fundamental pro-group of $Y$ is stably a finitely generated free group.
\end{theorem}

Theorem \ref{Gg12} is a corollary of Theorem \ref{GGm}, because a lemma from \cite{W92} implies that when $Y$ has pro-mono fundamental pro-group then, given any infinite cyclic group $J$ acting as covering transformations on $Y$, the generator $j$ of $J$ is coaxial.

\subsubsection {Semistability and Co-semistability at $\infty$}\label{SCS} 

The main theorem of \cite{GGM1} shows that a locally finite, simply connected, 1-ended  CW-complex $Y$ has semistable fundamental group at $\infty$ if and only if $Y$ satisfies two other conditions. These conditions involve non-cocompact group actions on $Y$ and are rather technical, but the ideas are exceptionally useful in showing a certain class of groups (long though to contain potential non-semistable groups) only contains groups with semistable fundamental group at $\infty$.

In its simplest form a group $G$ acts cocompactly on $Y$ and this result restricts attention to the subaction on $Y$ of an infinite finitely
generated subgroup $J$ having infinite index in $G$. The
topology of $Y$ at infinity is separated into 
\textquotedblleft the $J$-directions\textquotedblright\ and 
\textquotedblleft the directions in $Y$ orthogonal to $J$\textquotedblright, with the main result
being that, having
appropriate analogs of semistability in the two directions, implies that
$Y$ has semistable fundamental group at $\infty$.

First a special case of the
Main Theorem of \cite{GGM1}. A more far-reaching, but more technical,
version of the Main Theorem is given later in this section

\begin{definition}\label{ss+cossD} 
Suppose $J$ is a finitely generated group acting by cell-permuting covering
transformations on a 1-ended locally finite and simply connected CW
complex $Y$. Let $\Gamma\left(  J,J^{0}\right)  $ be the Cayley graph of $J$ with
respect to a finite generating set $J^{0}$ and let $m:\Gamma\rightarrow Y$ be
a $J$-equivariant map. Then\medskip

\begin{itemize}
\item[a)] $J$ is {\it semistable at infinity in }$Y$ if for any compact set
$C\subseteq Y$ there is a compact set $D\subseteq Y$ such that if $r$ and $s$
are two proper rays (based at the same point) in $\Gamma\left(  J,J^{0}\right)
-m^{-1}\left(  D\right)  $ then $mr$ and $ms$ are properly homotopic in $Y-C$
relative to $mr\left(  0\right)  =ms\left(  0\right)  $.

Standard methods show that the above property does not depend on the choice of finite
generating set $J^{0}$ or the $J$-equivariant map $m$.

\item[b)] $J$ is {\it co-semistable at infinity in }$Y$ if for any compact
set $C\subseteq Y$ there is a compact set $D\subseteq Y$ such that for any
proper ray $r$ in $Y-J\cdot D$ and any loop $\alpha$ based at $r(0)$ whose image lies in
$Y-J\cdot D$, $\alpha$ can be pushed to infinity by a proper homotopy in $Y-C$ with the base point
tracking  $r$.
\end{itemize}
\end{definition}

\begin{remark}
The ``initial" definition of co-semistability given in \cite{GGM1} should have been the one listed in item $(b)$ above. Instead, it required all loops $\alpha$ in $Y-D$ (as opposed to only loops in $Y-J\cdot D$) to satisfy the conclusion of $(b)$. In that sense, the theorem listed next is considerably stronger than Theorem 1.1 stated in \cite{GGM1}. One need only consider Definition \ref{Coss} to see why the following theorem holds under the weaker hypothesis listed here. The main theorem of \cite{GGM1} - Theorem 3.1 - and the more delicate definitions of co-semistability are stated correctly in \cite{GGM1}.
\end{remark}
\begin{theorem}
[A special case]\label{Theorem: Special case of Main Theorem}If
$J$ is both semistable at infinity in $Y$ and co-semistable at infinity in
$Y$, then $Y$ has semistable fundamental group at infinity.
\end{theorem}

We find the following remark in \cite{GGM1}.

\begin{remark} \label{Remark in Intro}
\hspace{.25in}
\begin{enumerate}

\item  To our knowledge, the theorems
stated here are the first non-obvious results that imply semistable
fundamental group at $\infty$ for a space $Y$ which might not admit a
cocompact action by covering transformations.

\item \label{Intro remark 2}In the special case where $J$ is an infinite
cyclic group, condition $(a)$ above is always satisfied since $\Gamma\left(
J,J^{0}\right)  $ can be chosen to be homeomorphic to $%
\mathbb{R};
$ any two proper rays in $%
\mathbb{R}
$ which begin at the same point and lie outside a nonempty compact subset of $%
\mathbb{R}
$ are properly homotopic in their own images. Moreover, since condition
$\left(  b\right)  $ is implied by the main hypothesis of \cite{GG12} (via
\cite[Lemma 3.1]{W92} or \cite[Th.16.3.4]{G}), Theorem
\ref{Theorem: Special case of Main Theorem} implies the main theorem of
\cite{GG12} .

\item The converse of Theorem \ref{Theorem: Special case of Main Theorem} is
trivial. If $Y$ is semistable at infinity and $J$ is any finitely generated
group acting as covering transformations on $Y$, it follows directly from the
definitions that $J$ is both semistable at infinity in $Y$ and co-semistable at
infinity in $Y$. So, our theorem effectively reduces checking the semistability
of the fundamental group at infinity of a space to separately checking two
strictly weaker conditions.

\item \label{Intro remark 4}In the more general version of Theorem
\ref{Theorem: Special case of Main Theorem} (see Theorem \ref{MTGGM}), the group $J$
will be permitted to vary for different choices of compact set $C$. No over-group containing these various groups is needed unless one wants to extend such results to locally compact ANRs. 
\end{enumerate}
\end{remark}

Now for the main theorem of \cite{GGM1}.
A detailed discussion of the definitions that
go into the theorem will follow. 

\begin{theorem}\label{MTGGM}
Let $Y$ be a 1-ended simply connected locally finite
CW complex. Assume that for each compact subset $C_{0}$ of $Y$ there is a
finitely generated group $J$ acting as cell preserving covering
transformations on $Y$, so that (a) $J$ is semistable at $\infty$ in $Y$ with
respect to $C_{0}$, and (b) $J$ is co-semistable at $\infty$ in $Y$ with
respect to $C_{0}$. Then $Y$ has semistable fundamental group at $\infty$.
\end{theorem}

\begin{remark}
If there is a group $G$ (not necessarily finitely generated) acting as
covering transformations on $Y$ such that each of the groups $J$ of Theorem
\ref{MTGGM} is isomorphic to a subgroup of $G$, then the condition that $Y$ is a
locally finite CW complex can be relaxed to: $Y$ is a locally compact absolute
neighborhood retract (ANR) (see Corollary 9.1 of \cite{GGM1}).
\end{remark}

In what follows, the {\it distance} between vertices of a CW complex will always
be the number of edges in a shortest edge path connecting them. The space $Y$
is a 1-ended simply connected locally finite CW complex, and for each compact
subset $C_{0}$ of $Y$, $J(C_{0})$ is an infinite finitely generated group
acting as covering transformations on $Y$ and preserving some locally finite
cell structure on $Y$. Fix $\ast$ a base vertex in $Y$. Let $J^{0}$ be a
finite generating set for $J$ and $\Gamma(J,J^{0})$ be the Cayley graph of
$J$ with respect to $J^{0}$. Let $z_{(J,J^{0})}:(\Gamma(J,J^{0}
),1)\rightarrow(Y,\ast)$ be a $J$-equivariant map so that each edge of
$\Gamma$ is mapped to an edge path of length $\leq K(J^{0})$. If $r$ is an
edge path in $\Gamma$, then $z(r)$ is called a $\Gamma${\it -path} in
$Y$. The vertices $J\ast$ are called $J${\it -vertices}.

\begin{definition} 
If $C_{0}$ is a compact subset of $Y$ then the group $J$ is {\it semistable
at $\infty$ in $Y$ with respect to $C_{0}$} if there exists a compact set $C$
in $Y$ and some (equivalently any) finite generating set $J^{0}$ for $J$ such
that for any third compact set $D$ and proper edge path rays $r$ and $s$ in
$\Gamma(J,J^{0})$ which are based at the same vertex $v$ and are such that
$z(r)$ and $z(s)$ have image in $Y-C$ then there is a path $\delta$ in $Y-D$
connecting $z(r)$ and $z(s)$ such that the loop determined by $\delta$ and the
initial segments of $z(r)$ and $z(s)$ is homotopically trivial in $Y-C_{0}$
\end{definition}
Note that this definition requires less than one requiring $z(r)$ and $z(s)$
be properly homotopic rel$\{z(v)\}$ in $Y-C_{0}$. 
 It may be that the path $\delta$ is not homotopic to a path
in the image of $z$ by a homotopy in $Y-C_{0}$. This definition is independent
of generating set $J^{0}$ and base point $\ast$ by a standard argument,
although $C$ may change as $J^{0}$, $\ast$ and $z$ do. When $J$ is semistable
at infinity in $Y$ with respect to $C_{0}$, we may say $J$ is
\textit{semistable at $\infty$ in $Y$ with respect to $J^{0}$, $C_{0}$, $C$
and $z$}. Observe that if $\hat C$ is compact containing $C$ then $J$ is also
semistable at $\infty$ in $Y$ with respect to $J^{0}$, $C_{0}$, $\hat C$ and
$z$.

If $J$ is 1-ended and semistable at $\infty$ or 2-ended, then $J$ is always
semistable at $\infty$ in $Y$ with respect to any compact subset $C_{0}$ of
$Y$.

The notion of $J$ being co-semistable at infinity in a space $Y$ is a bit
technical, but has its roots in a simple idea that is fundamental to the main
theorems of \cite{GG12} and \cite{W92}. In both of these papers 
 $J$ is an infinite cyclic group acting as covering transformations on a 1-ended
simply connected space $Y$ with pro-monomorphic fundamental group at $\infty$.
Wright \cite{W92} showed that under these conditions the following could be proved:

$(\ast)$ Given any compact set $C_{0}\subset Y$ there is a compact set
$C\subset Y$ such that any loop in $Y-J\cdot C$ is homotopically trivial in
$Y-C_{0}$.

Condition ($\ast$) is all that is needed in \cite{GG12} and \cite{W92} in
order to prove the main theorems. In \cite{GGM16} condition ($\ast$) is used
to show $Y$ is proper 2-equivalent to $T\times\mathbb{R}$ (where $T$ is a
tree). Interestingly, there are many examples of finitely presented groups $G
$ (and spaces) with infinite cyclic subgroups satisfying $(\ast)$ but the
fundamental group at $\infty$ of $G$ is not pro-monomorphic (see
\cite{GGM16}). In fact, if $G$ has pro-monomorphic fundamental group at
$\infty$, then either $G$ is simply connected at $\infty$ or (by a result of
B. Bowditch \cite{Bo04}) $G$ is virtually a closed surface group and $\pi
_{1}^{\infty}(G)=\mathbb{Z}$.

The co-semistability definition generalizes the conditions of $(\ast)$ in two
fundamental ways and Theorem \ref{MTGGM} still concludes that $Y$ has semistable
fundamental group at $\infty$ (just as in the main theorem of \cite{GG12}).

(1) First  $J$ is expanded from an infinite cyclic group to an arbitrary finitely
generated group and $J$ is allowed to change as compact subsets of $Y$ become larger.

(2) The requirement that loops in $Y-J\cdot C$ be trivial in
$Y-C_{0}$ is weakened to only requiring that loops in $Y-J\cdot C$ can be ``pushed"
arbitrarily far out in $Y-C_{0}$.

Now for the co-semistability definition. A subset $S$ of
$Y$ is \textit{bounded} in $Y$ if $S$ is contained in a compact subset of $Y$.
Otherwise $S$ is \textit{unbounded} in $Y$. Fix an infinite finitely generated
group $J$ acting as covering transformations on $Y$ and a finite generating
set $J^{0}$ of $J$. Assume $J$ respects a cell structure on $Y$. Let
$p:Y\rightarrow J\backslash Y$ be the quotient map. If $K$ is a subset of $Y$,
and there is a compact subset $D$ of $Y$ such that $K\subset J\cdot D$
(equivalently $p(K)$ has image in a compact set), then $K$ is a $J$%
-\textit{bounded} subset of $Y$. Otherwise $K$ is a $J$-\textit{unbounded}
subset of $Y$. If $r:[0,\infty)\rightarrow Y$ is proper and $pr$ has image in
a compact subset of $J\backslash Y$ then $r$ is said to be $J$%
-\textit{bounded}. Equivalently, $r$ is a $J$-bounded proper edge path in $Y$
if and only if $r$ has image in $J\cdot D$ for some compact set $D\subset Y$.
In this case, there is an integer $M$ (depending only on $D$) such that each
vertex of $r$ is within (edge path distance) $M$ of a vertex of $J\ast$. Hence
$r$ `determines' a unique end of the Cayley graph $\Gamma(J,J^{0})$.

\begin{definition} \label{Coss}  For a non-empty compact set $C_{0}\subset Y$ and finite subcomplex $C$
containing $C_{0}$ in $Y$, let $U$ be a $J$-unbounded component of $Y-J\cdot
C$ and let $r$ be a $J$-bounded proper ray with image in $U$. We say $J$ is
co-\textit{semistable at $\infty$ in }$U$ \textit{with respect to $r$, $C_0$  and}
$C$ if for any compact set $D$ and loop $\alpha:[0,1]\rightarrow U$ with
$\alpha(0)=\alpha(1)=r(0)$ there is a homotopy $H:[0,1]\times\lbrack
0,n]\rightarrow Y-C_{0}$ such that $H(t,0)=\alpha(t)$ for all $t\in
\lbrack0,1]$ and $H(0,s)=H(1,s)=r(s)$ for all $s\in\lbrack0,n]$ and
$H(t,n)\subset Y-D$ for all $t\in\lbrack0,1]$. This means that $\alpha$ can be
pushed along $r$ by a homotopy in $Y-C_{0}$ to a loop in $Y-D$. We say $J$ is
\textit{co-semistable at $\infty$ in }$Y$ \textit{with respect to $C_{0}$} if there exists a finite complex $C$ containing $C_0$
such that  $J$ is co-semistable at $\infty$ in $U$ with respect to $r$, $C_0$ and $C$ for each $J$-unbounded component $U$ of $Y-J\cdot C$, and any proper
$J$-bounded proper ray $r$ in $U$. 
\end{definition}

Note that if $\hat{C}$ is a finite complex
containing $C$, then $J$ is also co-semistable at $\infty$ in $Y$ with respect
to $C_{0}$ and $\hat{C}$.
It is important to notice that this definition only requires that loops in $U$
can be pushed arbitrarily far out in $Y-C_{0}$ along proper $J$-bounded proper rays
in $U$ (as opposed to all proper rays in $U$).

We extend Definition \ref{ss+cossD} 

\begin{definition}\label{ss+cossD2} 
Suppose $J$ is a finitely generated group acting by cell-permuting covering
transformations on a 1-ended locally finite and simply connected CW
complex $Y$. Let $\Gamma\left(  J,J^{0}\right)  $ be the Cayley graph of $J$ with
respect to a finite generating set $J^{0}$ and let $m:\Gamma\rightarrow Y$ be
a $J$-equivariant map. Then\medskip

\begin{itemize}
\item [c)] $J$ is {\it simply connected at $\infty$ in $Y$} if for any compact set $C$ in $Y$ there is a compact set $D$ in $Y$ such that any loop $\alpha$ in $\Gamma(J,J^0)-m^{-1}(D)$ has the property that $m(\alpha)$ is homotopically trivial in $Y-C$. 

\item [d)] $J$ is {\it coaxial in $Y$} if for any compact set $C$ in $Y$ there is a compact set $D$ in $Y$ such that any loop  in $Y-JD$ is homotopically trivial in $Y-C$. 

\item [e)] $J$ is {\it co-connected in $Y$} if for any compact set $D$ in $Y$, $Y-JD$ has exactly one $J$-unbounded component. 

\medskip
\end{itemize}
\end{definition}

The following two results are straightforward. 

\begin{theorem} \label{cosc1} 
Suppose $Y$ is a simply connected locally finite complex and $J$ is a finitely generated 1-ended group acting as covering 
transformations on $Y$. Assume 
\begin{itemize}
\item[(i)] $J$ is simply connected at $\infty$ in $Y$.
\item [(ii)] $J$ is coaxial in $Y$.
\end{itemize}
Then $Y$ is simply connected at each end.
\end{theorem}

\begin{theorem} \label{cosc2} 
Suppose $Y$ is a simply connected locally finite complex and $J$ is a finitely generated (not necessarily 1-ended) group acting as covering transformations on $Y$. Assume 
\begin{itemize}
\item[(i)] $J$ is co-connected in $Y$.
\item [(ii)] $J$ is coaxial in $Y$.
\end{itemize}
Then $Y$ is simply connected at each end.
\end{theorem}

\begin{theorem} \label{ProTriv}
If the finitely presented group $G$ is simply connected at $\infty$, then $G$ is semistable at $\infty$, the fundamental pro-group  $\pi_1(\varepsilon G)$ is pro-trivial and the fundamental group at $\infty$,  $\pi_1^e(G)$ is the trivial group.
\end{theorem}

\subsubsection{The Cayley 2-Complex for a Finite Presentation of a Group} \label{C2C}
Corollary \ref{pro2} allows us to select spaces that reflect special algebraic properties of a finitely presented group in order to show the group is semistable or simply connected at $\infty$. Many of the theorems that conclude a finitely presented group is semistable or simply connected at $\infty$ have their proofs based in a ``Cayley 2-complex" of a finite presentation of that group.  

\begin{definition}
Given a finite presentation $\mathcal P=\langle S \ | \ R\rangle$ for a group $G$, the {\it standard} 2-complex $X(\mathcal P)$ for this presentation has a single vertex $v$ and one directed loop (labeled $s$) at $v$ for each generator $s\in S$.  For each relation  $r=a_1^{\epsilon_1}a_2^{\epsilon_2} \cdots a_n^{\epsilon_n}\in R$ (here $a_i\in S$ and $\epsilon_i=\pm 1$ for all $i$) there is an $n$-gon $D_r$ with boundary labeled by $r$. The 2-cell $D_r$ is attached along it's boundary to the loops at $v$ so that the $i^{th}$ edge (labeled $a_i$) is wrapped around the loop with labeled $a_i$ with positive orientation if $\epsilon_i=1$ and with negative orientation if $\epsilon_i=-1$. The resulting 2-complex is $X(\mathcal P)$
\end{definition}

\noindent An elementary application of van Kampen's Theorem implies:
\begin{theorem}
If $G$ is a finitely presented group with finite presentation $\mathcal P$ then $\pi_1(X(\mathcal P))$ is isomorphic to $G$. 
\end{theorem}

\begin{definition}
If $G$ is a group with finite presentation $\mathcal P$ then let $\Gamma(\mathcal P)$ be the universal cover of $X(\mathcal P)$. The 2-complex $\Gamma(\mathcal P)$ is the {\it Cayley 2-complex} for $\mathcal P$. 
\end{definition}
The following is elementary:
\begin{theorem}
If $\mathcal P=\langle S\ | \  R\rangle$ is a finite presentation of a group $G$, then the 1-skeleton of $\Gamma(\mathcal P)$ is the Cayley graph $\Gamma(G,S)$. If for each vertex $v$ of $\Gamma(G,S)$ and each $r\in R$, a 2-cell is attached to $\Gamma(G,S)$ according to the word $r$, then the resulting 2-complex is $\Gamma(\mathcal P)$. 
\end{theorem}

\subsubsection{Semistability at $\infty$ for Finitely Generated Groups}\label{FGss}

We now define the notion of semistability for a finitely generated group following Mihalik  \cite{M4}. The (finitely generated) Lamplighter group is not semistable (see section \ref{NonSS}). 
Suppose $G$ is a finitely generated group with generating set $S\equiv \{s_1, s_2,\ldots , s_n\}$ and let $\Gamma
(G,S)$ be the Cayley graph of $G$ with respect to this generating set.  Suppose $\{r_1, r_2,\ldots , r_m\}$ is a finite set of relations in $G$ written in the letters $\{s_1^\pm, s_2^\pm,\ldots , s_n^\pm\}$.  For any vertex $v\in \Gamma(G,S)$, there is an edge path cycle labeled $r_i$ at $v$.   The two dimensional CW-complex $\Gamma_{(G,S)}(r_1,\ldots , r_m)$ is obtained by attaching to each vertex of $\Gamma(G,S)$, $2$-cells corresponding to the relations $r_1,\ldots ,r_n$.

Mihalik  \cite{M4}, shows that if $S$ and $T$ are finite generating sets for the group $G$ and there are finitely many $S$-relations $P$ such that $\Gamma_{(G,S)}(P)$ is semistable at $\infty$, then there are finitely many $T$-relations $Q$ such that $\Gamma_{(G,T)}(Q)$ is semistable at $\infty$. Hence the following definition:

\begin{definition} \label{defss3}
We say {\it $G$ is semistable at $\infty$} if  for some finite generating set $S$ for $G$ and finite set of $S$-relations $P$ the complex $\Gamma_{(G,S)}(P)$ is semistable at $\infty$. 
\end{definition}

There is another way to view $\Gamma_{(G,S)}(P)$. Let $\phi:F(S)\to G$ be the quotient map sending $s$ to $s$ for all $s\in S$. If $G_1$ has finite presentation $\langle S\ | \ P\rangle$, and $\phi_1:F(S)\to G_1$ is the quotient map sending $s$ to $s$ for all $s\in S$, then $ker (\phi_1)\subset ker(\phi)$ and $G_1$ maps onto $G$ with kernel $\phi_1(ker(\phi))$. Let $\Gamma_1$ be the Cayley 2-complex for the presentation $\langle S \ | \ P\rangle $. The group of covering transformations of $\Gamma_1$ is $G_1$ and $\phi_1(ker(\phi))$ is a normal subgroup of $G_1$. The space $\Gamma_{(G,S)}(P)$ is the quotient of $\Gamma_1$ by $\phi_1(ker(\phi))$ (with $G$ as a group of covering transformations).

\[
\xymatrix{
F(S) \ar[d]^{\phi_1} \ar[dr]^\phi & \\
G_1 \ar[r] & G
 }
\]
    
Note that if $G$ has finite presentation $\langle S \ | \ P\rangle$, then $G$ is semistable at $\infty$ with respect to Definition \ref{defss2} if and only if $G$ is semistable at $\infty$ with respect to Definition \ref{defss3} if and only if $\Gamma_{(G,S)}(P)$ is semistable at $\infty$.

\begin{lemma} [(Lemma 2,  \cite{M4}] \label{rel} 
Suppose the finitely generated group $G$ is 1-ended and semistable at $\infty$. If $S$ is a finite generating set for $G$ and $P$ is a finite set of $S$-relations in $G$ such that $\Gamma _{(G,S)}(P)$ is semistable at $\infty$, then there is a finite set $Q$ of $S$-relations such that: if $r$ and $s$ are proper rays 
in $\Gamma_{(G,S)}(P\cup Q)$, with $r(0)=s(0)$, then $r$ is properly homotopic to $s$ rel$\{r(0)\}$. 
\end{lemma}

\begin{remark}
In Theorem \ref{FGSS6} we show that in fact the set of relations $Q$ in the previous lemma are unnecessary in order to draw the same conclusion.  If $\Gamma_{(G,S)}(P)$ is semistable at $\infty$, and $r$ and $s$ are proper rays 
in $\Gamma_{(G,S)}(P)$ with $r(0)=s(0)$, then $r$ is properly homotopic to $s$ rel$\{r(0)\}$. 
\end{remark}

\begin{definition}\label{def4} 
For a group $G$ with finite generating set $S$ and a subset $T$ of $S$, we say an edge path in $\Gamma(G, S)$ is a {\it $T$-path} if each edge of the path is labeled by an element of  $T^{\pm}$. If the path is infinite and proper we call it a proper {\it  $T$-ray}.
\end{definition}
\noindent The next result generalizes Lemma \ref{rel} (compare with Theorem \ref{ssequiv} (6)). 
\begin{theorem} \label{FGSS6}  Suppose $G$ is a finitely generated group that is 1-ended and semistable at $\infty$. Let $S$ be a finite generating set for $G$ and $P$ a finite set of $S$-relations such that $\Gamma_{(G,S)}(P)$ is semistable at $\infty$. Then two proper rays $r$ and $s$ in $\Gamma_{(G,S)}(P)$, with $r(0)=s(0)$ are properly homotopic $rel\{r(0)\}$.
\end{theorem}
\begin{proof}
Let $X=\Gamma_{(G,S)}(P)$ and let $l$ be a geodesic edge path line in $X$ through the identity $\ast$, parametrized so that $l(0)=\ast$. Let $l^+(n)=l(n)$ for $n\geq 0$ and $l^-(n)=l(n)$ for all $n\leq 0$. Theorem \ref{ssequiv} (4) (with $C=\emptyset$) implies there is a compact ball $B_K(\ast)$ of radius $K$, such that any two proper rays $r$ and $s$ in $X-B_K(\ast)$ (with $r(0)=s(0)$) are properly homotopic $rel\{r(0)\}$.

For $v\in G(\equiv X^0)$ let $vl^+$ and $vl^-$ be the translates of $l^+$ and $l^-$ respectively to $v$. Given any vertex $v\in X-B_{2K+1}(\ast)$ only one of the geodesic rays $vl^+$ or $vl^-$ can intersect $B_K(\ast)$. Hence either there are infinitely many vertices $v\in X$ such that $vl^+$ avoids $B_K(\ast)$ or infinitely many vertices $v\in X$ such that $vl^-$ avoids $B_K(\ast)$. Without loss, assume the former holds. It is enough to show that every proper ray $r$ at $\ast$ is properly homotopic to $l^+$ relative to $\ast$. 

\begin{figure}
\vbox to 3in{\vspace {-1.7in} \hspace {.1in}
\hspace{-1 in}
\includegraphics[scale=1]{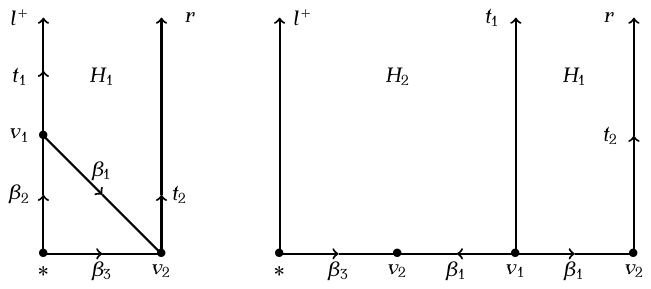}
\vss }
\vspace{-1.2in}
\caption{Combining the $H_i$} 
\label{Hfig}
\end{figure}

Choose a path $\beta_1$ from a vertex $v_1$ of $l^+$ to a vertex $v_2$ of $r$ such that the image of $\beta_1$ and the tail of $r$ at $v_2$ avoid $B_K(\ast)$. By the definition of $K$, $t_1$ (the tail of $l^+$ at $v_1$) is properly homotopic (relative to $v_1$) to $\beta_1$ followed by $t_2$ (the tail of $r$ at $v_2$) by a proper homotopy $H_1$. Let $\beta_2$ be the initial segment of $l^+$ from $\ast$ to $v_1$ and let $\beta_3$ be the initial segment of $r$ from $\ast$ to $v_2$. Let $\beta$ be the loop $(\beta_3, \beta_1^{-1}, \beta_2^{-1})$.  Choose a vertex (infinitely many exist) $w\in X$ such that the translate $wl^+$ of $l^+$ to $w$ and the translate $w\beta$ both avoids $B_K(\ast)$. By the definition of $K$, $wl^+$ is properly homotopic to $(w\beta, wl^+)$ relative to $w$. Translating by $w^{-1}$, $l^+$ is properly homotopic to $(\beta, l^+)$ relative to $\ast$. Equivalently, $l^+$ is properly homotopic to $(\beta_3, w\beta_2^{-1},t_1)$ relative to $\ast$ by a proper homotopy $H_2$. Combining $H_1$ and $H_2$ as in Figure \ref{Hfig} (and canceling $(\beta_1,\beta_1^{-1})$) shows that $l^+$ and $r$ are properly homotopic relative to $\ast$. 
\end{proof}

\subsubsection{Simply Connected at $\infty$ for Finitely Generated Groups}\label{GSSCinf}

The following definition defines what it means for a finitely generated subgroup of a finitely presented group to be simply connected at $\infty$ relatively to the finitely presented over group. This notion has been used successfully to prove interesting classes of finitely presented groups are simply connected at $\infty$. 

\begin{definition}\label{SCin} 
A finitely generated subgroup $A$ of a finitely presented group $G$ is {\it simply connected at $\infty$ in $G$ (or relative to $G$)} if for some (equivalently any by Lemma \ref{equiv} with $N=0$) finite presentation $\mathcal P=\langle \mathcal A,\mathcal B\ | \ R\rangle$ of the group $G$  (where $\mathcal A$ generates $A$ and $\mathcal A\cup \mathcal B$ generates $G$), the 2-complex $\Gamma_{(G,\mathcal A\cup \mathcal B)}(R)=\Gamma(\mathcal P)$ has the following property:

Given any compact set $C\subset \Gamma ( \mathcal P)$ there is a compact set $D\subset \Gamma ( \mathcal P)$ such that any edge path loop in  $\Gamma (A,\mathcal A)-D$ is homotopically trivial in $\Gamma ( \mathcal P)-C$. 
\end{definition}

In order to define what  it means for a finitely generated group $G$ to be simply connected at $\infty$, we must know that $G$ embeds in some finitely presented group. In 1961,G. Higman proved:

\begin{theorem} [Higman, \cite{Hig61}]  A finitely generated infinite group $G$ can be embedded in a finitely presented group if and only if  $G$ is recursively presented.
\end{theorem}

\begin{definition}\label{SCfg} 
A finitely generated and recursively presented group $A$ is {\it simply connected at $\infty$} if for any finitely presented group $G$ and subgroup $A'$ isomorphic to $A$, the subgroup $A'$ is simply connected at $\infty$ in $G$. 
\end{definition}

\begin{remark} If the finitely generated group $A$ is not recursively presented, the following is a potential simple connectivity at $\infty$ definition for $A$ (it is not clear how well this definition meshes with the one for recursively presented $A$.):

Let $\mathcal A$ be a finite generating set for $A$ with Cayley graph $\Gamma (A, \mathcal A)$. Suppose $A$ acts as covering transformations on a locally finite CW-complex $K$ and $m:\Gamma(A,\mathcal A)\to K$ is $A$-equivariant. Say $A$ is {\it simply connected at $\infty$ in $K$}  if  for each compact set $C$ in $K$ there is a compact set $D$ in $K$ such that for any loop $\alpha$ in $\Gamma(\mathcal A,A)-m^{-1}(D)$, the loop $m(\alpha)$ is homotopically trivial in $K-C$. 

(Such a complex always exists: If $\Gamma(\mathcal A,A)$ is ``cusped off" as in \cite{GMa08}, then the resulting locally finite 2-complex $K$, is in fact simply connected at $\infty$ and $A$ acts as covering transformations on $K$.)

Say $A$ is {\it simply connected at $\infty$} if whenever $A$ acts by covering transformations on a simply connected locally finite CW-complex $K$, then $A$ is simply connected at $\infty$ in $K$.
\end{remark}

Suppose that $A$ is a finitely presented group and $A$ satisfies the simply connected at $\infty$ condition of Definition \ref{defss2},
then $A$ satisfies Definition \ref{SCfg}. If $A$ is a finitely presented group satisfying Definition \ref{SCfg}, then if we let $A=G$ in Definition \ref{SCfg} we see $A$ satisfies Definition \ref{defss2}, and there is no ambiguity. 

\begin{lemma} [Lemma 2.9, \cite{M6}] \label{equiv} 
Suppose $A$ is a finitely generated subgroup of the finitely presented group $G$. Then $A$ is simply connected at $\infty$ in $G$ if and only if:

\noindent $(\dagger)$ For $\mathcal P$ an arbitrary finite presentation for $G$, $N\geq 0$ an integer and $C$ a compact subset of $\Gamma (\mathcal P)$, there is a compact set $D(C,N)\subset \Gamma (\mathcal P)$ such that if $\alpha$ is an edge path loop in $\Gamma(\mathcal P)-D$ and each vertex of $\alpha$ is within $N$ of some vertex  of $A(\subset \Gamma (\mathcal P))$ then $\alpha$ is homotopically trivial in $\Gamma-C$. 
\end{lemma}

\begin{remark} \label{R1}
For spaces (and finitely presented groups), simple connectivity at $\infty$ is stronger than semistability at $\infty$. In fact  Theorem \ref{ssequiv} (parts 1 and 2) state that a space $K$ is semistable at $\infty$ if and only if the fundamental pro-group of $K$ is pro-isomorphic to an inverse sequence of groups with epimorphic bonding maps. The space $K$ is simply connected at $\infty$ if and only if the fundamental pro-group of $K$ is pro-trivial. 
It is not clear whether or not our definition of simple connectivity at $\infty$ for a finitely generated group implies the group is semistable at $\infty$. Interestingly, these two notions can be combined in effective ways (see Lemma \ref{ss+sc}). 

Lemma \ref{equiv} implies the following. Suppose the finitely generated group $A$ is simply connected at $\infty$ in the finitely presented group $G$, $\mathcal P$ is a finite presentation for $G$, and  $v_1,\ldots, v_n$ are vertices of $\Gamma (\mathcal P)$. Then for any compact $C\subset \Gamma (\mathcal P)$ and integer $N\geq 0$ there is a compact set $D(C,N,\{v_1,\ldots ,v_n\})$ such that any loop in $\Gamma (\mathcal P)-D$, each of whose vertices is within $N$ of $v_iA$ for some $i\in \{1,\ldots ,n\}$, is homotopically trivial in $\Gamma (\mathcal P)-C$. 
What is not guaranteed is a compact set $D(C,N)$ satisfying the following: For all $v\in G$ and any edge path loop $\alpha$ in $\Gamma (\mathcal P)-D$ with each vertex of $\alpha$  within $N$ of $vA$, the loop $\alpha$ is homotopically trivial in $\Gamma-C$. We do gain this enhanced condition (see Lemma \ref{ss+sc}) if $A$ is both simply connected at $\infty$ in $G$ and semistable at $\infty$.
\end{remark}

\begin{lemma} [(Lemma 2.10, \cite{M6}] \label{ss+sc} 
Suppose $A$ is a finitely generated semistable at $\infty$ group and $A$ is simply connected at $\infty $ in the finitely presented group $G$, then

\noindent $(\ddagger)$ For $\mathcal P$ an arbitrary finite presentation for $G$, $N\geq 0$ an integer and $C$ a compact subcomplex of $\Gamma (\mathcal P)$, there is a compact set $D(C,N)\subset \Gamma (\mathcal P)$ such that if $\alpha$ is an edge path loop in $\Gamma (\mathcal P)-D$, $v$ is an element of $G$, and each vertex of $\alpha$ is within $N$ of the coset $vA(\subset \Gamma (\mathcal P)$ then $\alpha$ is homotopically trivial in $\Gamma (\mathcal P)-C$. 
\end{lemma} 

\newpage
\subsection{Semistability and Simple Connectivity at Infinity Results for Classes of Groups}\label{Gresults}

The theory of simple connectivity at $\infty$ preceded the semistability theory. In what follows, there are far more semistability results than simple connectivity at $\infty$ results, but several results have both a semistability and simple connectivity at $\infty$ component. First some historical prospective.

\subsubsection{Semistability, Simple Connectivity at $\infty$ and Manifolds} \label{CC}
To our knowledge, the first papers to consider the fundamental group at $\infty$ of a finitely presented group were those of F. E. A. Johnson \cite{Joh74}, \cite{Joh75} and R. Lee and F. Raymond \cite{LR75}. All of these papers were primarily concerned with ends of manifolds. 
One of the first general simple connectivity at $\infty$ results was due to Stallings. 
In 1962, J. Stallings defined what it means for a space to be $n$-connected at $\infty$, and proved the following:

\begin{theorem}  [\cite{St62}] \label{StallSC}  If $V^n$, $n\geq 5$, is a contractible PL $n$-manifold without boundary, then $V$ is PL-homeomorphic to $\mathbb R^n$ (with standard PL structure) if and only if $V$ is simply connected at $\infty$. 
\end{theorem}

The corresponding result for differentiable manifolds follows from work of J. Munkres \cite{Mun60}.

For $n\geq 3$,  $\mathbb R^n$ is simply connected at $\infty$. 
Stallings result can be extended to dimension 3 (following the work of Perelman) by L. Husch and T. Price (see \cite{HP70} and the addendum \cite{HP71}). For $\mathbb R^3$, DIFF, PL and TOP are the same. In dimension 4, there are exotic differential and PL structures on $\mathbb R^4$ (see S. Donaldson's paper \cite{Don83}). E. Luft's paper \cite{Luf67}, contains a topological version of Stallings theorem (in dimensions $n\geq 5$ and with an asymptotic hypothesis that is a bit stronger than simply connected at $\infty$). L. Siebenmann  improved Luft's result: 

\begin{theorem} [Theorem 1.1, \cite{Sie68}] 
Suppose $M^n$ is a contractible metrizable topological $n$-manifold without boundary, which is simply connected at $\infty$. If $n\geq 5$, then $M^n$ is homeomorphic to euclidean $n$-space $\mathbb R^n$.
\end{theorem}

The topological version for $\mathbb R^4$ is due to M. Freedman (see Corollary1.2 of \cite{Fre82}). 
These results make it clear that there is great value in understanding when the fundamental group of a closed manifold group is simply connected at $\infty$.  

 M. Davis \cite{Davis83} used right angled Coxeter groups to produce closed manifolds with infinite fundamental group in dimensions $n\geq 4$ whose universal covers were contractible but not simply connected at $\infty$ (and hence not homeomorphic to $\mathbb R^n$).

Before Perelman's famous work, one of the major pieces to Thurston's Geometrization Conjecture was:
 
 \medskip
 
\noindent {\bf Covering Conjecture for 3-Manifolds}: All closed irreducible 3-manifolds with infinite fundamental group are covered by $\mathbb R^3$.
 
\medskip

Semistability at $\infty$ tends to be significantly weaker than simple connectivity at $\infty$. But not for 3-manifolds. M. Brinn and T. Thickstun made a remarkable connection between semistability and the Covering Conjecture.

\begin{theorem} [See Theorem B, \cite{BT87}] and the discussion that follows.) 
A Closed irreducible 3-manifold with infinite fundamental group $G$ is covered by $\mathbb R^3$ if and only $G$ is semistable at $\infty$
\end{theorem}

The Covering Conjecture was shown to be true by the work of Perlman and the work that followed.

\begin{definition}\label{InTame}
A connected locally compact ANR $X$ is {\it inward tame} if for any compact set $C$ there is a compact set $D$ and homotopy $H:(X-D)\times [0,1]\to X$ such that $H(x,t) \in X-C$, $H(x,0)=x$ and $H(x,1)\in D-C$ for all $x\in X-D$ and all $t\in [0,1]$.
\end{definition}

Suppose $M^n$ and $N^n$  are closed (compact without boundary) orientable $n$-manifolds. Recall that $H_n(M^n)=\mathbb Z$. A continuous function $f:M^n\to N^n$ is {\it degree} $k$ if it induces a homomorphism on $H_n(M^n)$ that takes $1$ to $k$. C. Guilbault and F. Tinsley connect inward tame manifolds and semistability. 

\begin{theorem} [See \cite{GuTi2}] \label{Degree1}  
Suppose $M$ is an inward tame non-compact open $n$-manifold with 1-end then $M$ has semistable fundamental group at $\infty$. 
\end{theorem}

\begin{corollary} \label{Deg1}
Suppose $M$ is a closed manifold with inward tame and 1-ended universal cover. Then $\pi_1(M)$ has semistable fundamental group at $\infty$. 
\end{corollary}
A key principle underlying the proof of Theorem \ref{Degree1} is Hopf's surjectivity Theorem:

\begin{theorem} \label{HThm} 
If $M$ and $N$ are closed $n$-manifolds and $f:M\to N$ induces a degree $k\ne 0$ homomorphism on $n^{th}$ homology, then $f_\ast(\pi_1(M))$ has finite index $m$ in $\pi_1(N)$ and $m$ evenly divides $|k|$. In particular, if the degree of $f$ is $\pm 1$ then $f$ induces a surjection on $\pi_1$. 
\end{theorem}

\begin{proof}
Let $(\bar N, p)$ be the cover of $N$ corresponding to the subgroup $f_\ast(\pi_1(M))$ of $\pi_1(N)$. The map $f$ lifts to $\bar f:M\to  \bar N$ (so that $p\bar f=f$). The index of $f_\ast(\pi_1(M))$ in $\pi_1(N)$ is not infinite since otherwise, the $n^{th}$ homology of $\bar N$ is trivial and so the rank of $\bar f$ would be $0$. But the rank of $f$ is the product of the rank of $\bar f$ and the rank of $p$.
\end{proof}

A manifold $N^n$  with compact boundary is called a {\it homotopy collar} if the inclusion $\partial N^n\to N^n$ 
is a homotopy equivalence. If $N^n$ contains arbitrarily small homotopy collar
neighborhoods of infinity, we call $N^n$ a {\it pseudocollar}. In \cite{GuTi}, Guilbault and Tinsley characterize ``near pseudocollarable" manifolds of dimension $n\geq 6$ (Theorem 1.3) in terms of inward tameness and semistable fundamental group at $\infty$. Recall that a group is {\it perfect} if when abelianized, it becomes trivial. 

\begin{theorem} [Theorem 1.2, \cite{GuTi}] Let $M^n$ be an inward tame $n$-manifold with compact boundary. Then
$M^n$ has an $AP$-semistable (almost perfectly semistable) fundamental group at each
of its finitely many ends.
\end{theorem}

\subsubsection{Normal, Subnormal, Commensurated and Subcommensurated Subgroups}\label{SubCom} Theorem \ref{MainCM} generalizes many of the results of this section. We list a sequence of results that lead up to that theorem. 
In 1974,   R. Lee and F. Raymond first considered the fundamental group of an end of a group. In particular, they considered groups that are simply connected at $\infty$. 

\begin{theorem}  [\cite{LR75}] \label{LR} Let $G$ be a finitely presented group with normal subgroup $N$ isomorphic to $\mathbb Z^k$ and quotient $K= G/N$. Assume when $k=1$ that $K$ is 1-ended and that when $k=2$ that $K$ is not finite, and no restrictions when $k>2$. Then $G$ is simply connected at $\infty$. 
\end{theorem}

The next result is Theorem 1 of \cite{M87a}. It generalizes the case $k=1$ of Theorem \ref{LR} and is of particular value in showing the Sidki double of a group is simply connected at $\infty$ (see Theorem \ref{SCISD}).

\begin{theorem} [\cite{M87a}] \label{T187a} Let $1\to H\to G\to K\to 1$ be a short exact sequence of infinite groups. Assume that $G$ is finitely presented, $K$ is 1-ended, $H$ is abelian and $H$ contains an element of infinite order. Then $G$ is simply connected at $\infty$.
\end{theorem}

\begin{remark}
An elementary observation allows for a significantly simpler proof of Theorem \ref{T187a} than the one in \cite{M87a}. Suppose $g\in G$. If the (abelian) subgroup of $H$ generated by $h$ and $g^{-1}hg$ is not 2-ended, then it is 1-ended and Theorem \ref{P} implies $G$ is simply connected at infinity. So we may assume that $h$ and $g^{-1}hg$ generated a 2-ended subgroup of $H$. But then there are integers $m$ and $n$ (not equal to $0$) such that $g^{-1}h^mg=h^n$. 
This fact, (applied to each generator of $G$ - and its inverse) and combined with Lemma 1.1 of \ref{T187a} can be used to push each edge of a loop of the type in the conclusion of Lemma 1.1 arbitrarily far from the compact set $C$ of Lemma 1.1 between lines with all edge labels $h^{\pm 1}$. This allows one to kill the original loop by a homotopy avoiding $C$ and finishing the proof of Theorem \ref{T187a}. 
\end{remark}

The next result is a corollary of Theorem \ref{T187a}. 

\begin{corollary} \label{C187a1} 
Suppose $1\to H\to G\to K\to 1$ is a short exact sequence of infinite groups, $G$ is finitely presented and $h\in H$ is an element of infinite order such that $\langle h\rangle$ is normal in $H$. If $N$ is the normal closure of $h^2$ in $G$, then $N$ is abelian. If $N$ contains a copy of $\mathbb Z\oplus \mathbb Z$ or if $G/N$ is 1-ended, then $G$ is simply connected at infinity. \end{corollary} 
\begin{proof} If $N$ is abelian and $G/N$ is 1-ended then $G$ satisfies the hypothesis of Theorem \ref{T187a} and $G$ is simply connected at infinity. If $N$ contains a copy of $\mathbb Z\oplus \mathbb Z$, then $G$ satisfies the hypothesis of Theorem \ref{P}, implying $G$ is simply connected at infinity. 

It suffices to show that $N$ is abelian.  If $a\in H$ then $a^{-1}ha=h$ or $h^{-1}$ so that $a^{-2}ha^2=h$ and $h$ commutes with every squared element of $H$. Since $H$ is normal in $G$, this implies that for every $g\in G$, $g^{-1}hg$ commutes with every squared element of $H$. Every element of $N$ is a product of conjugates of elements of the form $g^{-1}h^{\pm 2}g$, so that $N$ is abelian. 
\end{proof}

\begin{remark}\label{R187a}
In the previous proof recall that if $g\in G$ then $g^{-1}hg$ commutes with every square element of $H$. Let $\bar H$ be the subgroup of $H$ generated by squares of elements of $H$. Every element of $N$ is a product of $G$-conjugates of $h^2$ and $h^{-2}$ and so $N\leq \bar H\leq H$ and $N$ is a central subgroup of $\bar H$. If $g\in G$ and $a\in H$, then $g^{-1}a^2g=(g^{-1}ag)^2\in \bar H$ so that $\bar H$ is normal in $G$. 
Theorem \ref{squares} shows that $H/\bar H$ is an abelian 2-group.
\end{remark}

 In 1982, B. Jackson generalized Theorem \ref{LR} and in 1983, M. Mihalik proved the first semistability at $\infty$ theorem for a class of finitely presented groups. These two  results serve as a starting point for this section. 

\begin{theorem} [\cite{J82a}] \label{J}  
If $H$ is an infinite, finitely presented, normal subgroup of infinite index in the finitely presented group $G$, and either $H$ or $G/H$ is 1-ended. Then $G$ is 1-ended and simply connected at $\infty$.   
\end{theorem}

The next result gives a converse to Theorem \ref{J}.
\begin{theorem} [Theorem 16.9.4, \cite{G}] \label{JConv}  
Suppose $1\to H\to G\to K\to 1$ is a short exact sequence of infinite finitely presented groups. Let $Y$, $X$ and $Z$ be finite path connected CW complexes whose fundamental groups are isomorphic to $H$, $G$ and $K$ respectively. Then the universal cover $\tilde X$ of $X$ is proper 2-equivalent to $\tilde Y\times \tilde Z$. 
\end{theorem}
It is elementary to see that if both $H$ and $K$ have more than one end (so that $\tilde Y$ and $\tilde Z$ have more than one end) then $\tilde Y\times \tilde Z$ is not simply connected at infinity. In fact more can be said. 

\begin{theorem} [Corollary 5.2, \cite{CLQR20}] \label{P2E}  Suppose $1\to H\to G\to K\to 1$ is a short exact sequence of infinite finitely presented groups, then $G$ is proper 2-equivalent to one of the following groups: $\mathbb Z\times \mathbb Z\times\mathbb Z$, $\mathbb Z\times \mathbb Z$ or $\mathbb F_2\times \mathbb Z$. 
\end{theorem} 
Note that when either $H$ or $K$ is 1-ended, then $G$ is proper 2-equivalent to $\mathbb Z\times \mathbb Z\times\mathbb Z$. When both $H$ and $K$ are 2-ended, then $G$ is proper 2-equivalent to $\mathbb Z\times \mathbb Z$.  In any other case, $G$ is proper 2-equivalent to $\mathbb F_2\times \mathbb Z$.

The first general semistability of groups result was published in 1983 in \cite{M1}. 
\begin{theorem} [\cite{M1})] \label{M1}  
If $H$ is an infinite, finitely generated, normal subgroup of infinite index in the finitely presented group $G$, then $G$ is 1-ended and semistable at $\infty$. 
\end{theorem}
In their 1987 seminal ``BNS invariants" paper R. Bieri, W. Neumann and R. Strebel prove the following result. 
\begin{theorem} [Theorem D, \cite{BNS87}] \label{BNS} 
Suppose $G$ is a finitely presented group with no non-abelian free subgroups and $rank (G/G')\geq 2$. Then $G$ contains a finitely generated normal subgroup $H \leq G$ with infinite cyclic quotient $G/H$. In fact, every normal
subgroup $L \leq  G$ with $G/L$ free abelian of rank 2 is contained in such an $H$.
\end{theorem}
The next result has the same proof as that of Theorem 1.2 of \cite{GM85} - the homological version of the same result. There are obvious consequences for solvable groups and amenable groups. 
\begin{theorem}\label{NOF2} 
Suppose $G$ is a finitely presented group which does not contain a free subgroup of rank 2, and suppose $\mathbb Z\oplus \mathbb Z$ is a quotient of $G$. Then $G$ is 1-ended and has semistable fundamental group at $\infty$. 
\end{theorem}
\begin{proof} 
Theorems \ref{BNS} and \ref{M1} imply $G$ is semistable at $\infty$. Note that since $H$ contains no free subgroup of rank 2, it is 1 or 2-ended.
\end{proof}


The main theorem in the 1993 PhD dissertation of V. Ming Lew generalized Theorem \ref{M1} and the main theorem of the 1990 PhD dissertation of J. Profio generalized Theorem \ref{J}: 

\begin{theorem} [\cite{Lew93}] \label{L}
Suppose $H$ is an infinite,  finitely generated, subnormal subgroup of the finitely generated group $G$: 
$$H = N_0 \lhd N_1 \lhd N_2 \lhd \ldots  \lhd N_k = G, \hbox{ for } k\geq 1$$
and $H$ has infinite index in $G$. Then $G$ is 1-ended and semistable at $\infty$. 
\end{theorem}

\begin{theorem} [\cite{P90}] \label{P}
Suppose $H\lhd N\lhd G$ is a normal series with $H$ and $G$ finitely presented, and $H$ 1-ended and of infinite index in $G$. Then $G$ is simply connected at $\infty$.
\end{theorem}

\begin{theorem} [Theorem 1, \cite{MMSS96}]
Suppose $N\leq A\leq G$ are groups with $A$ 1-ended, finitely presented and of infinite index in the finitely presented group $G$. If $N$ is a non-locally finite normal subgroup of $G$, then $G$ is semistable at $\infty$. 
\end{theorem}

\begin{question}
Under the hypothesis of the previous result, is $G$ simply connected at $\infty$? Does the conclusion of the previous result remain true if the hypothesis on $A$ is replaced by $A$ is finitely generated and infinite? What if the normality hypothesis on $N$ is replaced by subcommensurated? 
\end{question}

\begin{theorem} [(Theorem 2, \cite{MMSS96}] \label{Subset}  
Suppose $1\to H\to G\to K\to 1$ is a short exact sequence of infinite finitely generated groups, with $G$ finitely presented. If $K$ is 1-ended and $H$ is contained in a finitely presented subgroup $L$ of infinite index in $G$, then $G$ is simply connected at $\infty$.
\end{theorem}

\begin{example} 
The intersection of two finitely generated normal subgroups of a finitely generated group need not be finitely generated. Let $F$ and $ F'$, be free groups with generators $x,y$ and $x',y'$ respectively. Let $C$ and $C'$ be the commutator subgroups of $F$ and $F'$ respectively. For each element $c\in C$ let $c'\in C'$ be the element obtained from $c$ by replacing each $x$ by $x'$ and each $y$ by $y'$. Let $N_1$ be the normal closure in $F\times F'$ of $\mathcal C=\{c^{-1}c':c\in C\}$ and 
$$q_1:F\times F'\to (F\times F')/N_1=Q$$ 
be the quotient map. We will show that $q_1(F)\cap q_1(F')=q_1(C)\cong C$. 

Let $N_2$ be the normal closure of the set $\{q_1(x^{-1}x'), q_1(y^{-1}y'\}$ in $Q$ and consider the quotient map
$$q_2:Q\to Q/N_2.$$
Let $N_0$ be the normal closure of $\{x^{-1}x', y^{-1}y'\}$ in $F\times F'$ and 
$$q_0:F\times F'\to (F\times F')/N_0\cong F.$$ 
The map $q_0$ identifies each element of $F$ with the corresponding element of $F'$ in particular $\mathcal C\subset N_0$ and so $N_1\subset N_0$. We have that $q_0=q_2q_1$ (and so  $Q/N_2\cong F$).

$$F\times F' \ {\buildrel q_1\over \to}\  (F\times F')/N_1=Q\ {\buildrel q_2\over \to}\  Q/N_2\cong F.$$
Finally let $N_3=[F\times F',F\times F']$ be the commutator subgroup of $F\times F'$ (so $N_3=[F,F]\times [F',F']$). Then $\mathcal C\subset N_3$, so that $N_1\subset N_3$. Consider the quotient map 
$$q_3:F\times F'\to \mathbb Z^4$$
with kernel $N_3$. 
Since $q_0(=q_2q_1)$ is injective on $F$ and $F'$, $q_1$ is as well. In particular $q_1(C)=q_1(C')\cong C$ is not finitely generated.
 Certainly 
 $$q_1(C)\subset q_1(F)\cap q_1(F').$$ 
 If $d_1\in q_1(F)\cap q_1(F')$, choose $d\in F$ and $d'\in F'$ such that $q_1(d)=d_1=q_1(d')$. Then $q_1(d^{-1}d')=1$ so that $d^{-1}d'\in N_1\subset N_3$ and $q_3(d)=q_3(d')$. But $q_3(d)$ is trivial in the last two factors of $\mathbb Z^4$ and $q_3(d')$ is trivial in the first two factors of $\mathbb Z^4$. This means that $q_3(d)=q_3(d')=0$. The only elements of $F$ in the kernel of $q_3$ are in $C$, so $d\in C$. In particular, $d_1=q_1(d)\in q_1(C)$.
\end{example}


The main result of \cite{CM2} generalizes Theorems  \ref{J} and \ref{M1} in a direction different than these last four results.

\begin{definition}
If $Q$ is a subgroup of a group $G$, then $Q$ is {\it commensurated} in $G$ (written $Q\prec G$) if for each $g\in G$, $gQg^{-1}\cap Q$ has finite index in both $Q$ and $gQg^{-1}$. 
\end{definition}

 Note that if the subgroup $Q$ of $G$ is normal in $G$, then $Q$ is commensurated in $G$. 

\begin{theorem} [\cite{CM2}]\label{MainCM} 
If a finitely generated group $G$ has an infinite, finitely generated, commensurated subgroup $Q$, and $Q$ has infinite index in $G$, then $G$ is 1-ended and semistable at $\infty$. Furthermore, if $G$ and $Q$ are finitely presented and either $Q$ is 1-ended or the pair $(G,Q)$ has one filtered end, then $G$ is simply connected at $\infty$.
\end{theorem}

For $p$ a prime, the  group $SL_n(\mathbb Z[{1\over p}])$ is finitely presented.  Y. Shalom and G. Willis \cite{SW13} show that when $n>2$ the only normal subgroups of this group are either finite or of finite index. But, for $n>2$, the finitely presented 1-ended group, $SL_n(\mathbb Z)$ is commensurated in $SL_n(\mathbb Z[{1\over p}])$ and so Theorem \ref{MainCM} implies:

\begin{theorem} \label{SLn}  
For $n>2$, the group $SL_n(\mathbb Z[{1\over p}])$ is 1-ended and simply connected at $\infty$. 
\end{theorem}

It is an easy exercise to see the group $BS(1,2)$ has no finitely generated normal subgroups that are not either trivial or finite index. 
The subgroup $\langle x\rangle$ is commensurated in the Baumslag-Solitar group $BS(m,n)=\langle x,t  \ |\  t^{-1} x^mt=x^n\rangle$ (see (Example 3.1, \cite{CM13})). This along with Theorem \ref{MainCM} implies the following result, first proved as Theorem 3.4 of \cite{M1}.

\begin{theorem} [(Theorem 3.4, \cite{M1}] \label{BSmn}
The Baumslag-Solitar groups $BS(m,n)=\langle x,t  \ |\  t^{-1} x^mt=x^n\rangle$ are 1-ended and semistable at $\infty$.
\end{theorem}

If $C$ is a subgroup of finite index in $A$ and $B$ (possibly all groups $\infty$-ended) then $C$ is commensurated in $G=A\ast_C B$. Applying Theorem \ref{MainCM}:

 \begin{theorem} [See Theorem \ref{SSDecomp} for a more general result.] \label{FIss} 
 Suppose $G$ is the amalgamated product $A\ast_CB$ where $A$ and $B$ are finitely generated and $C$ has finite index in $A$ and $B$ (but $A\ne C\ne B$). Then $G$ is 1-ended and semistable at $\infty$. If additionally, $C$ is finitely presented and 1-ended, then $G$ is simply connected at $\infty$. 
 \end{theorem} 
 
\begin{example} \label{Examples3}  Three group $\Lambda_i$, $\Lambda_2$ and $\Lambda_3$ are constructed in \cite{Ratt07}. Each group $\Lambda_i$ can be decomposed in two ways as amalgamated products $F_9\ast_{F_{81}} F_9$, such that $F_{81}$ has index 10 in both $F_9$ factors. The group $\Lambda_1$ is a finitely presented, torsion free simple group. The groups $\Lambda_2$ and $\Lambda_3$ are not simple, but $\Lambda_2$ is virtually simple and $\Lambda_3$ has no non-trivial finite quotients.
Theorem \ref{FIss} implies each $\Lambda_i$ is 1-ended and semistable at $\infty$. Each $\Lambda_i$ is a normal subgroup of finite index in another group $\Gamma_i$ and $\Gamma_i$ acts geometrically on $\mathcal T_{10}\times \mathcal T_{10}$, the product of two 10-regular trees. In particular none of the $\Gamma_i$ nor any of the $\Lambda_i$ are simply connected at $\infty$. 
\end{example}

While Lew's theorem improved Theorem \ref{M1} by replacing normality by subnormality, Profio's result was the best attempt  in the last 30 years, to improve the normality hypothesis of Theorem \ref{J} to subnormality. As a corollary of Theorem \ref{MainA}, we obtain a subnormal version of Jackson's Theorem \ref{J}. The semistability part of Theorem \ref{MainA} was proved first and then used in an essential way in the proof of the simply connected at $\infty$ part of Theorem \ref{MainA}. The idea of the simple connectivity at $\infty$ of a finitely generated group, was introduced and used in a fundamental way to prove the second part of Theorem \ref{MainA}.  The author points out that he cannot prove this part of Theorem \ref{MainA}, even in the finitely presented case, without this new concept. 

\noindent Recall that if $Q$ is a commensurated subgroup of $G$ we use the notation $Q\prec G$.

\begin{theorem} [Theorem 1.9, \cite{M6}]  \label{MainA}  Suppose $H$ is a finitely generated infinite subgroup of infinite index in the finitely generated group $G$, and  $H$ is subcommensurated in $G$:
$$H=Q_0\prec Q_1\prec \cdots \prec Q_{k}\prec G$$
Then $G$ is 1-ended and semistable at infinity. If additionally, $H$ is 1-ended and finitely presented and $G$ is finitely generated and recursively presented then $G$ is simply connected at $\infty$.
\end{theorem}

\begin{example} [\cite{M96}]
There are short exact sequences for each $n>0$, of the form:
$$ 1\to H\to (\mathbb Z^n\ast \mathbb Z)\times (\mathbb Z^n\ast \mathbb Z)\to\mathbb Z^n\to 1$$
where $H$ is 1-ended and finitely generated. The group $\mathbb Z^n$ is $(n-2)$-connected at $\infty$, yet $(\mathbb Z^n\ast \mathbb Z)\times (\mathbb Z^n\ast \mathbb Z)$ is not simply connected at $\infty$. 
Note that Theorem \ref{Subset} implies any finitely generated subgroup of $(\mathbb Z^n\ast \mathbb Z)\times (\mathbb Z^n\ast \mathbb Z)$ containing $H$ is either of finite index in $(\mathbb Z^n\ast \mathbb Z)\times (\mathbb Z^n\ast \mathbb Z)$ or is not finitely presented. These examples show that the finitely presented hypothesis on $H$ in Theorems \ref{J} and \ref{MainA}, cannot be easily relaxed.
\end{example}

N. Brady and J. Meier  generalize the above examples.  They produce short exact sequences $1\to K\to G\to \mathbb Z^m\to 1$ with $G$ finitely presented  and $K$  a (non-finitely presented) $FP_\infty$ infinity group  such that $G$ is not simply connected at $\infty$ (see the second set of examples of \S 6, \cite{BrM01}.) 

\begin{example} [Corollary V.1.1, \cite{P90}] \label{Pro1}  If $G_1$ is any finitely presented group, then there exists groups $N\lhd H\lhd G$ with $N$ and $G$ finitely presented and $G/H=G_1$ such that $G$ is not simply connected at $\infty$. 

There are also examples showing that for each infinite finitely presented group $G_1$, there exists groups $G_1<H\lhd G$ where $H$ is a 1-ended finitely generated group, $G$ is finitely presented, $G/H=\mathbb Z$ and $G$ is not simply connected at $\infty$.
\end{example}  

Next we consider results about finitely generated simply connected at $\infty$ groups. 

\begin{theorem} [Theorem 5.1, \cite{M6}] \label{sc} 
Suppose the recursively presented group $G$ is finitely generated and isomorphic to $A\times B$ where $A$ and $B$ are finitely generated infinite groups and $A$ is 1-ended. Then $G$ is simply connected at $\infty$.
\end{theorem}

R. Grigorchuk constructed a finitely generated infinite torsion group $G$ and a finitely presented HNN extension $H$ of $G$ (\cite{GR1} and \cite{GR2}).   The group $G$ contains a subgroup $T$ of finite index in $G$ and $T$ is isomorphic to $T\times T$ (see \cite{dlH00} VIII.C. Theorem 28). Theorem \ref{sc} implies the finitely generated group $G$ is simply connected at $\infty$. The next result implies the finitely presented group $H$ is simply connected at $\infty$.

\begin{theorem} [Theorem 5.2, \cite{M6}] \label{hnn} 
Suppose $H$ is an ascending HNN extension of a 1-ended, finitely generated, semistable at $\infty$, and simply connected at $\infty$ group $G$. Then $H$ is simply connected at $\infty$. 
\end{theorem}

\subsubsection{The Bieri-Stallings Groups}\label{BSgps}

A space $X$ is {\it aspherical} if for all $n\geq 0$ any map of the $n$-sphere $\mathbb S^n$ into $X$ can be extended to the disk $\mathbb D^{n+1}$. A group $G$  is of type $\mathcal F_1$ when it can be finitely generated, $\mathcal F_2$ when it can be finitely
presented, and more generally $\mathcal F_n$ if there is an aspherical complex with fundamental group $G$ (a $K(G,1)$-complex) and finite $n$-skeleton.
A locally compact space $X$ is said to be {\it $r$-connected at $\infty$} if  for each compact set $K\subset  X$ there exists a compact set $L\subset X$ such that $K\subset L$ and any map $\alpha:\mathbb S^d\to X-L$ with $d\leq r$ extends to a map $\tilde \alpha : \mathbb D^{d+1}  \to  X -K$. For example, $X$ is $(-1)$-connected at $\infty$ if and only if it is non-compact, it is $0$-connected at $\infty$ if and only if it has $1$-end and it is 1-connected at $\infty$ if and only if it is 1-ended and simply connected at $\infty$.  Observe that if a group $G$ is not of type $\mathcal F_n$ then any finite complex $X$ with fundamental group $G$ is such that $\pi_{n-1}(X)\ne \{0\}$ and so $\pi_{n-1}(\tilde X)\ne \{0\}$ for $\tilde X$ the universal cover of $X$. In particular, $\tilde X$ is not $(n-1)$-connected at $\infty$ (for any compact $D\subset \tilde X$ there are maps of $\mathbb S^{n-1}\to \tilde X-D$ that cannot be extended to $\mathbb D^{n}$ in $X$).

\begin{theorem} [Theorem 17.1.6, \cite{G}] \label{PHT}  Let $n\geq 2$ and let $Y$
be a locally finite connected CW-complex. If $Y$ is 1-ended and simply connected at $\infty$ (properly 1-connected), then $Y$ is properly $n$-connected if and only if  $Y$ is properly $n$-acyclic
with respect to $\mathbb Z$.
\end{theorem}

A discrete group $\Gamma$ is said to be {\it $r$-connected at $\infty$} if it acts co-compactly, freely, and simplicially on an $r$-connected simplicial complex $X$ which is $r$-connected at $\infty$. If $\Gamma$ acts cocompactly, freely, and simplicially on another $r$-connected simplicial complex $Y$, then $Y$ is also $r$-connected at $\infty$. If $\Gamma$ is not type $\mathcal F_n$ then $\Gamma$ is not $(n-1)$-connected at $\infty$. If $\Gamma_1$ is a subgroup of finite index in the group $\Gamma$ then $\Gamma_1$ is $r$-connected at $\infty$ if and only if $\Gamma$ is $r$-connected at $\infty$.

Let $F_2$ be the free group on two generators.  Consider
$$F_2\times F_2\equiv \langle a_1,a_2, b_1,b_2 \ |\  [a_i,b_j]=1\hbox{ for all } i,j\in \{1,2\}\rangle.$$ 
Stallings \cite{Stall63} showed that the kernel $K$ of the homomorphism $$f:F_2\times F_2\to \mathbb Z\hbox{ where } f(a_1)=f(a_2)=f(b_1)=f(b_2)=1,$$
is finitely generated but not finitely presented. Stallings' also proved that a certain HNN-extension of $K$ (now known as Stallings' group) is $\mathcal F_2$ but not $\mathcal F_3$.

Let $B_n$ ($n \geq 2$) denote the kernel of the homomorphism 
$$F_2 \times  \cdots (n) \cdots \times F_2 \rightarrow {\mathbb Z}$$
 sending all generators to the generator $1$. (Here $K$ is isomorphic to $B_2$ and Stallings' group is isomorphic to $B_3$.)  Bieri \cite{Bieri81} proved that that $B_n$ is $\mathcal F_{n-1}$ but not $\mathcal F_n$. These groups are often call the {\it Bieri-Stallings} groups. 
 We have short exact sequences of the form 
 $$1 \rightarrow B_{n-1} \rightarrow B_n \rightarrow F_2 \rightarrow 1.$$ 
\begin{theorem} [Theorem 17.3.6,\cite{G}] Let $1\to N\to G\to Q\to 1$  be a short exact sequence of infinite
groups of type $\mathcal F_n$, and let $R$ be a PID. Working with respect to $R$, let $N$ be
$s$-acyclic at infinity and let $Q$ be $t$-acyclic at infinity where $s\leq n-1$ and
$t \leq n-1$. Then $G$ is $u$-acyclic at infinity where $u=min\{s + t + 2, n-1\}$. If
$s\geq 0$ or $t\geq 0$  (i.e., if $N$ or $Q$ has one end) then $G$ is $u$-connected at infinity.
 \end{theorem}
 \medskip
 
 This theorem implies that $B_n$ is $(n-3)$-connected at infinity, for $n \geq 4$
 (since $B_{n-1}$ is type $\mathcal F_{n-2}$ and not $\mathcal F_{n-1}$, we cannot conclude that $B_n$ is $(n-2)$-connected at infinity.) 
  
 \begin{theorem} [Theorem 1.1,\cite{MM24}]\label{MM24}  
Stallings' group $B_3$ is simply connected (1-connected) at infinity.
\end{theorem} 

\begin{lemma}  [Lemma 5.2,\cite{MM24}]  \label{OneE} The kernel $B_2$ of the map from $F_2\times F_2\to \mathbb Z$ (sending all generators to the generator $1$) is 1-ended. 
\end{lemma}
These last two results are evidence (and the first two cases) of:

\begin{conjecture}[Conjecture 1.2, \cite{MM24}] For $n\geq 2$ the group $B_n$ is $(n-2)$-connected at infinity.
\end{conjecture}

\subsubsection{Out($F_n$)} \label{OUT} 
If $n\geq 3$ and $Out(F_n)$ is the outer automorphism group of the free group on $n$ generators, then there is an exact sequence of finitely generated groups:
$$1\to H\to Out (F_n)\to Gl(\mathbb Z,n)\to 1$$

 For $n\geq 3$, the group $H$ is 1-ended. For $n\geq 4$, it is unknown if $H$ is finitely presented, so Theorem \ref{J} cannot be applied here. S. Kristi\'c and J. McCool \cite{kMc97} proved that for $n=3$, $H$ is not finitely presented. K. Vogtmann \cite{Vo95} proved that for $n\geq 5$, $Out(F_n)$ is simply connected at $\infty$. 
J. Rickert \cite{Ric00} proved that $Out(F_4)$ is simply connected at $\infty$. A theorem of M. Bestvina and M. Feighn implies that for $n\geq 3$, $Out(F_n)$ is simply connected at $\infty$. In fact, they prove much more.

\begin{theorem} 
[Theorem 1.1, \cite{BFe00}] \label{BF1}  For $n\geq 2$, the group $Out(F_n)$ is $(2n-5)$-connected at infinity.
\end{theorem}

A group $\Gamma$ that acts freely and cocompactly on a contractible cell complex
$X$ (by permuting cells) is a {\it duality group of dimension} $n$ if $H^i (\Gamma; \mathbb Z \Gamma)$ vanishes when $i\ne n$ and is torsion-free (as an abelian group) when $i = n$ (see R. Bieri and B. Eckmann's article \cite{BE73}). Theorem \ref{GME2} (2), (with $\Gamma$ and $X$ as above) implies the group $\Gamma$ is an $n$-dimensional duality group if and only if  (the reduced homology inverse system) $\bar H_r(\varepsilon X)$ is pro-trivial for all $r\leq n-2$. If $C_i$ is a cofinal sequence of compact sets in $X$, then $H_r(\varepsilon X)$ is the inverse system:
$$H_r(X-C_1)\leftarrow H_r(X-C_2) \leftarrow\cdots$$

\medskip 

\noindent ($\ast$) {\it If $G$ is a duality group of dimension $n\geq 3$ and $G$ is simply connected at $\infty$, then the Proper Hurewicz Theorem (\ref{PHT}) implies that $G$ is $k$-connected at $\infty$ for $k\leq n-2$.} 

\begin{theorem} 
[Theorem 1.4, \cite{BFe00}] \label{BF2}  For $n\geq 2$, $Out(F_n)$ is a virtual duality group of dimension $2n -3$.
\end{theorem}

A. Borel and J. P. Serre have results analogous to Theorems \ref{BF1} and \ref{BF2} (with an analogous approach) for Arithmetic groups \cite{BS73}

R. Lyman determines the number of ends of the outer automorphism group of $F$ a free product of finite and cyclic groups and shows $Out(F)$ is semistable at each end.  The group $F$ contains a free subgroup of finite index. The group $Out(F)$ is of type $FP_{\infty}$ and is virtually torsion free but need not  be a virtual duality group.  

\begin{theorem} [Theorem B, \cite{Lym25}]
Let $F = A_1\ast \cdots \ast A_n\ast F_k$ be a free product of the finite groups $A_i$ with a free group
of rank $k$. The invariants $n$ and $k$ determine the number of ends of $Out(F )$ in the following
way:
$$n \leq 1, k \leq  1 \hbox{ or } (n, k) = (2, 0) \ \ \  Out(F ) \hbox {is finite, i.e. has zero ends}.$$
$$\hskip -.1in (n, k) = (3, 0), (2, 1), (0, 2) \ \ \ \ \ \ \ \ Out(F ) \hbox{ has infinitely many ends.}$$
$$\hskip -1in \hbox{ otherwise }\hskip 1.6in Out(F ) \hbox{ has one end.}$$
\end{theorem}

\begin{theorem} [Theorem C, Lyman] Let $F = A_1\ast \cdots \ast A_n\ast F_k$ be a free product of the finite groups $A_i$ with a free group of rank $k$. The group $Out(F)$ is semistable at each end.
\end{theorem}

\subsubsection{Mapping Class Groups of Surfaces}\label{MCG}

Our goal in this section is to understand the connectivity at infinity of mapping class groups of surfaces possibly with punctures and/or boundary components. Theorem \ref{MCGSCINF} collects the various results in this section and provides a list that breaks down mapping class groups into the finite, virtually free, virtual extensions of two finitely generated infinite free groups (and hence 1-ended and semistable, but not simply connected at infinity) and (for $n\geq 1$) $n$-connected at infinity groups. See \S \ref{BSgps} for the basics of $n$-connected at infinity groups.  All of these groups are semistable at infinity and all but finitely many are 1-ended and simply connected at infinity. These groups are virtual duality groups of known dimension. Combining our results with the Proper Hurewicz Theorem determines the higher connectivity at infinity of mapping class groups. Let $F=F_{g,r}^s$ be an orientable surface of genus $g$ with $s$ punctures and $r$ boundary components. The {\it mapping class group} $\Gamma_{g,r}^s$  is the group of isotopy classes of orientation preserving diffeomorphisms of $F$ which preserve the punctures of $F$ individually and restrict to the identity on $\partial F$. Here the isotopies keep $\partial F$ fixed point-wise. If $S$ is an orientable surface the notation $Mod(S)$ is also common for the mapping class group of $S$.  
 It follows from geometry results in the 1960's (Baily \cite{B60} and Deligne-Mumford \cite{DM69}) that $\Gamma_{g,0}^0$ is finitely presented (see \cite{FM12}, P128).
Also see \cite{McC75} and Theorem 4.3.D of N. Ivanov's manuscript - Mapping Class Groups (December 21, 1998) 

\noindent https://s3.amazonaws.com/nikolaivivanov/Ivanov1999MappingClassGroups.pdf. 
B. Wajnryb \cite{Waj96} proved that mapping class groups of compact surfaces with 0 or 1 boundary component are in fact generated by two elements. 


First we consider a few special cases: If $S$ is a torus or a torus with 1 puncture then its mapping class group is isomorphic to $SL_2(\mathbb Z)$. This group is an amalgamated product of two finite groups and hence contains a free subgroup of finite index. Since free groups are simply connected at infinity, $\Gamma_{1,0}^0$ and $\Gamma_{1,0}^1$ are simply connected at infinity. If $S$ is a sphere with 4 punctures then $\Gamma_{0,0}^4$ has a subgroup of finite index isomorphic to $PSL_2(\mathbb Z)$ and again is virtually free. If $S$ is a sphere with $\leq 3$ punctures, then $Mod(S)$ is finite. Finite groups are vacuously simply connected at infinity. 

We combine Theorem \ref{J} with two well known exact sequences of groups to show most mapping class groups are simply connected at infinity.
The Birman exact sequence relates the mapping class group of surfaces with the same genus and boundary but a different number of punctures. It was proven by Joan Birman in 1969 \cite{BES69}. Also see Equation 4.2 of \cite{Har86}.

 \begin{theorem} \label{BES} There is an exact sequence
    $$ 1\to \pi_1 ( F_{g,r}^{s-1} )\to \Gamma_{g,r}^s \to  \Gamma_{g,r}^{s-1} \to 1.$$ 
In the case where $s>1$ the mapping class group $\Gamma_{g,r}^s$ must be replaced by a finite-index subgroup. 
One must exclude the cases $(g,r,s)\in \{(1,0,1), (0,0,k)\}$ where $k<4$ (considered above) as well.
\end{theorem} 

Equation 4.3 of \cite{Har86} gives the well known short exact sequence:
\begin{theorem} \label{HES}
If $r>0$ then the following sequence is exact:
  $$1\to \mathbb Z\to \Gamma_{g,r}^s\to \Gamma_{g,r-1}^{s+1}\to 1.$$
\end{theorem}
Combining these two exact sequences with the fact that $\Gamma_{g,0}^0$ is finitely presented shows that all $\Gamma_{g,s}^r$ are finitely presented (the extension of a finitely presented group by a finitely presented group is finitely presented). 
Next observe that this last exact sequence implies that $\Gamma_{0,1}^1$ is virtually $\mathbb Z$. 

\begin{theorem} \label{punbd} 
If $g>1$ and $r+s\geq 1$ then the group $\Gamma_{g,r}^s$ is 1-ended and simply connected at infinity. 
\end{theorem} 
\begin{proof}
The proof is by induction on $r+s$. First consider $r+s=1$.
Theorem \ref{BES} applied to $\bar \Gamma_{g,0}^1$ (a subgroup of finite index in $\Gamma_{g,0}^1$) gives the exact sequence:
$$1\to \pi_1(F_{g,0}^0)\to  \bar \Gamma_{g,0}^1\to \Gamma _{g,0}^0\to 1$$
The group $\pi_1(F_{g,0}^0)$ is 1-ended and
Theorem \ref{J} implies that $\bar \Gamma_{g,0}^1$ (and hence $ \Gamma_{g,0}^1$) 
is 1-ended and simply connected at infinity. Theorem \ref{HES} gives the exact sequence
$$1\to \mathbb Z\to \Gamma_{g,1}^0\to \Gamma_{g,0}^1\to 1.$$
Theorem \ref{J} implies $\Gamma_{g,1}^0$ is 1-ended and simply connected at infinity. Hence if $r+s=1$ then $\Gamma_{g,r}^s$ is 1-ended and simply connected at infinity. Inductively assume that $n\geq 1$ and if $r+s=n$, then $\Gamma_{g,r}^s$ is 1-ended and simply connected at infinity. If $r+s=n$. then Theorem \ref{BES} implies the sequence:
$$1\to \pi_1(F_{g,r}^{s})\to \bar \Gamma_{g,r}^{s+1}\to \Gamma _{g,r}^s\to 1$$
is exact (where $\bar \Gamma_{g,r}^{s+1}$ has finite index in $\Gamma_{g,r}^{s+1}$). Since $\Gamma_{g,r}^s$ is 1-ended, Theorem \ref{J} implies $\bar \Gamma_{g,r}^{s+1}$ (and hence $\Gamma_{g,r}^{s+1}$) is 1-ended and simply connected at infinity. Theorem \ref{HES}
implies the sequence 
$$1\to \mathbb Z\to \Gamma_{g, r+1}^s\to \Gamma_{g,r}^{s+1}\to 1$$ 
is exact. Theorem \ref{J} implies $\Gamma_{g, r+1}^s$ is 1-ended and simply connected at infinity. This completes the induction step.
\end{proof}

\begin{theorem} \label{punbd2} 
If $g=1$ then for $(r,s)\in \{(0,0), (0,1)\}$ the group $\Gamma_{1,r}^s$ is virtually free (and hence simply connected at infinity). Each of the groups $\Gamma_{1,1}^0$ and $\Gamma_{1,0}^2$ is virtually an extension of two infinite finitely generated free groups. In particular, they are 1-ended and semistable at infinity, but not simply connected at infinity. Otherwise, $\Gamma_{1,r}^s$ is 1-ended and simply connected at infinity. 
\end{theorem} 
\begin{proof}
As noted earlier, $\Gamma_{1,0}^1$ is isomorphic to $SL_2(\mathbb Z)$ and so is virtually free. Theorem \ref{HES} gives the exact sequence:
$$1\to \mathbb Z \to \Gamma_{1,1}^0\to \Gamma_{1,0}^1\to 1.$$ 
Theorem \ref{cohen} implies that $\Gamma_{1,1}^0$ is 1-ended and Theorem \ref{M1} implies $\Gamma_{1,1}^0$ is semistable at infinity. Theorem \ref{JConv} implies that $\Gamma_{1,1}^0$ is proper 2-equivalent to $\mathbb Z \times \Gamma_{1,0}^1$ which is not simply connected at infinity by Theorem \ref{NotSc1}. Theorem \ref{2-equiv} implies that $ \Gamma_{1,1}^0$ is not simply connected at infinity. Theorem \ref {BES} gives the exact sequence (here $\bar \Gamma_{1,0}^2$ has finite index in $\Gamma_{1,0}^2$): 
$$1\to \pi_1(F_{1,0}^1) \to \bar \Gamma_{1,0}^2\to \Gamma_{1,0}^1\to 1.$$
Again, Theorems \ref{cohen}, \ref{M1}, \ref{JConv}, \ref{NotSc1} and \ref{2-equiv} imply $\Gamma _{1,0}^2$ is 1-ended and semistable at infinity, but not simply connected at infinity. Theorem \ref{HES} gives the exact sequence:
$$1\to \mathbb Z\to \Gamma_{1,1}^1\to \Gamma_{1,0}^2\to 1.$$
Theorem \ref{J} implies that $\Gamma_{1,1}^1$ is 1-ended and simply connected at infinity. Theorem \ref{HES} gives the exact sequence:
$$1\to \mathbb Z\to \bar \Gamma_{1,2}^0\to \Gamma_{1,1}^1\to 1.$$
Theorem \ref{J} implies $\Gamma_{1,2}^0$ is 1-ended and simply connected at infinity. All of the groups $\Gamma_{1,r}^s$ with $r+s=2$ are 1-ended. An induction argument as in Theorem \ref{punbd} shows the remaining groups are 1-ended and simply connected at infinity. 
\end{proof}


The following argument was provided by Dan Margalit. 
 
 \begin{theorem} \label{DM1} 
 The mapping class group $Mod(S_2)$ (of a closed surface of genus 2) is 1-ended and simply connected at infinity. 
 \end{theorem}
 \begin{proof}
 Let $S_{0,n}$ be the sphere with $n$ marked points.
Birman-Hilden theory gives a map $Mod(S_2) \to Mod(S_{0,6})$ with kernel $\mathbb Z/2$ (see page 257 of \cite{FM12}).
Next note that $PMod(S_{0,6})$ is an index 6 subgroup of $Mod(S_{0,6}$).  Simple connectivity at infinity is invariant under subgroups of finite index and quotients by finite normal subgroups.

By the Birman exact sequence (Theorem \ref{BES}): 

$$1 \to F_4 \to PMod(S_{0,6}) \to PMod(S_{0,5}) \to 1.$$
 The group $PMod(S_{0,5})$ is a semidirect product of $F_2$ and $F_3$ and hence is 1-ended.  Theorem \ref{J} implies that 
  $PMod(S_{0,6})$ (and hence $Mod(S_2)$) is simply connected at infinity
 \end{proof} 
 
 
For $g>4$, $\Gamma_{g,0}^s$ exhibits a property significantly stronger than simply connected at infinity. If $x_0$ is a base point in a metric space $X$ and $n>0$, then a set $A$ is {\it $n$-avoidant} if $A$ does not intersect the $n$-ball about $x_0$. A {\it filling} of a loop is basically a homotopy killing the loop. 
  
 \begin{theorem} [Theorem 3.3, \cite{ABDDY12}] \label{FillingL}(Filling loops at infinity in the mapping class group). Suppose the closed surface $S$ has
genus at least 5 and any number of punctures, and let $X$ be the Cayley 2-complex
of $Mod(S)$. There is a constant $c > 0$ such that for any $n$, any $n$-avoidant loop of
length $l$ has an $n\over c$-avoidant filling of area $\leq  cn^2l^2$.
\end{theorem}
Note that if $x_0$ is the base point of $X$ and $\alpha$ avoids the $nc$-ball about $x_0$ then the filling given by this theorem defines a homotopy of $\alpha$ avoiding the $n$-ball about $x_0$. In particular, $X$ is simply connected at infinity and all mapping class groups considered in Theorem \ref{FillingL} are simply connected at infinity. 

The remaining two cases for closed surfaces are direct consequences of the next result.

\begin{theorem}[Theorem 1.1, \cite{MCG26}] \label{MCGPr} Suppose the group $G$ has a finite presentation $\langle S:R\rangle$ and $G$ contains a free abelian subgroup $A$ of rank $\geq 3$ with free generating set $T$ satisfying (1) and (2). Then $G$ simply connected at infinity.  

\noindent  (1) For $r\in R$ there is $t_r\in T$ such that $t_r$ commutes with each letter of $r$.

\noindent  (2) If $s\in S$ is a letter of $r\in R$ then some $z\in T-\{t_r\}$ commutes with $s$. 
\end{theorem}

Let $S=\{s_1,\ldots, s_n\}$. First observe that if $s_i$ does not appear in any relation of $R$, then $G$ splits as a free product $\mathbb Z_{s_i}\ast \langle S-\{s_i\}\rangle$. On the other hand, if for each $s\in S$ there is $r_s\in R$ such that $s$ is one of the letters of $r_s$, then $G$ is one-ended. (Apply Theorem \ref{combE} to $\langle s_i\rangle$ and $\langle T\rangle$ for each $i$ and then inductively to $\langle \{s_1,\ldots, s_i\}\cup T\rangle \cup \langle \{s_{i+1}\}\cup T\rangle$). 

 \begin{theorem} [Corollary 1.3, \cite{MCG26}] \label{MCGM26}
 The groups $\Gamma_{3,0}^0$ and $\Gamma_{4,0}^0$ are 1-ended and simply connected at infinity.
 \end{theorem}

Recall that a group $\Gamma$ that acts freely and cocompactly on a contractible cell complex
$X$ (by permuting cells) is a {\it duality group of dimension} $n$ if $H^i (\Gamma; \mathbb Z \Gamma)$ vanishes when $i\ne n$ and is torsion-free (as an abelian group) when $i = n$ (see R. Bieri and B. Eckmann's article \cite{BE73}). A group $\Gamma$ is a {\it virtual duality} group if it contains a duality group as a subgroup of finite index.

 \begin{theorem} [Theorem 4.1,\cite{Har86}] \label{MCGVDG}  
For $2g + r +s > 2$, the mapping class group $\Gamma_{g,r}^s$ is a virtual duality group  of dimension $d(g, r, s)$, where $d(g, 0,0)=4g-5$, $d(g,r,s)=4g+2r+s-4$, $g>0$ and $r+s>0$, and $d(0,r,s)=2r + s-3$. 
 \end{theorem}
 Virtual duality groups are finitely presented. They are 1-ended if they have dimension $>1$ (see Theorem \ref{HEnds}). Next, we consider mapping class groups that satisfy the duality condition $2g+s+r>2$ but are not 1-ended. Note that if $g>0$ and $2g+r+s=2$ then $g=1$, $r=s=0$ and the mapping class group of the torus, $\Gamma_{1,0}^0=GL_2(\mathbb Z)$ is (virtually free and) simply connected at infinity, but not 1-ended. All other cases for $g=1$ are already considered.

 If $g=0$ and $r+s>2$,   the only pairs $(r,s)$ not considered are in $\{(0,3), (0,4), (1,2)\}$. As mentioned earlier $\Gamma_{0,4}^0$ is virtually free and for $k\leq 3$, $\Gamma_{0,0}^k$  is finite. Theorem \ref{MCGVDG} implies that $\Gamma_{0,1}^2$ has virtual cohomological dimension 1. Stallings famous theorem \cite{St68} implies  $\Gamma_{0,1}^2$ is virtually free. 

Next we turn our attention to the the remaining cases with $g=0$. 

\begin{theorem}\label{MCG7}
Suppose $g=0$ and $r+s>2$ (so that $\Gamma_{0,r}^s$ is a duality group - see Theorem \ref{MCGVDG}). The group $\Gamma_{0,r}^s$ is 1-ended if and only if one of the following holds: 
$$(r=0 \hbox{ and } s \geq 5),\  (r=1 \hbox{ and } s\geq 3),\   (r=2 \hbox{ and } s\geq 1), \  (r>2).$$
The group $\Gamma_{0,s}^r$ is 1-ended and simply connected at infinity if and only if one of the following holds:
$$(r=0 \hbox{ and } s>5),\ (r=1 \hbox{ and } s>3),\  (r=2 \hbox{ and } s>1),\  (r>2).$$
\end{theorem}
\begin{proof}
If $G$ is a duality group of dimension $n$ then $H^i(G,\mathbb ZG)=0$ for $i<n$ and $H^n(G,\mathbb ZG)$ is non-trivial (and torsion free as an abelian group). We have $d(0,r,s)=2r+s-3$ and $d(0,r,s)=2$ precisely when:
$$(r,s)\in \{(0,5), (1,3), (2,1)\}.$$
Theorem \ref{HEnds} implies that $\Gamma _{0,r}^s$ is 1-ended if and only if $d(0,r,s)\geq 2$. 
The first statement of the theorem follows. 

 If $d(0,r,s)=2$ then $H^2(\Gamma_{0,r}^s, \mathbb Z\Gamma_{0,r}^s)\ne 0$ and Corollary \ref{GM2} implies that $\Gamma_{0,r}^s$ is not simply connected at infinity. It remains to show that under the second set of conditions, $\Gamma_{0,r}^s$ is simply connected at infinity. 

If $r=0$ and $s>5$, consider the exact sequence given by Theorem \ref{BES}
$$1\to \pi_1(\Gamma_{0,0}^{s-1})\to \bar \Gamma_{0,0}^{s}\to \Gamma _{0,0}^{s-1}\to 1.$$
The last term is $\Gamma_{0,0}^{s-1}$ (for $s-1\geq 5$), a 1-ended group. Theorem \ref{J} implies $\bar\Gamma_{0,0}^s$ (and hence $\Gamma_{0,0}^s$) is simply connected at infinity. 

In the remaining cases, consider the exact sequence given by Theorem \ref{HES}:
$$1\to \mathbb Z\to \Gamma_{0,r}^s\to \Gamma_{0,r-1}^{s+1}\to 1.$$

If $r=1$ and $s>3$, then the last term is $\Gamma_{0,0}^k$ for $k\geq 5$, a 1-ended group. Theorem \ref{J} implies $\Gamma_{0,1}^s$ is simply connected at infinity. 

If $r=2$ and $s>1$ then the last term is $\Gamma_{0,1}^k$ for $k\geq 3$, a 1-ended group. Theorem \ref{J} implies $\Gamma_{0,r}^s$ is simply connected at infinity. 

If $r>2$, then the last term is $\Gamma_{0,r-1}^{s+1}$ for $r-1\geq 2$, a 1-ended group. Theorem \ref{J} implies $\Gamma_{0,r}^s$ is simply connected at infinity. 
\end{proof}

Now we turn our attention to higher connectivity at infinity for mapping class groups. 

\medskip

\noindent ($\ast$) {\it If $G$ is a simply connected at infinity duality group of dimension $n\geq 3$, then the Proper Hurewicz Theorem (\ref{PHT}) implies $G$ is $(n-2)$-connected at $\infty$.} 
  
\begin{theorem} \label{MCGSC}  If $g>1$ then the group $\Gamma_{g,r}^s$ is $d(g,r,s)-2$-connected at infinity.  
 \end{theorem} 
 \begin{proof} 
 Theorem \ref{punbd} states that if $g>1$ and $r+s\geq 1$ then $\Gamma_{g, r}^s$ is 1-ended and simply connected at infinity. Theorem \ref{DM1} implies that $\Gamma_{2,0}^0$ is simply connected at infinity. Theorem \ref{MCGM26} implies that $\Gamma_{3,0}^0$ and $\Gamma_{4,0}^0$ are simply connected at infinity. 
 If $g>4$ then Theorem \ref{FillingL} implies that $\Gamma_{g,0}^0$ is simply connected at infinity. Since $g>1$, we have $2g+r+s>2$. Theorem \ref{MCGVDG} implies that for $g>1$, the group $\Gamma_{g,r}^s$ is a virtual duality group of dimension $\geq 3$ (in particular, these groups are 1-ended). 
 Theorem \ref{MCGVDG} and ($\ast$) combine to finish the proof.
 \end{proof}
 
 Collecting results:
 
 \begin{theorem}\label{MCGSCINF}
 Let $\Gamma_{g,r}^s$ be the mapping class group of the surface of genus $g$ with $s$-punctures and $r$-boundary components. 
 
1)  $\Gamma_{g,r}^s$ is finite if and only if $g=0$, $r=0$ and $s\leq 3$,  or $g=0$, $r=1$ and $s=0$.

2) $\Gamma_{g,r}^s$ is infinite and virtually free (hence simply connected at infinity) if and only if $(g,r,s) \in \{ (0,0,4), (0,1,1),  (0,1,2), (0,2,0), (1,0,0), (1,0,1) \}$. 

3) $\Gamma_{g,r}^s$ is virtually an extension of two non-trivial finitely generated free groups  if and only if $(g,r,s) \in \{(0,0,5), (0,1,3), (0,2,1), (1,1,0), (1,0,2) \}$. In this case, $\Gamma_{g,r}^s$ is 1-ended and semistable at infinity, but not simply connected at infinity. 

4) Otherwise, $\Gamma_{g,r}^s$ is $d(g,r,s)-2$-connected at infinity. In particular, all remaining mapping class groups  are 1-ended and simply connected at infinity. 
 \end{theorem}
 \begin{proof}
 Part 1) is well known. 
 
 For part 2) we want to show each group described there is virtually free and any group not described in part 1) or 2) is 1-ended. Theorem \ref{punbd} implies if $(g,r,s)\in \{1,0,0), (1,0,1)\}$ then $\Gamma_{g,r}^s$ is virtually free. If $(g,r,s)\in \{1,0,0), (1,0,1), (0,2,0)\}$ then Theorem \ref{MCGVDG} implies $\Gamma_{g,r}^s$ is a duality group of dimension 1. Stallings Theorem implies $\Gamma_{g,r}^s$ is virtually free. The exact sequence of Theorem \ref{HES} implies that $\Gamma_{0,1}^1$ is virtually free. Theorem \ref{MCG7} implies that $\Gamma_{0,0}^4$ and $\Gamma_{0,1}^2$ are not 1-ended. Theorem \ref{MCGVDG} implies that they are duality groups of dimension 1. Stallings theorem implies that they are virtually free. To see that the groups not described in parts 1) and 2) are 1-ended, it is enough to show they are duality groups of dimension $\geq 2$. We list the groups not in parts 1) and 2). If $(g,r,s)=(0,0,s)$ then $s>4$ and the dimension is $s-3\geq 2$. If $(g,r,s)=(0,1,s)$ then $s\geq 3$ and the dimension is $2+s-3\geq 2$.
 If $(g,r,s)=(0,2,s)$ then $s\geq 1$ and the dimension is $2r+s-3\geq 2$. 
 If $(g,r,s)=(0,r,s)$ for $r>2$ then the dimension is $2r+s-3\geq 3$. This takes care of all cases with $g=0$. 
 If $(g,r,s)=(1,0,k)$ then $k\geq 2$ and the dimension is $4+k-4\geq 2$. If $(g,r,s)=(1,r,s)$ and $r\geq 1$ then the dimension is $4+2r+s-3\geq 3$. This takes care of all cases with $g=1$. If $g\geq 2$ then the dimension is either $4g-5\geq 3$ or $4g+2r+s-4\geq 4$. Part 2) is complete. 
 
 Next consider part 3). Theorem \ref{BES}  and part 2) imply that if 
 $$(g,r,s)\in \{(0,0,5), ((0,1,3)), (0,2,1)\}$$ 
 then $\Gamma_{r,s}^s$ is virtually an extension of two finitely generated infinite free groups. Theorem \ref{punbd2} implies that if 
 $$(g,r,s)\in \{(1,1,0), (1,0,2)\}$$ 
 then $\Gamma_{r,s}^s$ is virtually an extension of two finitely generated infinite free groups. If $G$ is virtually an extension of two finitely generated infinite free groups, then $G$ is 1-ended by Theorem \ref{cohen}, semistable at infinity by Theorem \ref{M1}, proper 2-equivalent to the direct product of two finitely generated infinite free groups by Theorem \ref{JConv}, and not simply connected at infinity by Theorems \ref{NotSc1} and \ref{2-equiv}. 

Finally, $(\ast)$ implies that it suffices to show that the groups not listed in part 1), 2) and 3) or of the form $\Gamma_{g,0}^0$ for $g\in \{3,4\}$, are duality groups of dimension $\geq 3$ and simply connected at infinity. It is straight forward to apply Theorem \ref{MCGVDG} to see that all groups considered here satisfy the duality condition and have dimension $\geq 3$.
 Suppose $(g,r,s)$ is not listed in parts 1), 2) and 3). If $(g,r,s)=(0,0,s)$ then $s>5$. If $(g,r,s)=(0,1,s)$ then $s>3$. If $(g,r,s)=(0,2,s)$ then $s>1$. In each of these three cases, Theorem \ref{MCG7} implies that $\Gamma_{g,r}^s$ is simply connected at infinity. Also if $(g,r,s)=(0,r,s)$ and $r\geq 3$ then $\Gamma_{0,r}^s$ is not listed in the first three parts and Theorem \ref{MCG7} implies $\Gamma_{g,r}^s$ is simply connected at infinity. We may assume $g>0$. If $(g,r,s)=(1,0,s)$, then $s>2$. If $(g,r,s)=(1,1,s)$ then $s>0$. In both of these cases, Theorem \ref{punbd2} implies $\Gamma_{g,r}^s$ is simply connected at infinity. Any other triple not considered in the first three parts of the form $(1,r,s)$ is such that $r>1$ and Theorem \ref{punbd2} implies $\Gamma_{g,r}^s$ is simply connected at infinity. Theorem \ref{MCGSC} implies that if $g>1$ and $r+s\geq 1$ then the group $\Gamma_{g,r}^s$ is simply connected at infinity. Theorems \ref{DM1} and \ref{FillingL} imply that if $g=2$ or $g>4$ then $\Gamma_{g,0}^0$ is simply connected at infinity. 
 \end{proof}
 
\subsubsection{Gl($\mathbb Z,n)$}\label{GLNZ}
 We do not have a firm reference for the next result, although several people have outlined proofs. 
 
\begin{theorem} [Paragraph 2, \S 1, \cite{Vo95})] \label{GLNZSC}
For $n\geq 3$, $GL_n(\mathbb Z)$ is 1-ended and simply connected at $\infty$. 
\end{theorem}

The group $SL_n(\mathbb Z)$ is an index 2 subgroup of $GL_n(\mathbb Z)$. Symmetric space is $SL_n(\mathbb R)/SO(n)$ and $SL_n(\mathbb Z)$ is a virtual duality group with virtual cohomological dimension $d=vcd(SL_n(\mathbb Z))=dim(symm space)-(n-1)$. Since $SL_n(\mathbb Z)$ is simply connected at infinity, the Proper Hurewicz Theorem \ref{PHT} implies:

\begin{theorem}
For $n\geq 3$, $GL_n(\mathbb Z)$ (equivalently $SL_n(\mathbb Z)$) is 1-ended and $(d-2)$-connected at $\infty$ where $d=vcd(SL_n(\mathbb Z))$.
\end{theorem}


\subsubsection{Solvable and Metanilpotent Groups}\label{SolMet}

The most general semistability result of this section is Theorem \ref{metanil} where all metanilpotent groups are shown to be semistable at $\infty$. 
A metanilpotent group is nilpotent by nilpotent. The class of virtually metanilpotent groups includes all polycyclic groups, all finitely generated linear (over an arbitrary field) solvable groups (by the Kolchin-Lie-Malsev Theorem), all metabelian groups, and Kharlampovich's examples of finitely presented solvable groups with undecidable word problem. In fact, it covers most known results of finitely presented solvable groups. Only recently examples of non-virtually-metanilpotent, finitely presented, solvable groups were found (see Solutions, Problem 16.35 in https://arxiv.org/pdf/1401.0300.pdf). 

\begin{example} The following are examples given in: {\it Unsolved Problems in Group Theory, The Kourovka Notebook, No. 20.}  Here, $p$ is any prime. The group 
$$G = \langle a, b, c, d \ |\  b^a = b^p, c^a = c^p, d^a = d, c^b = c, d^c = dd^b, [d, d^b] = 1, d^p = 1\rangle$$ 
is soluble of derived length 3, but is not metanilpotent-by-finite. 

It is straightforward to show these groups are semistable at $\infty$ (in fact, it may be possible to use Theorem \ref{MT87} to show that these groups are simply connected at $\infty$). Note that killing the normal closure of $\{d,c\}$ gives a retraction of $G$ to $\langle a,b \ |\  b^a=b^p\rangle$. Killing the normal closure of $\{b,d\}$ gives a retraction of $G$ to $\langle a,c \ | \  c^a=c^p\rangle$. Theorem \ref{BSmn} implies that these quotient groups are 1-ended and semistable at $\infty$ and so the subgroups of $G$ with generators $\{a,b\}$ and $\{a,c\}$ are 1-ended and semistable at $\infty$. Theorem \ref{MComb} implies the subgroup of $G$ generated by $\{a,b,c\}$ is 1-ended and semistable at $\infty$. The subgroup of $G$ generated by $\{a,d\}$ is 2-ended and so Theorem \ref{MComb} implies $G$ is semistable at $\infty$. 
\end{example}

\begin{theorem} [(Theorem 6, \cite{M4}] \label{solv1} 
If $G$ is an infinite, finitely generated solvable group with commutator (derived) series 
$$G=G^{(0)}\triangleright G^{(1)}\triangleright \cdots\triangleright G^{(n)}=1$$
and no $G^{(i)}$ is locally finite and infinite, then $G$ is either 2-ended or 1-ended and semistable at $\infty$. 
\end{theorem}
Theorems \ref{MT87} and \ref{metanil} generalize this result.

The Lamplighter group 
$$L=\langle x, y  \ |\  y^2=1, [y,x^{-k}yx^k]=1 \hbox{ for all } k>0\rangle$$
is a finitely generated (not finitely presented) metabelian group that is not semistable at infinity (see Theorem \ref{LNss}).
\begin{theorem} [Main Theorem, \cite{M87a}] \label{MT87} 
Let $G$ be a finitely presented solvable group with commutator (derived) series 
$$G=G^{(0)}\triangleright G^{(1)}\triangleright \cdots\triangleright G^{(n)}\triangleright G^{(n+1)}=1$$
If $G^{(n)}$ contains an element of infinite order, then either $G$ is simply connected at $\infty$ or $G$ contains a normal subgroup $N$ of finite index, and $N$ contains a finite normal subgroup $F$ such that $N/F$ is isomorphic to one of the groups $\langle x,y\ |\ x^{-1}yx=y^p\rangle$. 
\end{theorem} 

In his PhD dissertation \cite{Silk1}, N. Silkin considers a finitely presented group $E$ (called the extended Lamplighter group) that is an ascending HNN extension of the Lamplighter group $L$. Recall
$$L=\langle x, y\ |\  y^2=1, [y,x^{-k}yx^k]=1 \hbox{ for all } k>0\rangle$$ 
The homomorphism $m:L\to L$ induced by $x\to x$ and $y\to x^{-1}yxy$ is a monomorphism with image an index 2 subgroup. 
The group $E$ is the ascending HNN-extension of $L$ over $m$. 
A finite presentation for $E$ (with stable letter $t$) is 
$$\langle x,y,t\ |\   y^2=1, [x,t]=1, t^{-1}yt=x^{-1}yxy\rangle.$$ 
The group $E$ is metabelian with normal subgroup $N$ (the normal closure of $y$ in $E$) an infinite sum of groups, each isomorphic to  $\mathbb Z_2$ and $E/N$ is isomorphic to $\mathbb Z\oplus \mathbb Z$.
The subgroup of $E$ generated by $y$ and $x$ is isomorphic to the Lamplighter group. Note that the map taking $x$ to $t$, $t$ to $x$ and $y$ to $y$ induces an isomorphism of $E$ to $E$. Hence the subgroup of $E$ generated by $t$ and $y$ is isomorphic to the Lamplighter group. 

\begin{theorem} (Silken, \cite{Silk1}) \label{SILK1}
The extended Lamplighter group is simply connected at $\infty$. 
\end{theorem}

The next several results are contributed by Denis Osin.

\begin{proposition}  \label{OHNN} 
Suppose that a finitely presented group $G$ admits a surjective homomorphism $\varepsilon \colon G\to \mathbb Z$. Then $G$ splits as an HNN-extension of a finitely generated subgroup $A\leq Ker(\varepsilon)$.
\end{proposition}

\begin{proof}
Consider a finite presentation
\begin{equation}\label{pres}
G=\langle t, a_1, \ldots, a_k\mid R_1, \ldots , R_m\rangle
\end{equation}
of the group $G$ such that $\varepsilon(t)$ generates $\varepsilon(G)=\mathbb Z$.
Note that the total sum of exponents of all occurrences of $t^{\pm 1}$ in each $R_i$ equals $0$. Let $N$ be the maximal number of occurrences of $t$ in the relations of $G$. Every $R_i$ can be rewritten (possibly after conjugation by a suitable power of $t$) as a product of words $b_{\alpha\beta}^{\pm 1}=t^\beta a_\alpha^{\pm 1} t^{-\beta}$, where $1\le \alpha\le k$ and $0 \le \beta\le N$. Let
$$
B=\{b_{\alpha\beta}\mid 1\le \alpha\leq k,\, 0 \leq \beta\leq N\},
$$
and let $S_i$ be the word in the alphabet $B^{\pm 1}$ obtained from $R_i$ after such a rewriting. By using Tietze transformations, we can rewrite (\ref{pres}) in the form
$$
G=\langle t, B \mid S_1, \ldots, S_m, \mathcal T\rangle ,
$$
where $\mathcal T$ is the set of all relations of the form
$$tb_{\alpha, \beta}t^{-1} = b_{\alpha, \beta +1}$$
for $1\leq \alpha\le k$ and $0 \leq \beta\le N-1$.
In particular, $G$ is an HNN-extension of the subgroup $A=\langle B\rangle \le Ker(\varepsilon)$.
\end{proof}

\begin{corollary} \label{OsHNN} 
Every finitely presented, (locally finite)--by--$\mathbb Z$ group is (finite)--by--$\mathbb Z$.
\end{corollary}

\begin{proof}
Suppose  $1\to L\to G\to\mathbb Z\to 1$ is exact with $G$ finitely presented and $L$ locally finite. By Proposition \ref{OHNN}, $G$ is an HNN-extension of a finite group $A\leq L$. If the associated subgroups are proper in $L$ then $G$ contains a non-cyclic free subgroup, which contradicts the assumption that $G$ is (locally finite)--by--$\mathbb Z$. Therefore, the associated subgroups coincide with $L$ and $G$ is an extension of $A$ by $\mathbb Z$.
\end{proof}

M. Gromov \cite{Gro81} proved that groups of polynomial growth are exactly the virtually nilpotent groups. 
\begin{lemma} (Osin) \label{nil1} 
Every finitely generated nilpotent group is either virtually cyclic or virtually surjects onto $\mathbb Z \times \mathbb Z$. (Combining with, Theorem \ref{NOF2}: Every finitely presented nilpotent group is semistable at $\infty$.)
\end{lemma}  
\begin{proof}
The proof is by induction on the nilpotency class. If the finitely generated nilpotent group $N$ is abelian, this is obvious. Otherwise, consider $N/Z(N)$. If $N/Z(N)$ surjects onto $\mathbb Z \times \mathbb Z$, so does $N$. If $N/Z(N)$ is virtually cyclic, then $N$ is virtually (center-by-cyclic). Every center-by-cyclic group is abelian (exercise), hence $N$ is virtually abelian.
\end{proof}
\begin{theorem} (Osin) \label{metanil} 
All finitely presented virtually metanilpotent groups are semistable at $\infty$. 
\end{theorem}
\begin{proof} 
It suffices to prove that finitely presented metanilpotent groups $G$ are semistable at $\infty$. 
Consider the short exact sequence of groups:
$$1\to N_1\to G\to N_2\to 1$$
Where $N_1$ and $N_2$ are nilpotent. Suppose $N_1$ contains an element $g$ of infinite order. Every subgroup of a nilpotent group is subnormal, hence $\langle g\rangle$ is subnormal in $G$ and Theorem \ref{L} implies $G$ is semistable at $\infty$. Hence we may assume $N_1$ is locally finite. 

If $N_2$ maps onto $\mathbb Z\times \mathbb Z$, then so does $G$ and Theorem \ref{NOF2} implies $G$ is semistable at $\infty$. By Lemma \ref{nil1} we may assume that $N_2$ is virtually cyclic so that  $G$ is locally finite by virtually cyclic. Now apply Corollary \ref{OsHNN} to a subgroup of finite index in $G$ to finish the proof. 
\end{proof}

There are several results related to Theorem \ref{metanil}. 
The following is Corollary 0.2 of \cite{FW07}. A number of other results in the literature also imply Proposition \ref{AD1} (see the discussion in  \cite{FW07} preceding this result).
 
\begin{proposition} ( \cite{FW07}) \label{AD1}  Let $G$ be a finitely presented group. The asymptotic dimension of G is one if and only if G contains a non-trivial free group as a subgroup of finite index.
\end{proposition}
Suppose $1\to A\to G\to \mathbb Z\to 1$ is exact such that $G$ is finitely presented and $A$ is locally finite. Note that $G$ does not contain a free group of rank 2. An observation of D. Osin (see Example 0.3 of  \cite{FW07} and the discussion that follows) explaining why the asymptotic dimension of $G$ is one. Proposition \ref{AD1} implies $G$ contains an infinite cyclic group of finite index and hence $G$ is 2-ended and semistable. This line of reasoning also leads to a proof of Theorem \ref{metanil}.

\subsubsection{Word Hyperbolic and CAT(0) Groups}\label{WHyp}


See \cite{ABC91} for the basic definitions for word hyperbolic groups and M. Bridson and A. Haefliger's book \cite{BrHa99} for CAT(0) groups. 

Word hyperbolic groups and CAT(0) groups have boundaries and these boundaries have deep connections to the corresponding groups. While the boundary of a word hyperbolic group is unique, a CAT(0) group may have more than one boundary \cite{CK}. If $G$ is 1-ended, and hyperbolic or CAT(0), then any boundary of $G$ is a compact connected metric space. 
The work of B. Bowditch \cite{Bow99B},  G. Levitt \cite{Lev98} and G. Swarup \cite{Swarup} show that if $G$ is a 1-ended word hyperbolic group then $\partial G$, the boundary of $G$, has no (global) cut point. M. Bestvina and G. Mess \cite{BM91}, show (Propositions 3.2 and 3.2) the absence of cut points in $\partial G$ imply $\partial G$ is locally connected. Combining this with the Hahn Mazurkiewicz Theorem (see Theorem 31.5 of \cite{W70}) we have:
 
 \begin{theorem} 
 The boundary of a 1-ended word hyperbolic group is the continuous image of the interval $[0,1]$.
 \end{theorem}  
 
 It was pointed out by R. Geoghegan 
 that work of R. Geoghegan and J. Krasinkiewicz \cite{GK91}]  and J. Krasinkiewicz \cite{Kra77} imply that
 a word hyperbolic group $G$ is semistable at $\infty$  if and only if $\partial G$ has the shape of a locally connected continuum (see \cite{GS17} for details).  In particular:
 
\begin{theorem}\label{WHss} 
All word hyperbolic groups are semistable at $\infty$. 
\end{theorem}

It is unknown if all CAT(0) groups are semistable at $\infty$ and this is one of the more heavily studied questions in the theory. The paper \cite{GS17} is a good source of information for this problem. 

\begin{theorem} [\cite{SSh}] \label{CATQ}  All CAT(0) groups that act geometrically on a proper CAT(0) cube complex are semistable at $\infty$.
\end{theorem}
 
As with word hyperbolic groups,  a 1-ended proper CAT(0) space $X$ is semistable at $\infty$ if and only if $\partial X$ has the shape of a locally connected continuum or equivalently is pointed 1-movable.  C. Plaut \cite{Pl22} introduces the notion of a {\it weakly chained space}. He shows that a compact metric space is weakly chained if and only if it is pointed 1-movable.  Plaut \cite{Pl23} also connects the weakly chained condition for boundaries of CAT(0) spaces to topological properties of those spaces. 

If a 1-ended group $G$ acts geometrically on a proper CAT(0) space $X$, then the boundary of $X$ ($\partial X$) is called a {\it boundary of $G$}. The group $G$ is quasi-isometric to $X$ and Theorem \ref{QIsp} implies $G$ is semistable at $\infty$ if and only if $X$ is semistable at $\infty$. 

A boundary of a 1-ended CAT(0) group does not have a cut point (P. Papasoglu and E. Swenson \cite{PS09} and E. Swenson \cite{Sw}). Unfortunately, this does not imply that $X$ is semistable at $\infty$ in the same way word hyperbolic groups are semistable at $\infty$. Boundaries of 1-ended CAT(0) groups need not be locally connected or even path connected \cite{CMT06} and \cite{CK}. There are a number of interesting questions as to when boundaries of CAT(0) groups are unique, path connected, locally connected, simply connected etc.

\begin{theorem} [Theorem 3.1, \cite{GS17}] Suppose $G$ is a 1-ended group acting geometrically on a CAT(0) or word hyperbolic space $X$. Let $\ast$ be a base point in $X$. If all geodesic rays based at $\ast$ are properly homotopic, then $X$ (and hence $G$) is semistable at $\infty$. 
\end{theorem} 

\begin{corollary} 
[Corollary 4.2, \cite{GS17}] If the CAT(0) group $G$ is not rank 1, then $G$ is semistable at $\infty$.
\end{corollary}

\begin{definition}\label{cut}
A point $c$ of a continuum $Y$ is a {\it weak cut point} if there are two other points $a$ and $b$ of $Y$ such that any sub-continuum containing $a$ and $b$ must also contain $c$. In that case we say that $c$ {weakly separates} $a$ from $b$.
\end{definition}

\begin{theorem} [Theorem 5.5, \cite{GS17}] 
Suppose $X$ is a 1-ended proper CAT(0) space  and $X$ does not have semistable fundamental group at infinity then there are points $a, b, c\in \partial X$ such that $c$ weakly separates $a$ from $b$.
\end{theorem} 

This result is improved.

\begin{theorem} [Theorem 6.1, \cite{GS17}]
Suppose $X$ is a 1-ended proper CAT(0) space that does not have semistable fundamental group at infinity then there are points $a, b, c\in \partial X$ such that the Tits ball $B_T(c,{\pi \over 2})$ separates $a$ from $b$.
\end{theorem}

\begin{example}
The group $G=\langle x,y,z,w \mid [x,y],[y,z],[z,w]\rangle$ is a 1-ended CAT(0) group. The group $G$ acts geometrically on a CAT(0) square complex $X$ and $\partial X$ has weak cut points, is connected, but not path connected and not locally path connected (see \cite{CMT06}). The subgroups $\langle x,y\rangle$, $\langle y,z\rangle$  and $\langle z,w\rangle$ of $G$ are all isomorphic to $\mathbb Z\oplus \mathbb Z$. Hence by several applications of Theorem \ref{MComb}, $G$ is semistable at $\infty$.
\end{example}

If $X$ is a connected CAT(0) cube complex, let $Op(w,v)$ denote the maximal subcomplex of $Lk(w, X )$ composed of simplexes which are further from $v$ than $w$. 

Reading the following theorem with $m=0$, $0$-connected means connected, $0$-connected at infinity means 1-ended and $\pi_1$-semistable means semistable at $\infty$. 

\begin{theorem} [Corollary 4.2, \cite{BrM01}] Let $X$ be a locally finite, CAT(0) cubical complex with vertex $v$. If the link of $v$ in $X$ is $m$-connected, and if for each vertex $w\ne v$, $Op(w, v)$ is $m$-connected, then $X$ is $m$-connected at infinity and $\pi_{m+1}$-semistable.
\end{theorem}

\subsubsection{Relatively Hyperbolic Groups}\label{RHGs}
We use \cite{GMa08} as a reference for the basics on cusped spaces and relatively hyperbolic groups.
Note that there is no semistability hypothesis on the peripheral subgroups $P_i$ in the next result.

\begin{theorem} [Theorem 1.1, \cite{MS18}]\label{MainRH}  
Suppose $G$ is a 1-ended finitely generated group that is hyperbolic relative to a collection of 1-ended finitely generated proper subgroups ${\bf P}=\{P_1,\ldots, P_n\}$. If $\partial (G,{\bf P})$ has no cut point, then $G$ has semistable fundamental group at $\infty$. In particular, $H^2(G,\mathbb Z G)$ is free abelian. 
\end{theorem}

Generalizations, limitations and results related to this theorem are described in \S 2 of \cite{MS18}. A principal limitation of Theorem \ref{MainRH} is resolved by the following result. Note that there is no restriction on the number of ends of $G$ or the $P_i$.

\begin{theorem} [Theorem 1.5 \cite{HM20}] \label{HMSSMain}   
Suppose $G$ is a  finitely presented group that is hyperbolic relative to a collection of finitely generated subgroups ${\bf P}=\{P_1,\ldots, P_n\}$. If each $P_i$ has semistable fundamental group at $\infty$ then $G$ has semistable fundamental group at $\infty$.
\end{theorem}

See Theorem \ref{HomRH} for a homology version of Theorem \ref{HMSSMain}.

\begin{theorem} [Theorem 1.4, \cite {MS19}]\label{MS-SS}   
If a 1-ended finitely presented group $G$ is hyperbolic relative to $\mathcal P=\{P_1,\ldots ,P_n\}$ a set of 1-ended finitely presented subgroups (with $G\ne P_i$ for all $i$) then a cusped space for $(G,\mathcal P)$ has semistable fundamental group at $\infty$. 
\end{theorem}

Theorem \ref{HMSSMain} generalized a number of semistability results that preceded it and we list some of those results.

A group is {\it slender} if all of its subgroups are finitely generated, and it is {\it coherent} if all finitely generated subgroups are finitely presented.

\begin{theorem} 
[Theorem 1.1, \cite{HR17}] Let $(G,\mathbb P)$ be relatively hyperbolic with no non-central element of order two. Assume each peripheral subgroup $P\in\mathbb P)$ is slender and coherent and all subgroups of $P$ are semistable. Then $G$ is semistable at $\infty$. 
\end{theorem}

The only slender, coherent groups that we are aware of are the virtually polycyclic groups. Since virtually polycyclic groups are known to be semistable by \cite{M1} or Theorem \ref{metanil} (and all of their subgroups are again virtually polycyclic) the following corollary is immediate. 

\begin{corollary} 
[Corollary 1.2, \cite{HR17}] Let $(G,\mathbb P)$ be relatively hyperbolic with no non-central element of order two. If each $P\in \mathbb P$ is virtually polycyclic, then $G$ is semistable at $\infty$.
\end{corollary}

In his PhD dissertation, Michael Ben-Zvi proves the following result (which generalizes (Corollary 1.3, \cite{HR17})).
\begin{theorem} [\cite{BZ19}]
Let $G$ act geometrically on a CAT(0) space $X$ with isolated flats. Then $\partial X$ is path connected and hence $G$ is semistable at $\infty$.
\end{theorem}

\subsubsection{Ascending HNN Extensions and Eventually Injective Endomorphisms}\label{AHNNE}
Since the early 1980's it has been suggested that the finitely presented groups that are ascending HNN extensions with {\it finitely generated} base may include a group with non-semistable fundamental group at $\infty$. In \cite{M7} it is shown that ascending HNN extensions naturally break into two non-empty classes, those with bounded depth and those with unbounded depth. Those with finitely presented base have bounded depth (by definition). The main theorem of \cite{M7} shows that bounded depth finitely presented ascending HNN extensions with  finitely generated base groups have semistable fundamental group at $\infty$.  A technique for constructing ascending HNN extensions with unbounded depth is developed in Section 4 (see Theorem 4.2) of \cite{M7}. We consider this  construction to be the best attempt so far to produce a finitely presented group with non-semistable fundamental group at $\infty$. 

\begin{definition} 
Let $A$ be a group with presentation $\langle \mathcal A \mid \mathcal R\rangle$ and $\phi:A\to A$ be a monomorphism. Let $p:F(\mathcal A)\to A$ be the homomorphism of the free group $F(\mathcal A)$ to $A$ that takes $a$ to $a$ for each $a\in \mathcal A$. Let $\hat \phi:F(S)\to F(S)$ such that $p\hat \phi(a)=\phi p(a)$ for all $a\in \mathcal A$. The {\it ascending HNN-extension} of $A$ by $\phi$ has presentation: 
$$\langle \mathcal A, t\mid \mathcal R, t^{-1}at=\hat \phi(a) \hbox{ for all } a\in \mathcal A\rangle$$
and is denoted $A\ast_\phi$, the {\it base} of the extension is $A$ and $t$ is the {\it stable letter}.
\end{definition}

If a group $G$ has presentation $\langle \mathcal A,t\mid \mathcal R, t^{-1}at=\hat \phi(a) \hbox{ for all }a\in \mathcal A\rangle$ where $\hat \phi(a)$ is an element in the free group $F(\mathcal A)$ for each $a\in \mathcal A$, and the elements of $\mathcal R$ are elements in the free group on $\mathcal A$, then $G$ is an ascending HNN extension of the subgroup $A$ of $G$ generated by $\mathcal A$. The map $\hat\phi$ induces a monomorphism from $A$ to $A$, but it may be that $\langle \mathcal A\mid \mathcal R\rangle$ is not a presentation of $A$. 

The first semistability/simple connectivity at $\infty$ theorem on the subject of ascending HNN-extensions is the following:

\begin{theorem} [Theorem 3.1, \cite{HNN1}] \label{MM}
 Suppose $H$ is an infinite finitely presented group, $\phi:H\to H$ is a monomorphism and $G=H\ast_\phi$ is the resulting ascending HNN extension. Then $G$ is 1-ended and semistable at $\infty$. If additionally, $H$ is 1-ended, then $G$ is simply connected at $\infty$. 
\end{theorem}
Compare this last result to Theorem \ref{hnn}.

The next result is similar to that of Theorem \ref{MM}, but the proof requires new ideas. It is important in the proof of Theorem \ref{SSDecomp}.

\begin{theorem} [Theorem 1.4, \cite{Mih22}]\label{HNNE2}
Suppose $H_0$ is an infinite finitely presented group, $H_1$ is a subgroup of finite index in $H_0$, $\phi:H_1\to H_0$ is a monomorphism and $G=H_0\ast_\phi$ is the resulting HNN extension. Then $G$ is 1-ended and semistable at $\infty$. If additionally, $H_0$ is 1-ended, then $G$ is simply connected at $\infty$. 
\end{theorem}

\begin{proposition} [Proposition 2.2, \cite{M7}] \label{strongss}  
If $G$ is a finitely presented ascending HNN extension of a finitely generated infinite group $A$ and $t$ is the stable letter, then for all $N\geq 0$, $t^NAt^{-N}$ is  semistable at $\infty$ in $G$.
\end{proposition}

Consider the ascending HNN-presentation:
$$(\ast) \ \ \ \ \ \ \ \ \ \ \ \ \ \mathcal P=\langle t, \mathcal A\mid \mathcal R, t^{-1}at=\hat \phi(a)\hbox{ for all }a\in \mathcal A\rangle$$
Here $\mathcal A$ is a finite generating set of $A$, $\mathcal R$ is a finite set of elements (words) in the free group $F(\mathcal A)$ and for each $a\in \mathcal A$, $\hat \phi(a)$ is an element of  $F(\mathcal A)$. Let $G$ be the group with presentation $\mathcal P$. 
The base group of this HNN extension is $A$, the subgroup of $G$ generated by $\mathcal A$. The function $\phi:\mathcal A\to F(\mathcal A)$ defines a monomorphism $\phi:A\to A$.  Note that $\hat \phi$ extends to a homomorphism $\hat \phi:F(\mathcal A)\to F(\mathcal A)$. Equivalently, one could obtain every finitely presented ascending HNN presentation by considering arbitrary homomorphisms $\hat\phi:F(\mathcal A)\to F(\mathcal A)$ (with $\mathcal A$ finite) and forming the presentation $\mathcal P$ in the same way.

In order to define what it means for an ascending HNN extension to have bounded depth, we must first understand $ker(p)$ where $p$ is the homomorphism $p:F(\mathcal A)\to A$ (defined by $p(a)=a$ for $a\in \mathcal A$).  Certainly $ker(p)$ contains $N_0(\mathcal R,\hat \phi)\equiv N( \cup _{i=0}^\infty \hat \phi^i (\mathcal R))$, where $N( \cup _{i=0}^\infty \hat \phi^i (\mathcal R))$ is the normal closure of $ \cup _{i=0}^\infty \hat \phi^i (\mathcal R)$ in $F(\mathcal A)$. But it may be that for some word 
$w\in F(\mathcal A)$ and some integer $m$, $\hat \phi^m(w)\in N_0(\mathcal R,\hat \phi)$, and $w\not \in N_0(\mathcal R,\hat \phi)$. Then $w\in ker(p)$. Consider the normal subgroup of $F(\mathcal A)$: 
$$N^{\infty}(\mathcal R,\hat \phi)\equiv \cup _{i=0}^{\infty} \hat \phi^{-i}(N_0(\mathcal R,\hat \phi))\triangleleft F(\mathcal A).$$
 By  (Theorem 4.1,\cite{M7}),  $\hat \phi^{-i}(N_0(\mathcal R,\hat\phi))< \hat\phi^{-i-1}(N_0(\mathcal R,\hat\phi))$ so that  $N^{\infty}(\mathcal R,\hat\phi)$ is an ascending union of normal subgroups of $F(\mathcal A)$ and that $N^{\infty}(\mathcal R,\hat\phi)$ is the kernel of $p$, so 
$$A=\langle \mathcal A\mid N^{\infty}(\mathcal R,\hat\phi)\rangle$$

\begin{definition} 
If there is an integer $B$ such that $N^\infty(\mathcal R,\hat \phi)=\cup _{i=0}^B\hat \phi^{-i}(N_0(\mathcal R,\phi))$ then the presentation $\mathcal P$ of $G$ has {\it bounded depth}. 
\end{definition}

\begin{theorem} [Theorem 1.2,\cite{M7}] \label{BDHNN}  
If the presentation $\mathcal P$ of the ascending HNN-extension $G$ has bounded depth, then $G$ is semistable at $\infty$. 
\end{theorem}

There are many interesting examples of ascending HNN-extensions with and without bounded depth. We next develop this idea. 

\begin{lemma} \label{L1}  Suppose the following diagram of groups and homomorphisms commute:
$$A\ \ {\buildrel \phi\over \longrightarrow}\ \ A$$
$$\ \ \downarrow q\ \ \ \ \ \  \downarrow q$$
$$A_0\ {\buildrel  \phi_0\over \longrightarrow}\ \ A_0$$ 
Let $N_0=ker(q)$ and $K_0=\{1\}$. Inductively let $N_i=\phi^{-1} (N_{i-1})$ and $K_i=\phi_0^{-1}(K_{i-1})$ for $i\geq 1$. Then $N_i=q^{-1}(K_{i})$ for all $i\geq 0$. 
\end{lemma}

\begin{proof} The proof is by induction on $i$. First observe that $N_0=ker(q)=q^{-1}(K_0)$. 
Assume that $N_{i-1}=q^{-1}(K_{i-1})$ for some $i> 0$. We have:
$$a\in q^{-1}(K_i) \iff q(a)\in K_i\iff \phi_0q(a)\in K_{i-1} \iff q\phi(a)\in K_{i-1}\iff $$
$$\phi(a)\in N_{i-1} \iff a\in N_i.$$
\end{proof}
Next we apply Lemma \ref{L1} to a finite presentation of a group. 

\begin{lemma}\label{L2} 
Suppose the group $A_0$ has finite presentation $\langle \mathcal A\mid \mathcal R\rangle$ and $\phi_0:A_0\to A_0$ is a homomorphism with kernel $K_1$. 
Consider the following commutative diagram of groups: 
$$F(\mathcal A){\buildrel \phi\over \longrightarrow}F(\mathcal A)$$
$$\downarrow q\ \ \ \ \ \  \downarrow q$$
$$ A_0\ \ \ {\buildrel \phi_0\over \longrightarrow}\ \  A_0$$
Where $F(\mathcal A)$ is the free group on $\mathcal A$ and $q$ the quotient map with kernel $N_0$, the normal closure of $\mathcal R$ in $F(\mathcal A)$. Let $N_{i}=\phi^{-1}(N_{i-1})$, $K_0=\{1\}$ and $K_{i}=\phi_0^{-1}(K_{i-1})$ for $i\geq 0$. Then $N_i=q^{-1}(K_i)$ for all $i\geq 0$. 
\end{lemma}
If $Q$ is a subset of a group $G$, then let $N(Q)$ be the normal closure of $Q$ in $G$. Consider the group $G$ with ascending HNN presentation 
$$\mathcal P\equiv \langle t,\mathcal A\mid \mathcal R, t^{-1}at=\phi(a) \hbox{ for all }a\in \mathcal A\rangle.$$
Recall, (Theorem 4.1 of \cite{M7}) the subgroup $A$ of $G$ generated by $\mathcal A$ has presentation 
$$A=\langle \mathcal A\mid N^{\infty}(\mathcal R,\phi)\equiv \cup_{i=0}^\infty\phi^{-i}(N(\cup _{j=0}^\infty \phi^j(\mathcal R)))\rangle.
$$ 
Let $N_0=N(\cup _{j=0}^\infty \phi^j(\mathcal R))$ and $N_i=\phi^{-i}(N_0)$ for all $i\geq 1$. 
Furthermore, we have the relations:
 
\begin{enumerate} 
\item  $N_i\subset N_{i+1} 
\hbox{ for all }i\geq 0,\hbox{ and}$
\item  $\phi(N^{\infty} (\mathcal R,\phi))\subset N^\infty(\mathcal R,\phi)=\phi^{-1}(N^\infty(\mathcal R,\phi)).$
\end{enumerate}

An endomorphism $\phi : G \to G$ is {\it eventually injective} if for
some $n$, the restriction of $\phi$  to $\phi^n(G)$ is an injection.
It then follows that the restriction of $\phi$ to $\phi^n(G)$ is injective for all $m \geq n$.
Eventual injectivity of $\phi$ is also equivalent to requiring that the sequence of
kernels $ker(\phi^n)$ terminates. Note that $\phi(\phi^n(G)) = \phi^n(\phi(G)) \subset \phi^n(G)$.

Combining results we have the following curious theorem:

\begin{theorem} \label{BD} 
Suppose the group $A_0$ has finite presentation $\langle \mathcal A\mid \mathcal R\rangle$ and
$\phi_0:A_0\to  A_0$ is a homomorphism such that the following diagram (with $F({\mathcal A})$ the free group on $\mathcal A$ and $q(a)=a$ for $a\in \mathcal A$) commutes:
$$F(\mathcal A){\buildrel \phi\over \longrightarrow}F(\mathcal A)$$
$$\ \ \downarrow q\ \ \ \ \ \  \downarrow q$$
$$\ \  A_0\ \ {\buildrel \phi_0\over \longrightarrow}\ A_0$$
The ascending sequence $\{K_i=\phi_0^{-i}(K_0)=ker(\phi_0^{i+1})\}$ of normal subgroups of $A_0$ stabilize (equivalently $\phi_0$ is eventually injective) if and only if the group $G$ with ascending HNN presentation 
$$\mathcal P\equiv \langle t,\mathcal A\mid \mathcal R, t^{-1}at=\phi(a) \hbox{ for all }a\in \mathcal A\rangle$$
has bounded depth (and so has semistable fundamental group at $\infty$). 
\end{theorem}

A group $G$ is {\it residually finite} if $\{1_G\}$ is the intersection of finite index
subgroups of $G$. A group $G$ is {\it Hopfian} if every surjective endomorphism
of $G$ is injective. Mal\'cev proved that finitely generated residually finite groups are Hopfian \cite{LS77}. Hirshon generalized Mal\'cev's result in \cite{Hir77} by showing that
if $\phi : G \to  G$ is an endomorphism of a finitely generated residually finite group, and
$[G : \phi (G)] < \infty$ then $\phi$ is eventually injective.  Sela proved in \cite{Sel99} that
every endomorphism of a torsion free word hyperbolic group is eventually
injective. Wise gave an example in \cite{Wise99} of a finitely generated residually finite group $G$
with an endomorphism $\phi : G \to G$ that is not eventually injective. The main theorem of a recent preprint of D. Wise is:
\begin{theorem} (Wise)
Suppose that $G$ is a finitely generated linear group. Then every endomorphism of $G$ is eventually injective. 
\end{theorem}

In the setting of Theorem \ref{BD}: When $A_0$ is linear or a torsion free word hyperbolic group the ascending HNN group $G$ is semistable at $\infty$. If the homomorphism $\phi_0$ is an epimorphism with non-trivial kernel (so that $A_0$ is not Hopfian) then $\mathcal P$ has unbounded depth.

It is not always the case that such ascending HNN extensions have bounded depth.  
 
\begin{theorem} [Theorem 4, \cite{M4}] \label{BaseSS} 
Let $G$ be a group with generators $\{x, h_1,\ldots, h_n\}$ such that the subgroup of $G$ generated by $\{h_1,\ldots, h_n\}$ is infinite. If the subset  $\{h_1,\ldots, h_n,x^{-1}h_1x,\ldots, x^{-1}h_nx\}$ of $G$ generates a 2-ended or 1-ended and semistable at $\infty$ subgroup of $G$, then $G$ is 2-ended or 1-ended and semistable at $\infty$. 
\end{theorem}

As an immediate corollary we have:

\begin{corollary} \label{BaseSS2}  
Suppose $H$ is a finitely generated 1-ended and semistable at $\infty$ group. If $G$ is an ascending HNN extension of $H$, then $G$ is semistable at $\infty$. 
\end{corollary} 
The next result examines the question: Suppose $f:H_1\to H$ is an epimorphism of finitely generated ascending HNN extensions (taking stable letter to stable letter and base generators to base generators). If $H$ is semistable at $\infty$ then when is $H_1$ semistable at $\infty$? 

\begin{theorem} [Theorem 1.1, \cite{LM23}]  \label{thmfg} Suppose $S$ is a finite set, $\phi : S \to F(S)$ is a map to the free group generated by $S$, $R$ is a subset of $F(S)$ and the finitely generated ascending HNN-extension $H$ with presentation $\langle S, t \ |\  R, t^{-1} s t = \phi(s), s \in S \rangle$ is semistable at infinity. Let $\bar R$ be the kernel of the obvious epimorphism of the free group $F(\{t\}\cup S)$ to $H$.  Then there is a finite subset $R_0 \subseteq R$ such that those ascending HNN-extensions $H_1=\langle S, t \ |\  R_1, t^{-1} s t = \phi(s), s \in S \rangle$, with $R_0 \subseteq R_1 \subseteq \bar R$, are all $1$-ended and semistable at infinity as well. \end{theorem}
The authors also state:
\begin{remark}
Following the notation of the theorem, another presentation of $H$ is $\langle S, t \ |\  \bar R, t^{-1} s t = \phi(s), s \in S \rangle$. There is an  obvious epimorphism of $H_1$ onto $H$. If $R_0'$ is a finite subset of $\bar R$ such that $R_0'$ along with the conjugation relations $C=\{t^{-1} s t = \phi(s), s \in S\}$ are such that $\Gamma_{(H,S\cup \{t\})}(R_0'\cup C)$ is semistable at $\infty$, then it is an elementary matter to show that $R_0'$ is a consequence (in $F(\{t\}\cup S)$) of the conjugation relations and a finite set $R_0''\subset R$. We show that $R_0''$ can play the role of $R_0$ in Theorem \ref{thmfg}. Note that if the subset $R_1$ of $\bar R$ is finite then $H_1$ is a finitely presented ascending HNN-extensions that is semistable at infinity.
\end{remark}

\subsubsection{Thompson's Group $F$} \label{TGF} 
Let $x^y=y^{-1}xy$. Define the group $F$ to have the following infinite presentation.

$$F=\langle x_0,x_1,\ldots \mid x_n^{x_i}=x_{n+1} \hbox{ for all } 0\leq i<n\rangle$$
For $n=2$, $3$ or $4$, define $x_n=x_{n-1} ^{x_{n-2}}$. Then $F$ has finite presentation:
$$\langle x_0,x_1\mid  x_2^{x_0}=x_3, x_3^{x_1} =x_4\rangle$$

The group $F$ has many intriguing geometric and algebraic properties. It is clear from the first presentation that $F$ is an ascending HNN-extension of itself (with base group generated by $x_1, x_2, \ldots$ and stable letter $x_0$).  Theorem \ref{MM} implies $F$ is simply connected at $\infty$, but much more can be said. K. Brown and R. Geoghegan \cite{BG84} prove the remarkable facts that $F$ has type $F_{\infty}$ (see Theorem 9.3.19, \cite{G}) and that $H^n(F,\mathbb ZF)=0$ for all $n$ (see Theorem 13.11.1, \cite{G}). Combining this with the fact that $F$ is simply connected at $\infty$ implies that $F$ is $n$-connected at $\infty$ for all $n$. Collecting a number of results:

\begin{theorem} [\cite{BG84}] \label{ThompNC}
Thompson's group $F$ is a torsion free group of type $F_\infty$ that contains a non-finitely generated free abelian group and $F$ is $n$-connected at $\infty$ for all $n$. 
\end{theorem}

\subsubsection{Artin Groups and Coxeter Groups}\label{AGCG}
The semistability and simple connectivity at $\infty$ of Coxeter and right angled Artin groups is well understood. (See Section \ref{ACGE} for basic definitions.)

\begin{theorem} [Theorems 1.1 and 1.4, \cite{M96}] \label{ACSS} 
All finitely generated Artin groups and Coxeter groups are semistable at $\infty$.
\end{theorem}

If $\Lambda$ is the presentation diagram of $(W,S)$, and the vertices of a complete subgraph of $\Lambda$ generate a finite subgroup of $W$, then add a corresponding simplex to $\Lambda$. The resulting simplicial complex is the  {\it nerve} of $(W,S)$. 
The main theorem of \cite{DM02} implies:

\begin{theorem} \label{CoxSC}
The Coxeter groups $W$ (with Coxeter system $(W,S)$) is simply connected at infinity if and only if $L$ (the nerve of the system)  and the subcomplexes $L-\sigma$ (where $\sigma$ ranges over all simplexes in $L$) are simply connected.
\end{theorem}

A simplicial complex is {\it flag} when every complete subgraph of its 1-skeleton bounds a simplex. If $L$ is a finite simplicial complex, then the {\it right angled Artin group} $A_L$ has finite presentation with generators $S$ the vertices of $L$ and relation $sts^{-1}t^{-1}=1$ if $s$ and $t$ are adjacent vertices in $L$. Each right angled Artin group $A_L$ has an associated $K(A_L, 1)$ constructed by gluing together tori; Brady-Meier \cite{BrM01} denote this complex by $K_L$ and note that the universal cover $\tilde K_L$ is a CAT(0) cubical complex. 

It is convenient to combine homology and homotopy at infinity results here.

\begin{theorem} [Theorem A, \cite{BrM01}] 
Right angled Artin groups are homotopically and homologicially semistable at infinity in all dimensions. 
\end{theorem}

\begin{theorem} [Theorem B, \cite{BrM01}] \label{BRMB}
The group $A_L$ is $m$-connected ($m$-acyclic) at infinity if and only if the link of a vertex in $\tilde K_L$ is $m$-connected ($m$-acyclic.) 
\end{theorem} 
In fact the conditions of Theorem \ref{BRMB} can be replaced by conditions in $L$. 

\begin{theorem} [Lemma 5.1, \cite{BrM01}]The link of a vertex in $\tilde K_L$ is $m$-acyclic if and only if $L$ is $m$-acyclic, and for each simplex $\sigma\subset L$, the link of $\sigma$ in $L$ is $(m-|\sigma|-1)$-acyclic. The link of a vertex in $\tilde K_L$ is $m$-connected if, in addition to the conditions above, $L$ is simply connected. 
\end{theorem}
The condition $-1$-acyclic means non-empty and lower terms are vacuous. Combining the previous two results produces the following:
\begin{theorem} [Corollary 5.2, \cite{BrM01}] Let $L$ be a finite flag complex which is different from the 0- or 1-simplex. Then the right angled Artin group $A_L$  is simply connected at infinity if and only if $L$ is simply connected and contains no cut vertex.
\end{theorem}

\subsubsection{Sidki Doubles} \label{Sed}
The main theorems of this section are Theorems \ref{SDSSI} and \ref{SCISD}. 
Given a group $G$ the Sidki double of $G$ denoted $\mathfrak{X} (G)$ is obtained from the free product $G\ast \overline G$ (where $\overline G$ is another copy of $G$) by adding the commutation relations $[g,\overline g]$ for all $g\in G$. Our notation here means, that $\overline g$ is the copy of $g$ in the second factor of $G\ast \overline G$. 
Theorem A of \cite{BK19} is the main theorem of that paper:

\begin{theorem} [\cite{BK19}] \label{MainBK} 
The group $\mathfrak{X}(G)$ is finitely presented if and only if $G$ is finitely presented.
\end{theorem}

Sidki proves the following subgroups of $\mathfrak{X} (G)$ are normal (see Lemma 4.1.3 of \cite{Sid80}):
$$D=D(G):=[G,\overline G];\ \ \ \ L=L(G):= \langle g^{-1}\overline g\mid g\in G\rangle.$$

It follows that $D$ is the kernel of the natural map $q_D:\mathfrak{X} (G)\to  G \times \overline G$ and that $L$ is the kernel of the map $q_L:\mathfrak{X} (G) \to  G$ that sends both $g$ and $\overline g$ to $g$. This gives a short exact sequence:
$$(A)\ \ \ \ \ \ \ \ 1\to L\to \mathfrak{X}(G)\to G\to 1.$$
Let 
$$W=W(G):= D\cap L.$$
Sidki shows (Lemma 4.1.6 of \cite{Sid80}) that $D$ and $L$ commute and therefore $W$ is central in $DL$; in particular, $W$ is abelian. We have a quotient map $q_W:\mathfrak{X}(G)\to Q_G$ and exact sequence :
$$(B)\ \ \ \ \ \ \ \ 1\to W\to \mathfrak{X}(G)\to Q_G\to 1$$
When $G$ is finitely presented we have $Q_G$ is finitely presented (see Section 2 of \cite{BK19}).

\begin{proposition} [Proposition 2.3, \cite{BK19}] \label{Lfg} 
For all finitely generated groups $G$, $L(G)$ is finitely generated. 
\end{proposition} 

Observe that if $g$ and $ h$ are distinct elements of $G$ then $g^{-1}\overline g\ne h^{-1}\overline h$ (otherwise $hg^{-1}=\overline h\overline g^{-1}$ and projecting $\mathfrak{X}(G)\to G\times \overline G$ implies  $g=h$). In particular, if $G$ is infinite then $L$ is infinite. 

Combine Propositions \ref{Lfg}, Theorem \ref{M1} and the short exact sequence (A) we have:

\begin{theorem} \label{SDSSI} 
If $G$ is an infinite finitely generated group then $\mathfrak{X}(G)$ is 1-ended and semistable at infinity. 
\end{theorem}

Let $L_1(G)=q_D(L)< G\times \bar G$. Then for each $g\in G$, $(g,\overline g^{-1})\in L_1$ and $L_1$ is infinite. Since $L_1=L/(D\cap L)=L/W$, the group $W$ has infinite index in $L$. This implies that $L_2(G)=q_W(L)<Q_G$, is infinite. Note that $L_2$ and $L_1$ are both isomorphic to $L/W$. 

Since $\mathfrak{X}(G)/D\cong G\times \overline G$ and $W< D$,  the quotient map  $q:Q_G\to G\times \overline G$ (by $q_W(D)$) isomorphically maps $L_2$ to $L_1$, the kernel of the map $G\times \overline G\to G$  that takes both $g$ and $\overline g$ to $g$ (recall that $D\cap L=W$, in order to see this is an isomorphism). This implies that $L_2$ has infinite index in $Q_G$. 

\begin{proposition}\label{INII}
For any infinite group $G$, $L_2(G)$ is an infinite normal subgroup of infinite index in $Q_G$.
\end{proposition}

By Proposition \ref{Lfg}, when $G$ is finitely generated, $L_2(G)$ is finitely generated. We have a short exact sequence:
$$(C)\ \ \ \ \ \ \  1\to L_2\to Q_G\to G\to 1.$$
Combining Propositions \ref{Lfg} and \ref{INII}, Theorems \ref{M1} and \ref{J}, and the exact sequence (C) we have:

\begin{theorem}\label{QG1E} 
For any infinite finitely generated group $G$, $Q_G$ is 1-ended and semistable at $\infty$. If $G$ and $L_2(G)$ are finitely presented and infinite, then $Q_G$ is finitely presented and simply connected at $\infty$. 
\end{theorem}

Combining Theorem \ref{T187a}, Theorem \ref{QG1E} (in order to see $Q_G$ has 1-end), and the exact sequence (B) we have:

\begin{theorem}\label{SCISD} 
If $G$ is a finitely presented infinite group and $W(G)$ contains an element of infinite order, then $\mathfrak{X}(G)$ is simply connected at $\infty$. 
\end{theorem}
It is remarked in \cite{BK19} that if $F$ is a finitely generated free group of rank $\geq 3$ then $W(F)$ is not finitely generated. Theorem \ref{SCISD} implies $\mathfrak{X}(F)$ is simply connected at infinity. It is remarkable that a finitely presented 1-ended group that sits between $F\ast \overline F$ and $F\times \overline F$ is simply connected at infinity. 



Theorem \ref{MainA} implies the following result in the finitely generated case:
\begin{theorem}
If $G$ is infinite, finitely generated and recursively presented, and $W(G)$ has rank $\geq 2$, then $\mathfrak{X}(G)$ is simply connected at $\infty$. 
\end{theorem}



The structure of $W(G)$ is examined in the latter sections of \cite{BK19}.

\subsubsection{Stable Fundamental Group at $\infty$} \label{StablePP1}
An inverse sequence of groups is {\it stable} if it is pro-isomorphic to an inverse sequence where the bonding maps are isomorphisms. B. Bowditch determines all the possibilities that result in stable fundamental pro-group for finitely presented groups. See Theorem \ref{stableproH} for the corresponding homology result.

\begin{theorem} [\cite{Bo04}] \label{stablepro} Suppose $G$ is a finitely presented group and $\pi_1(\varepsilon G)$, the fundamental pro-group of $G$, is stable. Then either $G$ is simply connected at $\infty$ or $G$ is virtually a closed surface group and $\pi_1(\varepsilon G)$ is pro-isomorphic to an inverse sequence where each group is $\mathbb Z$ and each bonding map is an isomorphism.
\end{theorem}

\subsubsection{Exotic Fundamental Pro-Group and Fundamental Group at $\infty$ for a Group}\label{WildPro} 

In the early stages of development of the semistability theory for groups, it was not clear whether or not the fundamental pro-group of a finitely presented group is always pro-free or not. 
\begin{example} \label{DavisEx} 
The first example of groups with non pro-free fundamental pro-group were exhibited by M. Davis  \cite{Davis83}. Examples were given where the fundamental pro-group of a finitely presented group $G$ were inverse sequences of $n$-fold free products of a group $H$ where the bonding map killed the last free factor and the group $H$ is the fundamental group of a homology 3-sphere which bounds a contractible 4-manifold. Thus $H$ can be chosen to be finite, giving torsion in $\pi_1^e(G)$. 
\end{example}

\begin{theorem} [\cite{ DM02}] \label{RecRA} For any recursively presented group $G$ there is a 1-ended right-angled Coxeter group $W$ with $G$ a subgroup of $\pi_1^e(W)$.
\end{theorem}

In the next theorem, the {\it mapping class group of $\tilde X$}, denoted $\mathcal M(\tilde X)$ is the (discrete) group of ambient isotopy classes of self-homeomorphism of $\tilde X$. The {\it weak mapping class group}  $\mathcal W \mathcal M(\tilde X)$ is the (discrete) group of proper homotopy classes of self-homeomorphisms; it is a quotient of $\mathcal M(\tilde X)$.  

\begin{theorem} [Theorem A, \cite{GM96}] \label{fginf}  
Let the infinite group $G=\pi_1(X,v)$, where $X$ is a finite connected complex. If $G$ has 2-ends, $\pi_1^e(\tilde X,r)$ is trivial for any proper base ray $r$. Otherwise, one of the following holds: 
\begin{enumerate} 
\item $\tilde X$ is 1-ended and simply connected at $\infty$;
\item $\tilde X$ is 1-ended and semistable at $\infty$, and $\pi_1^e(\tilde X)$ is discrete and infinite cyclic;
\item $\tilde X$ is 1-ended and semistable at $\infty$, and $\pi_1^e(\tilde X)$ is freely generated by an infinite (pointed) compact metric space;
\item letting $\rho:G\to \mathcal{WM} (\tilde X)$ denote the natural representation of $G$ by covering transformations, every element of $\hbox{ker }\rho$ has finite order in $G$. 
\end{enumerate}

\end{theorem}

Note that when $G$ is torsion free, 4 simplifies to:

 \medskip
 
{\it 4$'$. the representation $\rho$ is faithful.} 

\begin{corollary} [Corollary $A'$, \cite{GM96}] 
If in Theorem \ref{fginf}, $G$ is torsion free and 1-ended but is not semistable at $\infty$ then the representation $\rho:G\to \mathcal {WH}(\tilde X)$ is faithful.
\end{corollary}

If $G$ is 1-ended and semistable at $\infty$ and the fundamental pro-group $\pi_1(\varepsilon G)$ is pro-isomorphic to an inverse sequence where each bonding map is an isomorphism, then either $G$ is simply connected at $\infty$ or (by Theorem \ref{stablepro}) $G$ is virtually a closed surface group and $\pi
_{1}^{e}(G)=\mathbb{Z}$. (Compare with Theorem \ref{fginf} (2).)

\subsubsection{The Lamplighter Group and Other Non-Semistable at $\infty$ Finitely Generated Examples }\label{NonSS}

It was conjectured in \cite{M4} that the ``Lamplighter group" with presentation: 
$$L\equiv \langle x,y\mid y^2=1, [y,y^{x^i}]=1 \hbox{ for all } i\in \mathbb Z\rangle$$
is  not semistable at $\infty$ in the sense of Definition \ref{defss3}. Here $a^b=bab^{-1}$. 

In N. Silkin's thesis \cite{Silk1}, he shows the homomorphism $\phi:L\to L$ induced by $x\to x$ and $y\to yxy^{-1}x^{-1}$ is a monomorphism with image a subgroup of index 2. The resulting HNN-extension (called the extended lamplighter group) has finite presentation:
$$E\equiv  \langle t,x,y\mid  y^2=1, [t,x]=1, t^{-1}yt=yxy^{-1}x^{-1}\rangle$$
Note that (since $y=y^{-1}$), $1=t^{-1}y^2t=yxy^{-1}x^{-1}yxy^{-1}x^{-1}=[y,y^x]$.

Silkin goes on to show that in $E$, $N(y)$, the normal closure of $y$ is an infinite direct sum of groups all isomorphic to $\mathbb Z_2$ and that $E/N(y)$ is isomorphic to $\mathbb Z\times\mathbb Z$. So $E$ is metabelian of rank 2. Then he shows (by a geometric argument) that  $E$ is simply connected at infinity. The lamplighter group is not semistable at $\infty$ in the sense of Definition \ref{defss3}.

\begin{theorem} [(Theorem 1, \cite{MML}] \label{LNss} 
The lamplighter group is not semistable at $\infty$.
\end{theorem}

\begin{theorem} [Theorem 2, \cite{Silk1}]\label{ExLsc} 
The extended lamplighter group 
$$E=\langle x,y,t\mid  y^2=1, [x,t]=1, t^{-1}yt=x^{-1}yxy\rangle$$
is simply connected at $\infty$. 
\end{theorem}
  
Paradoxically, it is not clear whether or not the lamplighter group is simply connected at $\infty$ under Definition \ref{SCfg}.

\begin{question}
Is the lamplighter group simply connected at $\infty$?
\end{question}
Denis Osin suggests that the ideas behind the proof of Theorem \ref{LNss} may extended to a class of groups which contains all finitely generated groups $G$ that fit into a exact sequence:
$$1\to K\to G\to F\to 1$$ 
where $K$ is a locally finite infinite group and $F$ is an infinite finitely generated free group as well as finitely generated not finitely presented limits of virtually free groups. 

For this last class of groups there is a sequence of epimorphisms: 
$$F_0{\buildrel q_1\over  \twoheadrightarrow} F_1 {\buildrel q_2 \over \twoheadrightarrow}\cdots $$
where each $F_i$ is a finitely presented infinite virtually free group. If $K_i$ is the kernel of $q_i$ and $K^i$ is the pre-image of $K_i$ in $F_0$ then $K^1<K^2<\ldots$ is an ascending sequence of normal subgroups in $F_0$. Let $K=\cup_{i=1}^\infty K^i$ then the limit of the sequence of epimorphisms is  the group $G=F_0/K$. 

\medskip

\noindent {\bf A ``torsion free" version of the Lamplight group.} Recall that the Lamplighter group has presentation:
$$L\equiv \langle x,y\mid y^2=1, [y,y^{x^i}]=1 \hbox{ for all } i\in \mathbb Z\rangle .$$
Removing the relation $y^2=1$ gives a presentation:
$$H\equiv \langle x,y: [y,y^{x^i}]=1 \hbox{ for all } i\in \mathbb Z\rangle$$
Note that the normal closure of $y$ in $H$ is an infinite order free abelian subgroup of $H$ with basis $\{y^{x^i}:i\in \mathbb Z\}$. 
Let $F$ be the free group on $\{x,y\}$. First we show the map $\phi:F\to F$ determined by $x\to x$ and $y\to yy^x$ induces a monomorphism of $H$ to $H$. Observe that 
$$(\ast)\ \ \ \ \ \ \ \ \ \ \ \phi(y^{x^i})=x^iyy^xx^{-i}=y^{x^i}y^{x^{i+1}}.$$

In order to see that $\phi$ induces a homomorphism on $H$, it suffices to show that $\phi([y,y^{x^i}])$ is an $H$-relator for all $i\in \mathbb Z$.  First observe that in $H$
$$[y^{x^i},y^{x^j}]=1 \hbox{ for all } i,j\in \mathbb Z.$$
Applying $(\ast)$:
$$\phi([ y, y^{x^i}])=(yy^x)( y^{x^i} y^{x^{i+1}})( (y^{-1})^x y^{-1})  ((y^{-1})^{x^{i+1}}(y^{-1})^{x^i}).$$
All terms of this last expression commute in $H$ and each is paired with its inverse. Hence $\phi$ induces a homomorphism  $\phi':H\to H$.

First observe that in an $x,y$ word in the free group $F$ (or in $H$), the letters $x$ and $x^{-1}$ can be ``slid to the left" as follows:
$$yx=xy^{x^{-1}},\  y^{-1}x=x(y^{-1})^{x^{-1}},\  yx^{-1} =x^{-1}y^x\hbox{ and }y^{-1}x^{-1}=x^{-1}(y^{-1})^x$$ 
Hence, any element of $F$ (or $H$) can be written as 
$$(\ast\ast)\ \ \ \ \ \ \ \ h=x^ny_1^{t_1}y_2^{t_2}\cdots y_n^{t_n}$$
where $t_i=x^{k(i)}$ and $y_i=y^{j(i)}$ (here $j(i)$ and $k(i)$ are integers) for all $i$. The homomorphism $q:H \to \mathbb \langle x\rangle$ with kernel equal to the normal closure of $y$ is such that $q\phi'=q$. So that if $h\in ker(\phi')$ then $n=0$ .  Suppose $h\ne 1\in H$. The normal closure of $y$ in $H$ is a free abelian group with basis $\{y^{x^i}: i\in \mathbb Z\}$. Arrange the $t_i$ in $(\ast\ast )$ so that $k(1)<k(2)<\cdots <k(n)$. Then 
$$\phi'(h)= y_1^{x^{k(1)}}y_1^{x^{k(1)+1}}\cdots y_n^{x^{k(n)}}y_n^{x^{k(n)+1}}$$
and $\phi'(h)\ne 1$ since $y_n^{x^{k(n)+1}}$ does not cancel in this (abelian) expression. The homomorphism $\phi'$ is a monomorphism. 

The ascending HNN-group $G\equiv \langle t, H:t^{-1}ht=\phi(h) \hbox{ for all } h\in H\rangle$ has presentation:
$$\langle t,x,y:t^{-1}xt=x, t^{-1}yt=yy^{x}, [y,y^{x^i}]=1 \hbox{ for all }i\rangle.$$
The $H$-relation $[y, y^{x}]$ and the conjugation (by $t$) relations induce the relation  
$$\phi(yy^xy^{-1}(y^{-1})^x)=(yy^x)(y^xy^{x^2})((y^{-1})^xy^{-1})((y^{-1})^{x^2}(y^{-1})^x).$$
Conjugating the relation $[y, xyx^{-1}]$ by $x$ induces $[y^x,y^{x^2}]$ (so that $y^x$ commutes with both $y$ and $y^{x^2}$). These relations reduce the last expression to 
$$y(y)^{x^2}y^{-1}(y^{-1})^{x^2} (=[y,y^{x^2}]).$$
An induction argument shows that for each $i$, $[y,y^{x^i}]$ is a consequence of $[y, xyx^{-1}]$ (and the conjugation relations), reducing our presentation to:
$$G=\langle t,x,y:t^{-1}xt=x, t^{-1}yt=yxy^{-1}x^{-1}, [y,y^{x}]=1\rangle.$$

\subsubsection{Amalgamated Products, HNN-Extensions, 1-Relator Groups and Subgroup Combination Results} \label{Comb1R}

Theorem \ref{FIss} as well as the main results of Sections \ref{Sed} and \ref{GPofG} are combination results. In this section we concentrate on semistability results for graphs of groups decompositions of groups and subgroup combination theorems. 
Perhaps the most sophisticated result in the theory of semistability at $\infty$ for groups is the following. Note that it is non-trivial even when the groups $A$ and $B$ are free groups and $G=A\ast_CB$ (see Example \ref{Examples3}).

\begin{theorem} [(\cite{MT1992}] \label{MTComb}  
If $G$ is the fundamental group of a finite graph of groups where each vertex group is finitely presented with semistable fundamental group at $\infty$, and each edge group is finitely generated, then $G$ has semistable fundamental group at $\infty$. 
\end{theorem}

\begin{conjecture}\label{NSSsplit} 
Suppose $G$ is a finitely presented group that does not have semistable fundamental group at $\infty$, then $G$ splits non-trivially.
\end{conjecture}

The motivation for Conjecture \ref{NSSsplit} is Stallings' Splitting Theorem (\ref{Stall}). There is a natural notion of the {\it strong ends} of a space $X$ (the equivalence classes of proper maps $p:[0,\infty)\to X$ where $p$ and $q$ are equivalent if they are properly homotopic in $X$). Stallings' Theorem implies that groups with more than one end split. The analogy would be - more than one strong end forces the group to split. Note that if Conjecture \ref{NSSsplit} is true, then any group that can be embedded as a subgroup of finite index in an FA group would have semistable fundamental group at $\infty$. This observation would apply to groups with property $(T)$ (all subgroups of finite index of such groups also have property $(T)$) since such groups are FA. As an example $SL_3(\mathbb Z)$ has property $(T)$.

The next result implies Theorem \ref{MTComb} if vertex groups are 1 or 2-ended and edge groups are infinite.  

\begin{theorem} [Theorem 3, \cite{M4}] \label{MComb} 
Suppose  $A$ and $B$ are finitely generated 1 or 2-ended subgroups of the finitely generated group $G$ and both $A$ and $B$ are semistable at $\infty$.  If the set $A\cup B$ generates $G$ and the group $A\cap B$ contains a finitely generated infinite subgroup then $G$ has semistable fundamental group at $\infty$. 
\end{theorem}

The proof of the next result completely parallels the proof of Theorem \ref{MComb}. When $A$ and $B$ are finitely presented it is used in the proof of Theorem \ref{SSDecomp}.

\begin{theorem} [Lemma 5.1, \cite{Mih22}] \label{SS2} 
Suppose $G=A\ast_CB$ or $G=A\ast_C$, where $C$ is infinite, finitely generated and in the first case, of finite index in $B$. If $A$ is finitely generated, 1-ended and semistable at $\infty$, then $G$ is 1-ended and semistable at $\infty$.
\end{theorem}
Theorem \ref{MTComb} is used to prove the following result:

\begin{theorem} [Theorem 1, \cite{MT92}] \label{OneR} 
All 1-relator groups are semistable at $\infty$.
\end{theorem}

The following result generalizes Theorem \ref{FIss}. 

\begin{theorem} [Theorem 1.3, \cite{Mih22}] \label{SSDecomp} 
Suppose $G$ is the fundamental group of a connected reduced graph of groups, where each edge group is infinite and finitely generated, and each vertex group is finitely presented and either 1-ended and semistable at $\infty$ or has an edge group of finite index. 
Then $G$ is 1-ended and semistable at $\infty$. 
\end{theorem}

 It is straightforward to see the following definition is independent of generating sets. 

\begin{definition}\label{SSSin} 
Suppose $A$ is a finitely generated subgroup of the finitely generated group $G$. Let $\mathcal A\subset \mathcal S$ where $\mathcal A$ generates $A$ and $\mathcal S$ generates $G$. We say that $A$ is {\it strongly semistable at $\infty$ in $G$}  if there is a finite set $R$ of $G$-relations (in the letters $\mathcal S$) such that:

For any compact set $C$ in  $\Gamma_{(G,\mathcal S)}(R)$ there is a compact set $D$ in $\Gamma_{(G,\mathcal S)}(R)$ so that if $r$ and $s$ are proper $\mathcal A$-edge path rays at the  vertex $v\in \Gamma_{(G,\mathcal S)}(R)-D$, then $r$ and $s$ are properly homotopic rel$\{\ast\}$ in $\Gamma_{(G,\mathcal S)}(R)-C$. 
\end{definition}

With minor adjustments the proof of Theorem \ref{MComb} implies the following result. Note that we do not assume that $A$ or $B$ is 1-ended.

\begin{theorem} \label{SSSComb} 
Suppose  $A$ and $B$ are finitely generated subgroups of the finitely generated group $G$ such that $A$ (respectively $B$) is a 2-ended group or is strongly semistable at $\infty$ in $G$,  the set $A\cup B$ generates $G$ and the group $A\cap B$ contains a finitely generated infinite subgroup $J$, then $G$ is 1-ended and semistable at $\infty$. 
\end{theorem}

The next definition is in analogy with Definition \ref{ss+cossD}, but weaker than Definition \ref{SSSin}.

\begin{definition}\label{SSin} 
Suppose $A$ is a finitely generated subgroup of the finitely generated group $G$. Let $\mathcal A\subset \mathcal S$ where $\mathcal A$ generates $A$ and $\mathcal S$ generates $G$. We say that $A$ is {\it  semistable at $\infty$ in $G$}  if there is a finite set $R$ of $G$-relations (in the letters $\mathcal S$) such that:

For any compact set $C$ in  $\Gamma_{(G,\mathcal S)}(R)$  there is a compact set $D$ in $\Gamma_{(G,\mathcal S)}(R)$  so that if $r$ and $s$ are proper $\mathcal A$-edge path rays at the vertex $v\in \Gamma_{(G,\mathcal S)}(R)$ , then $r$ and $s$ are properly homotopic rel$\{\ast\}$ in $\Gamma_{(G,\mathcal S)}(R)-C$. 
\end{definition}

\begin{question} \label{MMComb} 
Suppose  $A$ and $B$ are finitely generated subgroups of the finitely generated group $G$ such that $A$ (respectively $B$) is a 2-ended group or is semistable at $\infty$ in $G$,  the set $A\cup B$ generates $G$ and the group $A\cap B$ contains a finitely generated infinite subgroup $J$, then is $G$ either 1 or 2-ended and semistable at $\infty$? 
\end{question}


In the statement of the next four theorems, B. Jackson assumes that the subgroup $H$ of the group $G$ is finitely presented, but his proof only requires $H$ to be finitely generated (see the first remark in \cite{J82b}). We state the more general results. 

\begin{theorem} [Theorem 1, \cite{J82b}]  \label{JackI} 
Suppose $G=G_1\ast_HG_2$ where $G_1$ and $G_2$ are finitely presented and 1-ended and $H$ is finitely generated with more than 1-end, then $G$ has 1-end but is not stable at $\infty$. 
\end{theorem}

\begin{theorem} [Theorem 2, \cite{J82b}] \label{JackIi}
Suppose $G=G_1\ast_HG_2$ where $G_1$ and $G_2$ are finitely presented 1-ended and simply connected at $\infty$ and $H$ is finitely generated and 1-ended. Then $G$ is simply connected at $\infty$.
\end{theorem} 

\begin{theorem} [Theorem $1^\ast$, \cite{J82b}] \label{JackII} 
Suppose $G$ is the HNN group $G_1 \ast _{f:H\to K}$ where $G_1$ is finitely presented and 1-ended and $H$ is finitely generated with more than 1-end, then $G$ is 1-ended, but is not stable at $\infty$. 
\end{theorem}

\begin{theorem} [Theorem $2^\ast$, \cite{J82b}]\label{JackIIi} 
Suppose $G$ is the HNN group $G_1 \ast _{f:H\to K}$ where $G_1$ is finitely presented, 1-ended and simply connected at $\infty$. If $H$ is finitely generated and 1-ended, then $G$ is 1-ended and simply connected at $\infty$
\end{theorem}

See Theorem \ref{JHom} for the homological versions of Theorems \ref{JackI} and Theorem \ref{JackII}

\subsubsection{Graph Products of Groups}\label{GPofG}
Graph products of groups generalize right angled Artin groups (when all vertex groups are $\mathbb Z$)  and right angled Coxeter groups (when all vertex groups are $\mathbb Z_2$). 

\begin{theorem} [\cite {MMGP}]\label{GraphP}   
Suppose $G$ is a graph product on the finite connected graph $\Lambda$ where each vertex group ($G_v$ for the vertex $v$) is finitely presented. Then $G$ does not have semistable fundamental group at $\infty$ if and only if there is a vertex $v$ of $\Lambda$ such that: 

(1) $G_v$ does not have semistable fundamental group at $\infty$ and  

(2) the link of $v$ is a complete graph with each vertex group finite.
\end{theorem} 

If $\Lambda_1$ is a subgraph of $\Lambda$ then let $\langle \Lambda_1\rangle$ denote the subgroup of $G$ generated by the vertex groups of $\Lambda_1$. For $v$ a vertex of $\Lambda$, let $lk(v)$ (the link of $v$)  be the full subgraph of $\Lambda$ on the vertices adjacent to $v$. Let $st(v)$ (the star of $v$) be the full subgraph on $\{v\}\cup lk(v)$. 
When the graph product $G$ is not semistable (and $\Lambda$ is not complete), then $G$ ``visually" splits as an amalgamated product 
$$G=\langle st(v)\rangle \ast_{\langle lk(v)\rangle} \langle \Lambda-\{v\}\rangle$$
where $v$ is a vertex of $\Lambda$, $G_v$ is not semistable and $\langle lk(v)\rangle$ is a finite group (so $\langle st(v)\rangle$ is the direct product of the finite group $\langle lk(v)\rangle$ with the non-semistable group $G_v$ and so is not semistable at $\infty$). In this situation $G$ is covered by the next result.  

\begin{theorem} [Theorem 3.3, \cite{MMGP} or Proposition 3.1, \cite{CLQR21}] \label{Fsplit}
 Suppose the group $G$ has a graph of groups decomposition $\mathcal G$ where each edge group is finite and each vertex group is finitely presented. The group $G$ has semistable fundamental group at $\infty$  if and only if each vertex group of $\mathcal G$ has semistable fundamental group at infinity. 
\end{theorem}


Part (1) of the next result follows directly from Theorem \ref{GraphP}. Part (2) follows directly from Theorems \ref{GraphP} and \ref{OV}.
 
\begin{corollary} \label{C11}
 Suppose $G$ is a graph product on the finite connected graph $\Lambda$ and each vertex group of $\Lambda$ is finitely presented. If either of the following conditions are met, then $G$ has semistable fundamental group at $\infty$:

 (1) Each vertex group has semistable fundamental group at $\infty$.
 
  (2) The group $G$ is 1-ended and $\Lambda$ is not complete. 
\end{corollary}

\begin{example}
Consider the graph which is a hexagon. Alternate vertex groups as $\mathbb Z_2$ and $\mathbb Z\ast \mathbb Z$. This graph group is 1-ended (Theorem \ref{OV})  and semistable (Theorem \ref{GraphP}).

\end{example}




\subsubsection{Asymptotic Dimension 2 and Simple Connectivity at $\infty$}  \label{AD2}

If the asymptotic dimension of a finitely presented group is 2, then intuitively the group is 1-dimensional at $\infty$. This idea is supported by boundaries of hyperbolic and CAT(0) groups. For instance, surface groups are of asymptotic dimension 2 and have 1-dimensional boundary. Spaces that are 1-dimensional and contain a circle have non-trivial fundamental group. So, one would not expect groups with asymptotic dimension 2 that contain a surface group to be simply connected at $\infty$. The two groups of the next example seem counter-intuitive. 

\begin{example} The Extended Lamplighter group and the Sidki double of the free group $F_n$, for $n\geq 3$ have asymptotic dimension 2, but are simply connected at $\infty$ (see Theorems \ref{ExLsc}, \ref{SCISD} and the remarks that follow). The Sidki double of $F_n$ is torsion free, as well, and ``sits between" $F_n\ast F_n$ and $F_n\times F_n$.
$$F_n\ast F_n\to \mathfrak{X}(F_n) \to F_n\times F_n.$$
Note that the CAT(0) boundary of $F_n\times F_n$ is the (1-dimensional) join of two Cantor sets.
\end{example}

\subsubsection{Useful Technical Results}\label{tech}

The following is elementary.
\begin{theorem}\label{squares}
Suppose $H$ is a group. Let $\bar H$ be the subgroup of $H$ generated by all elements $h^2$ for $h\in H$. The group $\bar H$ is normal in $H$ and $H/ \bar H$ is an (abelian) 2-group (each non-trivial element has order 2). In particular, if $H$ is finitely generated, then $\bar H$ has finite index in $H$.   
\end{theorem}
\begin{proof}
If $h\in H$ then $h=h_1^2\cdots h_n^2$ for some collection of $h_i\in H$. If $g\in H$ then $g^{-1}hg=g^{-1}h_1^2g\cdots g^{-1}h_n^2g\in \bar H$, so $\bar H$ is normal in $H$. Consider the quotient space $K=H/\bar H$. Every non-trivial element of $K$ is order 2.  (Such groups are abelian: If $k\in K$, then $k^2=1$, and so $k=k^{-1}$ for all $k\in K$.  If $k_1,k_2\in K$ then $k_1k_2k_1^{-1}k_2^{-1}=k_1k_2k_1k_2=(k_1k_2)^2=1$.) 
\end{proof}

\begin{theorem} [Corollary 3.1.4, \cite{M1}]
Suppose 
\begin{enumerate}
\item The space $X$ is a finite complex such that $\pi_1(X)=G$ and $G$ is a 1-ended group. 

\item The group $H$ is a finitely generated subgroup of infinite index in $G$.

\item The $p: \tilde X\to \tilde X/H$ is the  quotient map. 

\item The group of covering transformations of $\tilde X/H$ contains a finitely generated infinite subgroup.  

\item The proper rays $r$ and $s$ in $\tilde X$ are such that both $pr$ and $ps$ have image in a compact subset of $\tilde X/H$.
\end{enumerate}
Then $r$ and $s$ are properly homotopic in $\tilde X$. 
\end{theorem}

The next result has been used successfully in a number of situations. Its use arises when an over space is known to be  semistable at $\infty$ and certain subspaces have special properties. Recall that Cusped Spaces for relatively hyperbolic groups $(G,\mathcal P)$ may always be semistable at $\infty$ (Theorem \ref{MS-SS}).  If the first homology of the peripherals $P\in \mathcal P$ is known to be semistable, then the proof of Theorem \ref{HomRH} uses the next result effectively to show that $G$ has semistable first homology at $\infty$. If $G=A\ast_CB$, $C$ is finite and $A$ is  not semistable at $\infty$ the next result is used to prove that $G$ is not semistable at $\infty$ (Theorem \ref{Fsplit}).

The term {\it disk pair} in the theorem refers to a pair  $(E,\alpha)$ where $E$ is an open subset of $[0,\infty)\times [0,1]$ homeomorphic to $\mathbb R^2$, $E$ is a union of cells, $\alpha$ is an embedded edge path bounding $E$ and $E$ union $\alpha$ is a closed subspace of $[0,\infty)\times [0,1]$ homeomorphic to a closed ball or a closed half space in $[0,\infty)\times [0,1]$.

\begin{theorem} [Theorem 6.1, \cite{MMGP}] \label{excise} 
Suppose $M:[0,\infty)\times [0,1]\to  X$ is a proper simplicial homotopy rel$\{\ast\}$ of proper edge path rays  $r$ and $s$ into a connected locally finite simplicial 2-complex $X$, where $r$ and $s$ have image in a subcomplex $Y$ of $X$. Say $\mathcal Z=\{Z_i\}_{i=1}^\infty$ is a collection of connected subcomplexes of $Y$ such that only finitely many $Z_i$ intersect any compact subset of $X$.
Assume that each vertex of $X-Y$ is separated from $Y$ by exactly one $Z_i$. 

Then there is an index set $J$ and for each $j\in J$, there is a disk pair $(E_j,\alpha_j)$ in  $[0,\infty)\times [0,1]$ where the $E_j$ are disjoint, $M$ maps $\alpha_j$  to $Z_{i(j)}$ (for some $i(j)\in \{1,2,\ldots\}$) and $M([0,\infty)\times [0,1]-\cup_{j\in J} E_j)\subset Y$. Finally, if $Z_i$ is a finite subcomplex of $X$, then (the proper map) $M$ maps only finitely may $\alpha_j$ to $Z_i$, and each corresponding $E_j$ is bounded in $[0,\infty)\times [0,1]$.
\end{theorem} 
The proof of the next result is basically the proof of Lemma 2 of \cite{M86}.

\begin{lemma} \label{LocFin} 
Suppose $\Lambda$ is a connected, locally finite, infinite graph. Then for each vertex $v\in \Lambda$ there is a proper edge path ray $r_v$ such that for any compact $K\subset \Lambda$, there are only finitely many vertices $v\in\Lambda$ such that the image of $r_v$ intersects $K$. 
\end{lemma}
\begin{proof}
Choose compact sets $C_1,C_2,\ldots$ such that for each $i$, $C_i$ lies in the interior of $C_{i+1}$, $\Lambda-C_i$ is a union of unbounded path connected sets and $\cup_{i=1}^\infty C_i=\Lambda$. Let $v$ be a vertex of $\Lambda$ and choose $N$ such that $v\in C_{N}$, but $v\not\in C_{N-1}$. 
Since $\Lambda-C_{N-1}$ is a union of unbounded path components, there is an edge path $\alpha_1$ in $\Lambda-C_{N-1}$ from $v$ to a vertex $v_1\in \Lambda-C_N$. Similarly there is an edge path $\alpha_2$ in $\Lambda-C_N$ from $v_1$ to a vertex $v_2\in \Lambda-C_{N+1}$. Define $\alpha_n$ inductively and define $r_v=(\alpha_1,\alpha_2,\ldots)$.
\end{proof}

\begin{theorem} \label{SSto1E} 
Suppose $X$ is a connected locally finite infinite CW-complex and any two proper rays in $X$ are properly homotopic. Then $X$ is 1-ended.
\end{theorem}
\begin{proof} 
If $v$ is a vertex of $X$, let $r_v$ be the proper edge path ray at $v$ given by Lemma \ref{LocFin}. Suppose $C$ is compact in $X$. Choose $D$ compact containing $C$ and all vertices $v$ such that $r_v$ intersects $C$. It suffices to show vertices in $X-D$ can be connected by a path in $X-C$. Let $a$ and $b$ be vertices of $X-D$. Let $H:[0,\infty)\times [0,1]\to X$ be a proper homotopy such that $H(t,0)=r_a(t)$ and $H(t,1)=r_b(t)$. Since $H$ is proper, there is an integer $N$ such $H([N,\infty)\times [0,1])\cap C=\emptyset$. The path  $r_a|_{[0,N]}$ followed by $H|_{\{N\}\times[0,1]}$ followed by the inverse of the path $r_b|_{[0,N]}$, is from $a$ to $b$ and avoids $C$. 
\end{proof}

\newpage
\subsection{Semistable Homology at Infinity}\label{SSHom}

For the most part, we are interested in the question of whether or not $H^2(G,\mathbb ZG)$ is free abelian for all finitely presented groups.  We believe it dates back to H. Hopf. This question is connected to the question asking if all finitely presented groups are semistable at $\infty$. In fact, if $G$ is a finitely presented group then $H^2(G,\mathbb ZG)$ is free abelian if and only if the first homology at the end of $G$ is semistable (Corollary \ref{GM2}). 

In Section 13.8 of \cite{G}, an example of a finite aspherical 3-pseudomanifold is constructed. Hence, its fundamental group $G$ has type $F$ and geometric dimension 3, but $H^3(G,\mathbb ZG)$ is isomorphic to $\mathbb Z_2$. This was the first exhibited example of a group of type $F_n$ for which $H^n(G,\mathbb ZG)$ is not free abelian.

\subsubsection{Reducing the $H^2(G,\mathbb ZG)$ Problem to 1-ended Groups} \label{RedH2-1} 

The following result was proved by a geometric argument. It is generalized by Theorem \ref{Reduction}, which has a completely homological proof. Theorem \ref{Reduction} also appears in \cite{MS19}.

\begin{theorem} [Theorem 6, \cite{M87}]\label{Equival1}
If all finitely presented 1-ended groups $K$ are such that $H^2(K,\mathbb ZK)$ is free abelian, then $H^2(G,\mathbb ZG)$ is free abelian for all finitely presented groups.
\end{theorem}

\begin{theorem} [See \S13.7.1 \cite{G}] \label{TorF} 
If $G$ is a finitely presented group and $R$ a ring, then the $R$-module $H^2(G,RG)$ is torsion free. 
\end{theorem}

\begin{theorem}[Chapter VIII, \S5, Exercise 4A,  \cite{Br82}] \label{FreeS} 
Suppose $H$ is a finitely presented subgroup of the finitely presented group $G$, then $H^2(H; \mathbb ZG)$ is isomorphic to the direct sum of copies of $H^2(H,\mathbb ZH)$, one for each coset of $H$ in $G$.

If $H$ and $G$ are only finitely generated, then $H^1(H, \mathbb ZG)$ is isomorphic to the direct sum of copies of $H^1(H,\mathbb ZH)$, one for each coset of $H$ in $G$.
\end{theorem}

For a ring $R$ and any finite group $G$, $H^k(G,R)$ is torsion for $k\geq 1$. Combining this with Theorems \ref{TorF} and \ref{FreeS}:

\begin{theorem} \label{H20} 
If $R$ is a ring and $H$ is a finite subgroup of a finitely presented group $G$ then  $H^2(H,RH)=0$ and $H^2(H,\mathbb ZG)=0$.
\end{theorem}

There are useful cohomological `Mayer-Vietoris' sequences associated to amalgamated products and HNN extensions. We state the amalgamated product version with $\mathbb ZG$ coefficients. 

\begin{theorem} [Theorem 2.10, \cite{Bieri76}] \label{AmalH2}
Let $G = G_1 \ast_S G_2$ $R$ be a ring and $A$ be a right $RG$ module. Then one has a natural long exact sequence:
$$\cdots \to H^k(G,A)\to H^k(G_1,A)\oplus H^k(G_2,A)\to H^k(S, A)\to H^{k+1}(G,A)\to\cdots$$
\end{theorem}

\begin{theorem} [Theorem 2.12, \cite{Bieri76}] \label{HNNH2} 
Let $G = H \ast_S$ be an HNN extension, $R$ be a ring and $A$ be a right $RG$ module. Then one has a long exact sequence:
$$\cdots \to H^{k-1}(S,A)\to H^k(G,A)\to H^k(H,A)\to H^k(S, A)\to\cdots$$
\end{theorem}
 
 Our goal is to reduce the problem about $H^2(G,\mathbb ZG)$ being free abelian, to 1-ended groups $G$. 
 
 \begin{theorem} \label{AmalRed} 
 Suppose the finitely presented group $G$ decomposes as $G_1\ast_SG_2$ where $S$ is finite. Then 
 $$H^2(G,\mathbb ZG)\cong  H^2(G_1,\mathbb ZG)\oplus H^2(G_2,\mathbb ZG).$$
 \end{theorem}
 \begin{proof}
Theorem \ref{AmalH2} implies:  
$$H^1(S,\mathbb ZG)\to H^2(G,\mathbb ZG)\to H^2(G_1,\mathbb ZG)\oplus H^2(G_2,\mathbb ZG)\to H^2(S, \mathbb ZG)$$ Theorem \ref{H20} implies $H^2(S, \mathbb ZG)=0$. Theorem \ref{TorF} implies $H^2(G,\mathbb ZG)$ is torsion free. Since $H^1(S,\mathbb ZG)$ is torsion we have: 
$$H^2(G,\mathbb ZG)\cong  H^2(G_1,\mathbb ZG)\oplus H^2(G_2,\mathbb ZG).$$
 \end{proof}
 
\begin{theorem} \label{HNNRed} 
Suppose $R$ is a ring, $A$ is a right $RG$ module and the finitely presented group $G$ decomposes as $H\ast_S$ where $S$ is finite. Then 
 $$H^2(G,\mathbb ZG)\cong  H^2(H,\mathbb ZG).$$
 \end{theorem}
 \begin{proof}
Theorem \ref{HNNH2} implies the following sequence is exact:
$$H^1(S,\mathbb ZG)\to H^2(G,\mathbb ZG)\to H^2(H,\mathbb ZG)\to H^2(S, \mathbb ZG)$$
Theorem \ref{H20} implies $H^2(S, \mathbb ZG)=0$. Theorem \ref{TorF} implies $H^2(G,\mathbb ZG)$ is torsion free. Since $H^1(S,\mathbb ZG)$ is torsion we have: 
$$H^2(G,\mathbb ZG)\cong  H^2(H,\mathbb ZG).$$
 \end{proof}


\begin{corollary} \label{HDun} 
Suppose $G$ is a finitely presented group and $\mathcal G$ is a finite graph of groups decomposition of $G$ with finite edge groups. If $H_1,\ldots, H_n$ are the infinite vertex groups of $\mathcal G$  then 
$$H^2(G,\mathbb ZG)\cong \oplus_{k=1}^n H^2(H_k;\mathbb ZG).$$ 
\end{corollary}
\begin{proof} 
We may assume $\mathcal G$ is connected since we allow the trivial edge group. The idea is to inductively apply Theorems \ref{AmalRed}, \ref{HNNRed} and \ref{H20} to $G$ and $\mathcal G$ (always keeping $\mathbb ZG$ coefficients).
The theorem is obviously true if the graph has no edges (and so is a single vertex) Assume the theorem is true for graphs with $N$ edges. Take a graph $\mathcal G$ with $N+1$ edges. If an edge $e$ of $\mathcal G$ separates $\mathcal G$, then $G$ is $A\ast_{G_e}B$ where $G_e$ is the edge group of $e$, and $A$ and $B$ are subgroups of $G$ with graphs of groups decompositions $\mathcal G_A$ and $\mathcal G_B$ (respectively) induced by $\mathcal G$. Theorem \ref{AmalRed} implies $H^2(G, \mathbb ZG)\cong H^2(A,\mathbb ZG)\oplus H^2(B,\mathbb ZG)$. Inductively $H^2(A,\mathbb ZG)$ is the direct sum of the groups $H^2(H,\mathbb ZG)$ for all 1-ended vertex groups $H$ of $\mathcal G_A$. Similarly for $B$ - completing the proof in this case.

If $e$ does not separate $G$ then let $A$ be the group with graph of groups decomposition $\mathcal G_A$ given by removing $e$ from $\mathcal G$. Then $G$ is an HNN extension of $A$ over the finite edge group $G_e$. Theorem \ref{HNNRed} implies $H^2(G,\mathbb ZG)\cong H^2(A,\mathbb ZG)$. Inductively $H^2(A,\mathbb ZG)$ is the direct sum of the groups $H^2(H,\mathbb ZG)$ for all 1-ended vertex groups of $\mathcal G_A$. But $H$ is a 1-ended vertex group of $\mathcal G_A$ if and only if it is a 1-ended vertex group of $\mathcal G$. 
\end{proof}



The next result completes our reduction.

\begin{theorem}\label{Reduction} 
Suppose $G$ is a finitely presented group and $\mathcal G$ is a finite graph of groups decomposition of $G$ with finite edge groups. Suppose the infinite vertex groups of $\mathcal G$ are $H_1, \ldots, H_n$. The group $H^2(G,\mathbb ZG)$ is isomorphic to the direct sum $\oplus_{i=1}^nA_i$, where $A_i$ is isomorphic to $\oplus _{[H_i:G]}H^2(H_i,\mathbb ZH_i)$ (and $[H_i:G]$ is the index of $H_i$ in $G$).

In particular, if $\mathcal G$ is a Dunwoody decomposition for $G$ (see Theorem \ref{DunAcc}) then $H^2(G,\mathbb ZG)$ is free abelian if and only if  for each 1-ended vertex group $V$ of $\mathcal G$, $H^2(V,\mathbb ZV)$ is free abelian.
\end{theorem}
\begin{proof}
Simply combine Corollary \ref{HDun} with Theorem \ref{FreeS}.
\end{proof}
 
There is a homological generalization of B. Jackson's  Theorems \ref{JackI} and \ref{JackII}. Corollary \ref{GM2} implies that if a finitely presented group $G$ is simply connected at $\infty$, then $H^2(G,\mathbb ZG)=0$. 

\begin{theorem} \label{JHom} 
Suppose $G=G_1\ast_SG_2$ or $G=G_1\ast_S$ where $G_i$ is finitely presented and 1-ended and $S$ is finitely generated with more than 1 end. Then $H^2(G, \mathbb ZG)$ is non-trivial. 
\end{theorem}
 \begin{proof}
Suppose $H^2(G,\mathbb ZG)=0$.  Theorems \ref{AmalRed} and \ref{HNNRed} gives the exact sequences:
 $$H^1(G_1,\mathbb ZG)\oplus H^1(G_2,\mathbb ZG)\to H^1(S,\mathbb ZG)\to H^2(\mathbb ZG)$$
 $$H^1(G_1;\mathbb ZG)\to H^1(S,\mathbb ZG)\to H^2(G,\mathbb ZG)$$
 
 Theorem \ref{HEnds} implies $H^1(G_i,\mathbb ZG_i)=0$ for $i\in \{1,2\}$ and so Theorem \ref{FreeS} implies $H^1(G_i,\mathbb ZG)=0$. Inserting into the above exact sequences we have (in both cases) that $H^1(S,\mathbb ZG)=0$. Since $S$ has more than 1 end, Theorem \ref{HEnds} implies $H^1(S, \mathbb ZS)$ is non-trivial and Theorem \ref{FreeS} implies $H^1(S,\mathbb ZG)$ is non-trivial. 
 \end{proof}

M. Davis \cite{Davis83} used right angled Coxeter groups to produce closed manifolds with infinite fundamental group (in dimensions $n\geq 4$) whose universal covers were contractible but not simply connected at $\infty$ (and hence not homeomorphic to $\mathbb R^n$). The right angled Coxeter groups that Davis considered correspond to triangulations of compact homology $n$-spheres (with non-trivial fundamental group). His manifolds have fundamental group, a torsion free subgroup of finite index in this Coxeter group. Davis shows that the fundamental pro-group at infinity of the universal cover of his compact manifolds are inverse systems of free products of the fundamental group of the homology sphere.  Hence, when $G$ is the fundamental group of such a manifold, $G$ is not simply connected at $\infty$. But the fundamental group of a homology sphere abelianize to the trivial group and so the first pro-homology of $G$ is trivial. This implies $H^2(G,\mathbb ZG)=0$, and so Theorem \ref{JHom} restricts the type of splittings $G$ may have. Since $G$ is a subgroup of finite index in a  Coxeter group, there are many ``visual" splittings of such a group.

\subsubsection{Equivalent Versions of Homological Semistability}\label{EVHS}

When coefficients are omitted in homology, they  are ordinary $\mathbb Z$-coefficients. We use $\bar H_n$ for reduced homology. The universal cover of a CW-complex $X$ is $\tilde X$ and $\tilde X^n$ is the $n$-skeleton of the universal cover of $X$.

\begin{definition} A group $G$ with $K(G,1)$-complex with finite $n$-skeleton is said to be {\it of type} $F(n)$. When $G$ is finitely presented, this is equivalent to the algebraic finiteness property $FP_n$. 
\end{definition}

\begin{definition}\label{ProH} 
Let $X$ be a connected locally finite CW complex, and let $C_i$ be a cofinal sequence of compact sets in $X$. 
The inverse sequence of abelian groups:
$$(\ast){\hskip .4in} H_n(X-C_1)\leftarrow H_n(X-C_2) \leftarrow\cdots$$
(all bonding maps are induced by inclusion)
is called the $n^{th}$ {\it pro-homology group of the end of $X$} and is denoted $H_n(\varepsilon X)$. It is {\it semistable} if it is pro-isomorphic to an inverse sequence of groups with epimorphisms as bonding maps. Similarly for reduced homology. 
\end{definition}

\begin{remark} \label{GSS} 
Geometrically, $H_1(X-C_1)\leftarrow H_1(X-C_2) \leftarrow\cdots$ is semistable if and only if given a compact set $C$ there is a compact set $D(C)$ such that for any third compact set $E$ and loop $\alpha$ in $X-D(C)$, $\alpha$ is homologous in $X-C$ to a sum of loops $\Sigma_{i=1}^n\alpha_i$ in $X-E$. If $X$ is 1-ended and $X-E$ is path connected, then $\alpha$ is homologous in $X-C$ to a loop in $X-E$. (Simply take a point $e\in X-E$ and path $\tau_i$ in $E$ from the initial point of each  $\alpha_i$ to $e$. Since $\alpha_i$ is homologous in $X-E$ to $(\tau_i^{-1},\alpha_i, \tau_i)$, $\alpha$ is homologous in $X-C$ to $\Pi_{i=1}^n(\tau_i^{-1},\alpha_i, \tau_i)$.)
\end{remark}
\begin{definition} 
Let $X$ be a topological space. If $r,s:([0,\infty),\{0\})\to (X,\ast)$ are proper rays then we say that $r$ is {\it properly homologous} to $s$ (in $X$) if there is a 2-manifold $N$ with boundary $L(t)$ (for $t\in (-\infty,\infty)$) homeomorphic to a real line and a proper map $p:N\to X$ such that $p(L(t))=r(t)$ for $t\geq 0$ and $p(L(t))=s(t)$ for $t\leq 0$.
\end{definition}

\begin{lemma} If $N$ is a 2-manifold with real line boundary $L$, then the two ends of $L$ determine the same end of $N$. In particular, if $r$ and $s$ are properly homologous in $X$, then they determine the same end of $X$. 
\end{lemma}
\begin{proof}
Suppose $K$ is a compact subset of $N$ that separates the two ends of $L$. Assume that $k>0$ such $L(-k)$ cannot be joined to $L(k)$ by a path in $N-K$ and so that $L((-\infty,-k] \cup [k,\infty))\subset N-K$. Let $M$ be a compact submanifold of $N$ containing $K\cup L([-k,k])$ in its topological interior. The boundary of $M$ is a disjoint union of circles. One of these circles $S$ contains $L([-k,k])$ so that $S=L((-k,k))\cup \alpha$ where $\alpha$ is homeomorphic to a closed   interval that begins at $L(-k)$ and ends at $L(k)$. But $\alpha$ has image in $N-K$, contrary to our assumption.
\end{proof} 

\begin{theorem} [Proposition 12.8, \cite{GH81}] Suppose $f:S^1\to X$ is continuous. Then $f$ is homologicially trivial in $X$ if and only if $f$ extends to $f:W\to X$ where $W$ is an orientable surface with boundary $S^1$.
\end{theorem} 

\begin{definition} The locally compact connected ANR $X$ has {\it semistable first homology at $\infty$} if there is a compact set $C$ such that for any point $\ast \in X-C$, any two proper rays $r,s:([0,\infty),\{0\}) \to (X-C,\ast)$ converging to the same end of $X$ are properly homologous. 
\end{definition} 

\begin{remark} 
Note that if $\ast_1$ and $\ast_2$ are points in the same (unbounded) path component of $X-C$ then this definition is satisfied for $\ast_1$ if and only if it is satisfied for $\ast_2$. 
\end{remark}

\begin{remark}\label{E1C} 
Note that if $H_1(X)=0$ and $X$ has semistable first homology at $\infty$ then any two proper rays $r,s:([0,\infty),\{0\}) \to (X,\ast)$ converging to the same end of $X$ are properly homologous. (Simply take a path $\alpha$ in $X-C$ connecting say $r(k)$ to $s(k)$, where $r([k,\infty)$ and $s([k,\infty)$ avoid $C$. Then $r|_{[k,\infty)}$ is properly homologous to $(\alpha, s|_{[k,\infty)})$. Since $H_1(X)=0$, $\alpha$ is homologous to the path that follows $r$ from $r(k)$ to $r(0)=s(0)$ and then follows $s$ from $s(0)$ to $s(k)$. Combining we have $r$ is properly homologous to $s$.) 
\end{remark}

The next result is the homological version of Theorem \ref{SSloop}.

\begin{theorem} \label{SSloop2}
Let $K$ be a connected, 1-ended locally finite CW-complex. Suppose $r$ is a proper ray in $K$ and $H_1(\varepsilon K)$ is semistable (Definition \ref{ProH}).  Then for any compact set $C\subset K$ there is a compact set $D_1(C)\subset K$ such that for any loop  $\alpha$ based at $r(x)$ and with image in $K-D_1(C)$, the proper rays $r|_{[x,\infty)}$ and $(\alpha, r|_{[x,\infty)})$ are properly homologous in $K-C$. Furthermore, there is a collection of compact 2-manifolds, each with circle boundary, that can be stacked (as in Figure \ref{FigSSloop2}) to build the 2-manifold $M$ with $\mathbb R^1$ boundary and a proper map $F:M\to K$ showing $r|_{[x,\infty)}$ and $(\alpha, r|_{[x,\infty)})$ are properly homologous in $K-C$.
\end{theorem}

\begin{proof} 
Let $D_1$ be compact containing $D(C)$ and also large enough so that if $r(x)\in K-D_1$, then $r([x,\infty))\subset K-C$.  Assume that $E$ is compact containing $D_1$ and that $K-E$ is path connected. Now let $\alpha$ be a loop based at $r(x)$ and with image in $K-D_1$.  By Remark \ref{GSS},  $\alpha$ is 
homologous in $K-C$ to a loop $\beta_i$ with image in $K-E$. Choose $y>x$ such that $r([y,\infty))\subset K-E$. If $\tau$ is a path in $K-E$ from $r(y)$ to $\beta(0)$ then $\beta$ is homologous to $(\tau,\beta,\tau^{-1})$ in $K-E$. 
This implies that the loop $\alpha$ is homologous to the loop $\gamma= (\tau,\beta,\tau^{-1})$ in $K-C$. Since   $\gamma$ is homologous to $(r|_{[x,y]}, \gamma, r|_{[x,y]}^{-1})$ in $K-C$, the loop $\rho= (\alpha, r|_{[x,y]}, \gamma^{-1}, r|_{[x,y]}^{-1})$ is homologicially trivial in $K-C$. There is a 2-manifold $M$ with circle boundary $S$ and map $H:M\to K-C$ such that $F|_S=\rho$. 
 
\begin{figure}
\vbox to 3in{\vspace {-2in} \hspace {-.5in}
\hspace{-1 in}
\includegraphics[scale=1]{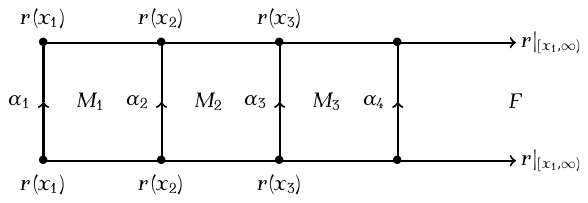}
\vss }
\vspace{-2.1in}
\caption{Stacking 2-Manifolds} 
\label{FigSSloop2}
\end{figure}

A simple stacking argument will finish the proof of the theorem (see Figure \ref{FigSSloop2}).
Let $C$ be a compact subset of $K$ and $r$ a proper ray in $K$. Let $C= C_0, C_1,\ldots$ be a cofinal sequence of compact subsets in $K$ such that $D(C_i)\subset C_{i+1}$ for all $i\geq 0$.  Also choose $C_i$ such that for all $i\geq 1$, if $r(x)\in K-C_i$, then $r([x,\infty))\subset K-C_{i-1}$. Let $x=x_1$ be such that $r(x_1)\in K-C_1$. For $i\geq  1$, let $x_i>0$ be a strictly increasing sequence of integers such $r(x_i)\in K-C_i$. Suppose $\alpha=\alpha_1$ is a loop in $K-C_1$ based at $r(x)=r(x_1)$. We have shown there is a 2-manifold $M_1$ with circle boundary $S_1$ and $F_1:M_1\to X-C_1$ such that $F_1|_{S_1}$ is the loop $(\alpha_1, r|_{[x_1,x_2]}, \alpha_2^{-1}, r|_{[x_1,x_2]}^{-1})$ where $\alpha_2$ has image in $K-C_2$. Inductively we have a 2-manifold $M_i$ for $i\geq 1$ with circle boundary $S_i$ and $F_i:M_i\to X-C_i$ such that $F_i|_{S_i}$ is the loop $(\alpha_i, r|_{[x_i,x_{i+1}]}, \alpha_{i+1}^{-1}, r|_{[x_i,x_{i+1}]}^{-1})$ and $\alpha_{i+1}$ has image in $K-C_{i+1}$. Stack these homotopies as in Figure \ref{FigSSloop2} to complete the proof.
\end{proof}

Theorem \ref{lim1} implies:

\begin{theorem}
If $X$ is a connected locally finite CW complex then $H_1(\varepsilon X)$ is semistable if and only if $\varprojlim ^1 H_1(\varepsilon X)$ is trivial. 
\end{theorem}

The next result should be compared to Theorem \ref{ssequiv}. The proof is a simple variation of that for Theorem \ref{ssequiv}. For a more general version (any number of ends) see Theorems 16.1.8, 16.1.21 and 12.5.10 of \cite{G}.

\begin{theorem}\label{ssequivH} 
Suppose $K$ is a locally finite, connected and 1-ended CW-complex. Then the following are equivalent:
\begin{enumerate}
\item $K$ has  semistable first homology at $\infty$.
\item  The first pro-homology group of the end of $K$ ($H_1(\varepsilon K)$) is semistable (pro-isomorphic to an inverse system of groups with epimorphic bonding maps). 
\item For any compact set $C$, there is a compact set $D$ such that for any third compact set $E$ and loop $\alpha$  with image in $K-D$, $\alpha$ is homologous in $K-C$ to a loop in $K-E$. Equivalently, there is a compact 2-manifold $N$ and a map $m:N\to K-C$ such that $m$ takes one boundary component of $N$ to $\alpha$ and all others have image in $K-E$. 
\item For any compact set $C$ there is a compact set $D$ such that if $r$ and $s$ are proper rays based at $v$ and with image in $K-D$, then $r$ and $s$ are properly homologous in $K-C$. 
\end{enumerate}

If $H_1(K)=\{0\}$, then a fifth equivalent condition can be added to this list:

\medskip

5. If $r$ and $s$ are proper rays based at $v$, then $r$ and $s$ are properly homologous.   
\end{theorem}

\begin{proof} ({\it 2} implies {\it 4}). 
 Suppose  the inverse sequence of groups 
 $$(\ast) \ \ \ \ \ \ \ \ H_1(K-C_0){\buildrel p_1\over \longleftarrow} H_1(K-C_1) {\buildrel p_2\over \longleftarrow}\cdots$$ 
 is semistable. Theorem \ref{SSloop2} implies that  for  any  compact set  $C\subset K$ there is a compact set $D(C)\subset K$, such that  for any loop, $\alpha$,  in $K - D$ (with $\alpha$  based at  $r(x)$),  the proper rays $r|_{[x,\infty)}$ and $(\alpha,r|_{[x,\infty)})$ are properly homologous in $K - C$. 
Without loss  assume that  $D(C_i)\subset C_{i+1}$ for all $i$. So that 

\medskip

\noindent ($\ast$) If $\alpha$ is a loop in  $K - C_i$, based at  $r(x)$,  then $r|_{[x,\infty)}$ is properly homologous to $(\alpha, r|_{[x,\infty)})$ in $K- C_{i-1}$. 

\medskip

Theorem \ref{ES} allows us to assume that $K-C_i$ is connected. Let $C$ be any compact subset of $X$. We assume $C\subset C_0$ and show any two proper rays in $K-C_1$ that are based at a vertex $v$ are properly homologous in $K-C_0$. First we let $v=r(x_1)$.  Let  $s:([0, \infty),\{0\}) \to (K-C_1, \{r(x_1)\})$ be  a  proper ray. 
  
We show $s$ and  $r|_{[x_1,\infty)}$ are properly homologous in $K-C_0$. Since $s$ is arbitrary, this implies any two proper rays based at $r(x_1)$ and with image in $K-C_1$ are properly homologous in $K-C_0$. For $i\geq 1$, let $\beta_i=r|_{[x_i,x_{i+1}]}$. Let $a_1=0$ (so that $r(x_1)=s(0)=s(a_1)$) and for $i>1$ choose $a_i\in [0, \infty)$ such that  $s([a_i, \infty)) \subset  K-C_i$. Without loss assume that  $a_i <  a_{i+1}$ for  all $i$. Let $\alpha_i=s|_{[a_{i},a_{i+1}]}$.   Let  $\gamma_1$ be the constant path at $r(x_1)$. For $i>1$, let $\gamma_i: [0,1] \to  K-C_i$, such that  $\gamma_i(0) = s(a_i)$ and  $\gamma_i(1) = r(x_i)$ (See Figure \ref{Fig231ac}).   
  
\begin{figure}
\vbox to 3in{\vspace {-2in} \hspace {.5in}
\hspace{-1 in}
\includegraphics[scale=1]{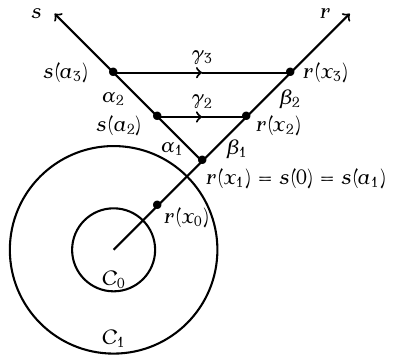}
\vss }
\vspace{-1in}
\caption{The case $v=r(x_1)$} 
\label{Fig231ac}
\end{figure}

\medskip

For $i\geq 1$, let $\delta_i$  be the loop $(\gamma_{i}^{-1},\alpha_i,\gamma_{i+1},\beta_i^{-1})$. It is easy to  show that  $s$, which we represent by $(\alpha_1,\alpha_2,\ldots)$, is properly homotopic to the map of $[0, \infty) \to K$ represented by  $(\delta_1,\beta_1,\delta_2,\beta_2,\ldots)$ (simply eliminate backtracking). For $i\geq 1$, ($\ast$) implies that $r|_{[x_i,\infty)}$ is properly homologous to $(\delta_i,r|_{[x_i,\infty)})$ in $K-C_{i-1}$. 

This means there is a 2-manifold $M_i$ with boundary equal to a line $L_i$ and proper map $F_i:M_i\to K-C_{i-1}$ such that on  one end of $L_i$, $F_i$ is $(\beta_i, \beta_{i+1}, \ldots)=r|_{[x_i,\infty)}$ and on the other end of $L_i$, $F_i$ is $(\delta_i, \beta_i,\beta_{i+1},\ldots)$. Combining the $F_i$ as in Figure \ref{Fig231bc}, produces a 2 manifold  $M$ with boundary a line $L$  and map $F:M\to K-C_0$, such that on one end of $L$, $F$ is $r=(\beta_1,\beta_2,\ldots )$ and on the other end of $L$, $F$ is $(\delta_1,\beta_1,\delta_2,\beta_2, \ldots)$.

\begin{figure}
\vbox to 3in{\vspace {-2in} \hspace {.5in}
\hspace{-1 in}
\includegraphics[scale=1]{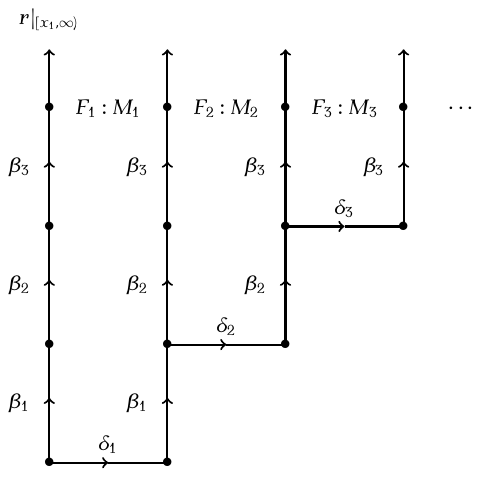}
\vss }
\vspace{-.1in}
\caption{Combining the $F_i$} 
\label{Fig231bc}
\end{figure}

\medskip

Next we show that $F$ is proper. Suppose $A$ is compact in $K$. Choose $n$ such that $A\subset C_n$. Then for $i\geq n+1$, the image of $F_i$ avoids $C_n$ (and hence avoids $A$). So $F^{-1}(A)=\cup_{i=1}^nF_i^{-1}(A)$ (a finite union of compact sets). So $F$ is proper. 

We have $r|_{[x_1,\infty)}$ is properly homologous to $(\delta_1, \beta_1,\delta_2,\beta_2,\ldots)$ in $K-C_0$ which is properly homologous to $s$ in $K-C_0$. As $s$ is an arbitrary proper ray in $K-C_1$ we have any two proper rays based at $r(x_1)$ and with image in $K-C_1$ are properly homologous in $K-C_0$. 

Next suppose $s_1$ and $s_2$ are proper rays based at the vertex $v$ and with image in $K-C_1$. Let $\tau$ be a path in $X-C_1$ from $r(x_1)$ to $v$. Then the proper rays $(\tau, s_1)$ and $(\tau,s_2)$ are properly homologous in $X-C_0$. This implies $s_1$ is properly homologous to $s_2$ in $K-C_0$. 
\end{proof} 

It is straightforward to see that  {\it 2} is equivalent to {\it 3}, and {\it 4} implies {\it 1}. In order to show the equivalence of the first four conditions, it is enough to prove {\it 1} implies {\it 3}. 

\begin{proof} ({\it 1} implies {\it 3})



 Let  $G_1  {\buildrel \phi_1 \over \leftarrow} G_2{\buildrel \phi_2\over \leftarrow}\cdots$ be  an  inverse sequence of  groups. Recall $\varprojlim ^1 \{G_n\}$ is  the  pointed set  of  equivalent classes under the  equivalence relation on $\Pi_{n>0}G_n$  defined by  $\langle x_n\rangle \sim \langle y_n\rangle$ if  there is  $\langle g_n\rangle$ such that  $\langle y_n\rangle= \langle g_nx_n\phi_n(g_{n+1}^{-1})\rangle$. Theorem \ref{lim1} implies that if each $G_n$ is countable, then $\varprojlim^1\{G_n\}$ is trivial if and only if $\{G_n\}$ is semistable. 

\begin{figure}
\vbox to 3in{\vspace {-2in} \hspace {.5in}
\hspace{-1 in}
\includegraphics[scale=1]{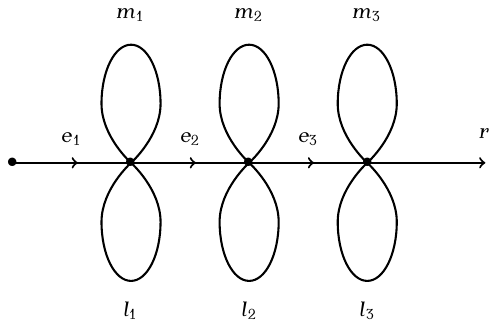}
\vss }
\vspace{-.3in}
\caption{Equivalent elements} 
\label{Fig231cc}
\end{figure}

Let $\{C_i\}_{i=1}^\infty$ be a cofinal sequence of compact sets in $K$, where $C_1$ is the compact set $C$ of condition 1. Our goal is to show that if  all proper maps $[0, \infty) \to  K-C_1$  with the same initial point are  properly homologous, then $\varprojlim H_1(\varepsilon K)$ is  trivial. Let $r$ be a proper ray in $X$. For simplicity assume we have reparametrized $r$ so that $r([n,\infty))\subset K-C_n$ for all $n\geq 1$.   Represent $r$ by $(e_1,  e_2,\ldots)$  where  $e_n(t) =  r(n-1  +  t)$ for $t\in [0,1]$.   Choose $\langle [m_n']\rangle$ and  $\langle [l_n']\rangle$  in  $\Pi_{n>0}H_1(K-C_n)$. Each $m_i$ and $l_i$ is homologous to a loop $m_i$ and $l_i$ respectively at $e_i(1)$ (see  Figure \ref{Fig231cc}).  We  show  $\langle [m_n]\rangle \sim\langle [l_n]\rangle$. 

Since all  proper maps $[0, \infty)\to K-C$  are  properly homologous there is a  2-manifold $M$ with boundary a line $L$ and a proper map $F:  M \to K$  such that on one end of $L$, $F$ is  $(e_1,  m_1,  e_2,m_2,\ldots)$ and on the other end of $L$, $F$ is $(e_1, l_1, e_2, l_2,\ldots)$. Let $M_i$ be a compact submanifold of $M$ such that $F^{-1} (C_i)\cup M_{i-1}$ is a subset of the topological interior of  $M_i$. Now the boundary of $M_i$ is a sum of circles. One is $S_i$ which contains part of the line $L$ mapped by $F$ to $(e_1,m_1,e_2,m_2,\ldots, m_{n(i)}, g_i, l_{n'(i)}^{-1}, \ldots, l_2^{-1}, e_2^{-1}, l_1^{-1} e_1^{-1})$ (where $g_i$ is a path in $K-C_i$) and the sum of the rest are say $Q_i$, where each term of $Q_i$ has image under $F$ in $K-C_i$. 
Let 
$$h_i =(m_i,e_{i+1},m_{i+1},e_{i+2},\cdots ,m_{n(i)},g_i,l_{n'(i)}^{-1},\cdots ,e_{i+2}^{-1},l_{i+1}^{-1},e_{i+1}^{-1}, l_i^{-1})+Q_i.$$

\begin{figure}
\vbox to 3in{\vspace {-2in} \hspace {-1in}
\hspace{-1 in}
\includegraphics[scale=1]{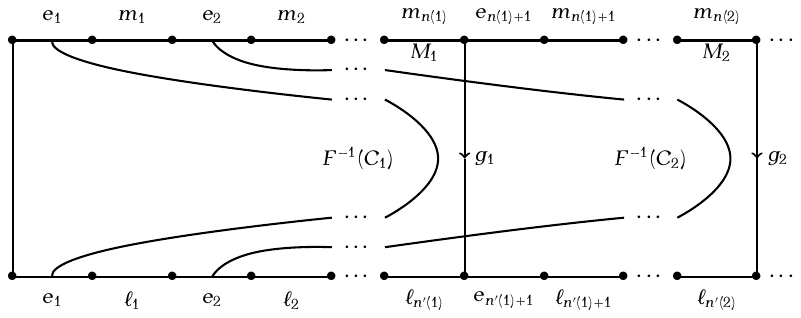}
\vss }
\vspace{-1.2in}
\caption{Avoiding compact sets} 
\label{Fig231dc}
\end{figure}

See  Figure \ref{Fig231dc}  representing $M$. By definition, the  bonding maps $\phi_i$ of $H_1(\varepsilon K)$ are  such that  $\phi_i([h_{i+1}])=[(e_{i+1},h_{i+1},e_{i+1}^{-1})]=[h_{i+1}]$.  It remains to show that  $m_i$, is homologous to $(h_i,l_i,\phi_i(h_{i+1})^{-1})$  in $K-C_i$.  We  show $m_1$ is homologous to $(h_1, l_1, \phi_1(h_2)^{-1})$ in  $K -  C_1$.  The general case is completely analogous. 

$$(h_1,l_1, \phi_1(h_2)^{-1})=(m_1, e_2, m_2,\cdots, m_{n(1)}, g_1, l_{n'(1)}^{-1},\cdots, l_2^{-1},e_2^{-1},l_1^{-1},$$
$$l_1,e_2, l_2, e_3, l_3, \cdots l_{n'(2)}, g_2^{-1}, m_{n(2)}^{-1},\cdots, m_3^{-1},e_3^{-1},m_2^{-1},e_2^{-1})+Q_1-Q_2$$

Eliminate edges and  their inverses, and  the  chain:
$$(g_1,e_{n'(1)+1}, l_{n'(1)+1},\cdots, l_{n'(2)}, g_2^{-1}, m_{n(2)}^{-1},\cdots, m_{n(1)+1}^{-1},e_{n(1)+1}^{-1})+Q_1+Q_2$$
(which bounds a closed submanifold of $M_2$ in  $K-C_1$ - see Figure \ref{Fig231dc}).  What remains is $m_1$, and  Figure \ref{Fig231dc} shows the induced maps takes place in $K-C_1$
\end{proof}

\begin{proof} (If $H_1(K)$ is trivial, then {\it 1} is equivalent to {\it 5}.) 
Certainly {\it 5} implies {\it 1}. 

Suppose {\it 1}. Let $r$ and $s$ be proper rays at $v$. Choose $n$ large enough such that there is a path $\alpha$ from $r(n)$ to $s(n)$ in $K-C$. By hypothesis, $r|_{[n,\infty)}$ is properly homologous to $(\alpha, s|_{[n,\infty)})$. Since $H_1(K)$ is trivial, the loop $(r|_{[0,n]}, \alpha, s|_{[0,n]}^{-1})$ is homologicially trivial. Combining, we have $r$ is properly homologous to $s$. 
\end{proof} 
\noindent This completes the proof of Theorem \ref{ssequivH}.


\subsubsection{Groups, 1-acyclicity at Infinity and Pro-finite First Homology at $\infty$} \label{ac1}

The space $X$ is {\it 1-acyclic at infinity}, has {\it pro-finite first homology at infinity}, or has {\it pro-torsion first homology at infinity}  if for any compact set $C$ there is a compact set $D$ such that the image of $H_1(X-D)$ in $H_1(X-C)$ under the homomorphism induced by the inclusion of $X-D$ into $X-C$ is respectively trivial, finite or torsion.  
(A loop $\alpha$ in $X-C$ is homologicially trivial if there is an orientable 2-manifold $M$ bounded by an embedded loop $\alpha'$ and a continuous map $H:M\to X-C$ such that $H$ restricted to $\alpha'$ maps (in the obvious way) to $\alpha$ (see Proposition 12.8, \cite{GH81}).)

\medskip

\noindent {\bf Note.} If a locally finite complex $Y$ is simply connected then the Mayer-Vietoris sequence for homology implies that the first homology at infinity of $Y$ is pro-finitely generated. In particular, the finitely presented group $G$ has pro-finite first homology at infinity if and only if it has pro-torsion first homology at infinity. 

\medskip

A finitely presented group $G$ is {\it 1-acyclic at infinity}, (has {\it pro-finite first homology at $\infty$}) if for some (equivalently any) finite complex $X$ with $\pi_1(X)=G$ the universal cover of $X$ is 1-acyclic at infinity (has pro-finite first homology at $\infty$). 
 
One of the main theorems of \cite{MS24} describes a splitting condition ensuring that a finitely presented 1-ended group $G$ does not have pro-finite first homology at infinity (equivalently $H^2(G,\mathbb ZG)\ne 0$ - see  Corollary \ref{GM2}). Such groups are not duality groups of dimension $\geq 3$. This result follows directly from the following technical result:

\begin{theorem} [Theorem 1.4, \cite{MS24}] \label{Acyclic}
Suppose $X$ is a locally finite CW-complex, and $X_1$ and $X_2$ are connected 1-ended subcomplexes of $X$ such that $X_1\cup X_2=X$. If $K$ is a finite subcomplex of $X$ (possibly empty) such that $(X_1\cap X_2)-K$ has more than one unbounded component, then $X$ does not have pro-finite first homology at infinity. 
\end{theorem} 
As a direct corollary to this result we have: 
\begin{corollary} [Corollary 1.5, \cite{MS24}]
Suppose the 1-ended finitely presented group $G$ splits non-trivially as $A\ast_CB$  or HNN extension $A\ast_C$ where $A$ and $B$ are finitely presented, and $C$ has more than one end.
If the two halfspaces associated to $C$ are 1-ended then $H^2(G,\mathbb ZG)\ne \{0\}$. In particular, $G$ is not simply connected at $\infty$.
\end{corollary}

\begin{corollary} [Corollary, \cite{MS24}] If $G$ is a finitely presented 1-ended group that splits non-trivially over a virtually free group then $H^2(G,\mathbb ZG)\ne \{0\}$. 
\end{corollary}

The proof of Theorem \ref{Acyclic} easily extends to more general spaces than locally finite CW-complexes. We state a more general result and include a proof that is basically the same as the one for Theorem \ref{Acyclic} in \cite{MS24}. First recall:

\medskip

\noindent {\bf Theorem} \ref{ES} {\it Suppose $X$ is connected, locally compact, locally connected
and Hausdorff and $C$ is compact in $X$,
then $C$ union all bounded components of $X - C$ is compact, and $X-C$ has only finitely many unbounded components.
Here bounded means compact closure.}

\begin{theorem}
	Suppose the connected Hausdorff space $X$ is the union of two closed subspaces $X_1$ and $X_2$ such that:
	
	1)  The subspaces $X_1$ and $X_2$ are path connected, locally path connected, locally compact and  1-ended. 
	
	2) The (closed) subspace $X_0=X_1\cap X_2$ is locally connected. 
	
	3) There is compact subset $K$ of $X_0$ (possibly empty) such that for any compact set $E\subset X_0$ containing $K$, $X_0-E$ contains points $a$ and $b$ in distinct components of $X_0-K$. 
	
	Then $X$ does not have pro-torsion first homology at infinity.
\end{theorem}
\begin{proof}
In a Hausdorff space compact sets are closed.
Assume that $X$ has pro-torsion first homology at infinity. 
Let $D$ be a compact subset of $X$ containing $K$ such that for any loop $\tau$ in $X-D$ there is $n>0$ such that $\tau^n$ is  homologicially trivial in $X-K$. For $i\in\{1,2\}$ let $E_i$ equal the union of  $D\cap X_i$ and  all bounded components of $X_i-D$. Theorem \ref{ES} implies that $E_i$ is compact in $X_i$ (and hence in $X$) and (since $X_i$ is 1-ended and locally path connected) the open subset $X_i-E_i$ of $X_i$ is path connected and unbounded. Let $E=E_1\cup E_2$. Then $E$ is a compact subset of $X$ containing $D$ such that any two points of $X_0-E$ can be joined by a path in $X_1-E_1$ and by a path in $X_2-E_2$. 
	By {\it 3)} there are points $v$ and $w$ in $X_0-E$ such that $v\in V$ and $w\in W$ where $V$ and $W$ are distinct components of $X_0-K$. Let $\alpha$ be a path in $X_1-E_1$ from $v$ to $w$ and let $\beta$ be a path in $X_2-E_2$ from $w$ to $v$. In particular, $\alpha$ and $\beta$ avoid $D$.
	There is an integer $n>0$ such that the loop $(\alpha, \beta)^n$ is homologicially trivial in $X-K$. Let $M$ be an orientable 2-manifold bounded by an embedded loop $(\alpha'_1,\beta_1',\ldots, \alpha_n',\beta_n')$ and $H:M\to X-K$ such that $H$ restricted to $\alpha_i'$ and $\beta_i'$ maps (in the obvious way) to the paths $\alpha$ and $\beta$ respectively. Let $v_i'$ and $w_i'$ be the initial and terminal points of $\alpha_i'$ (See Figure \ref{Fig1}).
	
	\begin{figure}
		\vbox to 3in{\vspace {-2in} \hspace {-.3in}
			\hspace{-1 in}
			\includegraphics[scale=1]{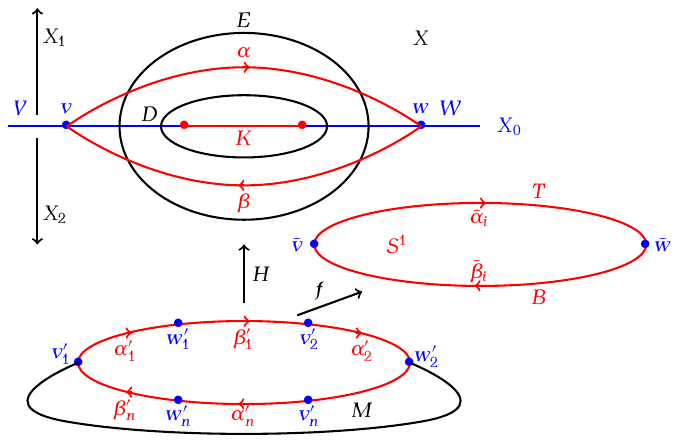}
			\vss }
		\vspace{-.5in}
		\caption{The Relevant Maps}
		\label{Fig1}
	\end{figure}

	\begin{lemma} \label{VW} The sets $H^{-1}(W)$ and $H^{-1}(X_0-W)$ are disjoint closed subsets of $M$ containing $\{w_1',\ldots   w_m'\}'$ and $v_1',\ldots, v_n'\}$ respectively. 
	\end{lemma}
	\begin{proof}
		First note that $H^{-1}(X_0)$ is a closed subset of $M$.  As $X_0$ is locally connected, each component of $X_0-K$ is open in $X_0$. In particular, $X_0-W$ is closed in $X_0$ and hence closed in $X$. Then $H^{-1}(X_0-W)$ is closed in $M$. The set $W\cup K$ is closed in $X_0$ and hence in $X$. Since the range of $H$ is $X-K$, the closed set $H^{-1}(W\cup K)$ is equal to $H^{-1}(W)$.	
\end{proof} 
	
	Consider the circle $S^1$ with diametrically opposite vertices $\bar v$ and $\bar w$. Let $T$ and $B$ be the two closed arcs in $S^1$ such that $T\cup B=S^1$ and $T\cap B=\{\bar v,\bar w\}$. Map the (disjoint) closed sets $H^{-1}(W)$ and $H^{-1} (X_0-W)$ in $H^{-1}(X_1)$ to $\{\bar w\}$ and $\{\bar v\}$ respectively. Extend this map by Tietze's extension theorem to a continuous map $f_1:H^{-1}(X_1)\to T$. Similarly define $f_2:H^{-1}(X_2)\to B$ (with $f_2(H^{-1}(W))=\{\bar w\}$ and $f_2(H^{-1}(X_0-W))=\{\bar v\}$). Since $f_1$ and $f_2$ agree on the closed set $H^{-1}(X_1)\cap H^{-1} (X_2)=H^{-1}(X_0)= H^{-1}(W)\cup H^{-1} (X_0-W)$, we have a continuous function $f:M\to S^1$ that agrees with $f_1$ and $f_2$ on $H^{-1}(X_1)$ and $H^{-1}(X_2)$ respectively. Note that for each $i$, $f\alpha_i'=f_1\alpha_i'=\bar \alpha_i$ is a path in $T$ from $\bar v$ to $\bar w$ and $f\beta_i'=f_2\beta_i'=\bar \beta_i$ is a path in $B$ from $\bar w$ to $\bar v$. Furthermore, for each $i$, the class of loop $(\bar \alpha_i, \bar \beta_i)$ at $\bar v$ generates $\pi_1(S_1,\bar v)=\mathbb Z$. The contradiction arises since the class of the loop $(\alpha_1', \beta_1',\ldots, \alpha_n',\beta_n')$ is a commutator in $\pi_1(M, v')$ and so maps to the trivial element under the homomorphism $f_\ast:\pi_1(M,v')\to \pi_1(S^1,\bar v)$. But instead this class is mapped by $f_\ast$ to  the class of $(\bar \alpha_1,\bar \beta_1,\ldots, \bar \alpha_n, \bar \beta_n)$, a non-trivial element of $\pi_1(S^1,\bar v)$. 
\end{proof}

We conclude this section with the following:

Theorem B of \cite{BrM01}] classifies the right angled Artin groups that are $m$-connected ($m$-acyclic) at infinity. (see \S \ref{AGCG})
 
\subsubsection{One-point Compactifications and Semistability}\label{OnePt}

The results in this section are partially contributed by G. Conner and C. Guilbault. We are interested in 1-ended groups/spaces in this section. If $X$ is a locally finite, connected 1-ended complex, let $X^{\ast}$ be its 1-point compactification.  Recall that if $Y$ is a connected finite complex, then $\tilde Y$ is the universal cover of $Y$. 
The proofs of the following lemmas are straightforward. 

\begin{lemma}\label{OnePt2E}
Suppose $X$ and $Y$ are non-compact, locally finite and connected complexes. If $X$ and $Y$ are proper 2-equivalent then $\pi_1(X^\ast)$ is isomorphic to $\pi_1(Y^\ast)$ and $X^\ast$ is locally simply connected if and only if $Y^\ast$ is locally simply connected. (Note that $X$ and $Y$ are locally contractible.)
\end{lemma}

Combining Lemma \ref{OnePt2E} with Theorem \ref{P1E} we have:
\begin{theorem}
Suppose $X$ and $Y$ are connected finite complexes with isomorphic fundamental groups $G$. If $G$ is 1-ended and the 1-point compactification of $\tilde X$ is simply connected (locally simply connected)  then the 1-point compactification of $Y$ is simply connected (locally simply connected).
\end{theorem}

\begin{theorem} Suppose $X$ is a non-compact, locally finite and connected 1-ended complex. If $X$ has semistable fundamental group at $\infty$ then $X^{\ast}$ is simply connected and locally simply connected.
\end{theorem}
\begin{proof}
We use part 4 of Theorem \ref{ssequiv} for our definition of semistability. There are finite subcomplexes $C_1,C_2,\ldots$ such that $\cup_{i=1}^\infty C_i=X$ where $C_i$ is a subset of the interior of $C_{i+1}$ and any two proper rays $r,s$ in $X-C_i$ and based at $v$ are properly homotopic rel$\{v\}$ by a homotopy in $X-C_{i-1}$.  Let $p$ be the additional point of the 1-point compactification of $X$. Let $\alpha:([0,1],\{0,1\})\to (X,\{p\})$ be a loop. If $(a,b)$ is a component of $[0,1]-\alpha^{-1}(p)$ then note that $\alpha(a)=\alpha(b)=p$. If $\alpha((a,b))$ has image in $X-C_i$ there there is a proper homotopy in $X-C_{i-1}$ between the two proper rays determined by $(a,b)$. This proper homotopy defines a null homotopy of the loop, $\alpha$ restricted to $[a,b]$. combining homotopies (possibly infinitely many) for each component of $[0,1]-\alpha^{-1}(p)$ defines a null homotopy for $\alpha$.  In particular, $X^{\ast}$ is simply connected. Note that the same argument implies that $X^{\ast}$ is locally simply connected at $p$. As $X$ is locally contractible, $X^{\ast}$ is locally simply connected at every point. 
\end{proof}

As and immediate corollary we have:
\begin{corollary}\label{OnePtSC}
If $G$ is a 1-ended, finitely presented group that is semistable at infinity. Then for any finite connected complex $Y$ with $\pi_1(Y)=G$, the 1-point compactification of the universal cover of $Y$ is both simply connected and locally simply connected. 
\end{corollary}
Note that if $G$ is only finitely generated then there is an analogous result since we do not require the space $X$ of the theorem to be simply connected. 

\begin{theorem} Suppose $X$ is a non-compact, locally finite and connected 1-ended complex. If $X^{\ast}$ is simply connected and locally simply connected then $X$ has semistable first homology at $\infty$ . 
\end{theorem}
\begin{proof}
Let $p$ be the point at infinity in $X^\ast$. If $N$ is a neighborhood of infinity, then $N^\ast$ is a neighborhood of $p$ in $X^\ast$. Let $M^\ast$ be a smaller neighborhood of $p$ so that loops in $M^\ast$ contract in $N^\ast$. Now let $a: S^1\to M$ be a loop in $M$. Then there is an extension $A: B^2\to N^\ast$. Consider the set $A^{-1}(p)$, which is a compact subset of $B^2$ disjoint from $S^1$. If $W$ is a manifold neighborhood of $A^{-1}(p)$ in $B^2$ then its outermost boundary circles cobound a surface (a punctured disk) with the loop $a$, so $a$ is homologous in $N$ to the sum of those loops appropriately oriented. Call that sum $a'$.  By choosing the neighborhood $W$ to be arbitrarily small, we can make $a'$ live in an arbitrarily small neighborhood of infinity.
\end{proof}

\begin{question}
Is the universal cover of every finite connected complex with 1-ended fundamental group both simply connected and locally simply connected?
\end{question}
 
 \subsubsection{A Catalogue of Homological Results} \label{CatRe}

Recall the previously mentioned results:
 \medskip
 
\noindent {\bf Theorem} \ref{TorF} If $G$ is a finitely presented group and $R$ a ring, then the $R$-module $H^2(G,RG)$ is torsion free. 

\medskip

\noindent {\bf Theorem} \ref{H20} If $G$ is finite $H^2(G,R)=0$.

\begin{theorem} [Corollary 5.2, \cite{FFT}] \label{FFT1} 
If $G$ is a finitely presented group containing an element of infinite order, then $H^2(G,\mathbb ZG)$ is either $0$, $ \mathbb Z$ or not finitely generated. 
\end{theorem}  

\noindent {\bf Note.} M. Kapovich and B. Kleiner claimed to have improved Theorem \ref{FFT1} by eliminating the hypothesis that $G$ contains an element of infinite order.  
 
 \medskip
 
We are interested in $H_1(\varepsilon \tilde X)$ when $\tilde X$ is the universal cover of a finite connected complex $X$ with fundamental group $G$. Just as  $\pi_1(\varepsilon \tilde X)$ only depends on $G$, $H_1(\varepsilon \tilde X)$ only depends on $G$ and we denote this by $H_1(\varepsilon G)$.

The next result should be compared to the corresponding result (Theorem \ref{stablepro}) for $\pi_1(\varepsilon G)$. 

\begin{theorem}[\cite{Bo04}] \label{stableproH}   Suppose $G$ is a finitely presented group and $H_1(\varepsilon G)$ is stable. Then either $H_1(\varepsilon G)$ is pro-trivial or $G$ is virtually a closed surface group and $H_1(\varepsilon G)$ is pro-isomorphic to an inverse sequence where each group is $\mathbb Z$ and each bonding map is an isomorphism.
\end{theorem} 

The Davis group $G$ of Example \ref{DavisEx}  is not simply connected at $\infty$, but $H_1(\varepsilon G)$ is pro-trivial. 

We know only a few first homology semistability results that do not follow from fundamental group semistability results (see Corollary \ref{GM2}).

\begin{theorem} [Theorem 1.2, \cite{GM85}] Suppose $G$ is a finitely presented group which does not contain a free subgroup of rank 2, and suppose $\mathbb Z\oplus\mathbb Z$ is a quotient of $G$. Then $H^2(G,\mathbb ZG)$ is free abelian.
\end{theorem}
 
 
 \begin{theorem} [Theorem1.1, \cite{MS19}] \label{HomRH} 
Suppose a finitely presented group $G$ is hyperbolic relative to $\mathcal P=\{P_1,\ldots ,P_n\}$ a set of 
finitely presented  subgroups (with $G\ne P_i$ for all $i$). If $H^2(P_i,\mathbb ZP_i)$ is free abelian for all $i$, then  $H^2(G;\mathbb ZG)$ is free abelian.
\end{theorem} 

Recall that B. Jackson (Theorem \ref{J}) shows that  if $H$ is an infinite, finitely presented, normal subgroup of infinite index in the finitely presented group $G$, and either $H$ or $G/H$ is 1-ended. Then $G$ is simply connected at $\infty$. The converse follows from the next result. 

\begin{theorem} \label{Proper2exact} 
If $1\to N\to G\to H\to 1$ is a short exact sequence of infinite finitely presented groups, and neither $N$ nor $H$ is 1-ended then the first homology at $\infty$ of $G$ is non-trivial. Equivalently $H^2(G,\mathbb ZG)$ is non-trivial. 
\end{theorem}
\begin{proof}
Geoghegan shows $G$ is proper 2-equivalent to $N\times H$ (Theorem 16.8.4, \cite{G}). 
\end{proof}

\begin{theorem} [Theorem 1.1, \cite{BM17})] \label{Holo}
Let $G$ be a torsion-free holomorphically convex group. Then $H^2(G,\mathbb ZG)$ is a free abelian group.
\end{theorem}

\begin{corollary} [Corollary 1.2, \cite{BM17}] \label{Proj} 
If $G$ is a linear torsion-free projective group, then $H^2(G,\mathbb ZG)$ is a free abelian group.
\end{corollary}

The next two results were proved in 1985. They were improved in 1986, by changing $<$ to $\leq$ in parts 3 and 4. We state the improved versions. 

\begin{theorem} [Theorem 1.1, \cite{GM85}]\label{GMEnd}   
Let $G$ be a group of type $F(n)$, and let $X$ be a $K(G,1)$ CW-complex having finite $n$-skeleton and let $\tilde X$ be the universal cover of $X$. 
\begin{enumerate}
\item For $k\leq n$, $H^k(G,\mathbb ZG)$ mod torsion is free abelian if and only if $H_{k-1}(\varepsilon\tilde X^n)$ is semistable. 
\item For $k\leq n$, $H^k(G,\mathbb ZG)$ is torsion free if and only if $H_{k-2}(\varepsilon \tilde X^n)$ is pro-torsion free. 
\item For $G$ infinite, and $k\leq n$, $H^k(G,\mathbb ZG)$ is a torsion group if and only if $\bar H_{k-1}(\varepsilon \tilde X^n)$ is pro-finite.
\item For $G$ infinite, and $k\leq n$, $H^k(G, \mathbb ZG)$ mod torsion is free abelian of finite rank $R$ if and only if $\bar H_{k-1}(\varepsilon\tilde X^n)$ mod torsion is stable with free abelian inverse limit of finite rank $R$.
\end{enumerate}
\end{theorem}

\begin{corollary}\label{GME2}  
With $G$ and $X$ as in the previous theorem:
\begin{enumerate}
\item $H^r(G,\mathbb ZG)=0$ for all $r<n$ and $H^n(G,\mathbb ZG)$ is free abelian  if and only if $H_{r}(\varepsilon\tilde X^n)$ is  pro-trivial for all $r\leq n-2$ and $H_{n-1}(\varepsilon\tilde X^n)$ is semistable. 

\item  $H^r(G,\mathbb ZG)=0$ for all $r<n$ and $H^n(G,\mathbb ZG)$  is torsion free if and only if $\bar H_{r}(\varepsilon \tilde X^n)$ is pro-trivial for all $r\leq n-2$. 

\item  $H^r(G,\mathbb ZG)=0$ for all $r\leq n$ if and only if $\bar H_{r}(\varepsilon \tilde X^n)$ is pro-trivial for all $r\leq n-2$ and $\bar H_{n-1}(\varepsilon \tilde X^n)$ is pro-finite.

\item  $H^r(G, \mathbb ZG)=0$ for all $r\leq n-1$ and $H^n(G,\mathbb ZG)$ is free abelian of finite rank $R$ if and only if $\bar H_{r}(\varepsilon\tilde X^n)$ is pro-trivial for all $r\leq n-2$ and $H_{n-1}(\varepsilon \tilde X^n)$ mod torsion is stable with free abelian inverse limit of finite rank $R$.
\end{enumerate}
\end{corollary}

In order to prove these results, connections were drawn between the $k^{th}$ cohomology of $\tilde X^n$ with compact supports, $H^k(G,\mathbb ZG)$ and $H_{k-1}(\varepsilon \tilde X^n)$. The next result follows from the previous. It is also implicit in \cite{Ho77}, also see Theorems 16.5.1 and 16.5.2 of \cite{G}.

Note that $H_0(\varepsilon \tilde X^2)$ is always pro-free. If $G$ is $\pi_1(\varepsilon \tilde X^2)$ is semistable at infinity (simply connected at infinity) then $H_1(\varepsilon \tilde X^2)$ is semistable at infinity (pro-trivial). 
So all of the semistable fundamental group at $\infty$ and all simply connected at $\infty$ results have corresponding $H^2(G,\mathbb ZG)$ results. 

\begin{corollary} \label{GM2} 
Assume the hypotheses of Theorem \ref{GMEnd} (with $n=2$). The group $H^2(G,\mathbb ZG)$ is free abelian if and only if $ H_1(\varepsilon\tilde X^2)$ is semistable.  Also, $H^2(G,\mathbb ZG)=0$ if and only if $\bar H_1(\varepsilon \tilde X^2)$ is pro-finite. 
\end{corollary}
\begin{proof}
Theorem \ref{TorF} implies that $H^2(G,\mathbb ZG)$ is torsion free. Now simply apply Theorem \ref{GMEnd} 1) for the first part. For the second part, apply parts 2) and 3) of Theorem \ref{GMEnd}. 
\end{proof}

\begin{question} If the finitely generated group $A$ is simply connected (semistable) at $\infty$ in the finitely presented group $G$ then what is the relationship of $H^2(G;\mathbb ZG)$ to  $H^2(A;\mathbb ZA)$?
\end{question}

One might expect that the inclusion map $i:A\to G$  would induce the trivial map $H^2(G,\mathbb ZG)\to H^2(A,\mathbb ZA)$ in the simply connected at $\infty$ setting. 
\newpage

\section{List of Simply Connected at Infinity Group Results}\label{LSCIR}

Theorem \ref{ssqi} - Quasi-isometry and Proper 2-equivalence]

Theorem \ref{Fin2E} - Finite and 2-Ended Groups

Theorem \ref{cosc1} and \ref{cosc2} - Coaxial Actions

Theorem \ref{ProTriv} - Pro-trivial Fundamental Group at Infinity

Theorem \ref{T187a} - Normal Abelian Subgroups

Theorem \ref{J} - Group Extensions

Theorem \ref{L} - Subnormal Subgroups

Theorem \ref{Subset} - Subnormal Subgroups 2

Theorem \ref{MainCM} - Commensurated Subgroups

Theorem \ref{MainA} - Subcommensurated Subgroups

Theorem \ref{sc} -  Finitely Generated Products

Theorem \ref{hnn} - Finitely Generated HNN-Extensions 

Theorem \ref{PHT} - The Proper Hurewicz Theorem

Theorem \ref{MM24} - Stallings Group

Theorem \ref{BF1} - Out$(F_n)$

Theorem \ref{MCGPr} - Special Presentations 

Theorem \ref{MCGSCINF} - Mapping Class Groups of Surfaces

Theorem \ref{GLNZSC} - $GL_n(\mathbb Z)$ and $SL_n (\mathbb Z)$

Theorem \ref{MT87} - Solvable Groups

Theorem \ref{SILK1} - Extended Lamplighter

Theorem \ref{MM} - Ascending HNN-Extensions

Theorem \ref{HNNE2} - Near Ascending HNN-Extensions

Theorem \ref{ThompNC} - Thompson's Group

Theorem \ref{CoxSC} - Coxeter Groups

Theorem \ref{BRMB} - Right Angled Artin Groups

Theorem \ref{SCISD} - Sidki Doubles

Theorem \ref{JackIi} - Amalgamated Products

Theorem \ref{JackIIi} - HNN-Extensions

Theorem \ref{OnePtSC} - One Point Compactifications

\section{Group/Subgroup Index}\label{GSI}

$\ \ \ \ $ Amalgamated Products of Groups \S \ref{Comb1R}

Amenable Groups (see Theorem \ref{NOF2})

Artin Groups \S \ref{ACGE}, \S \ref{AGCG}

Ascending HNN-Extensions \S \ref{AHNN}, \S \ref{AHNNE}

Baumslag-Solitar Groups \S \ref{SubCom}

Bieri-Stallings Groups \S \ref{BSgps}

CAT(0) Groups \S \ref{WHyp}

Combination Results for Groups \S \ref{Comb1R}

Commensurated and Subcommensurated Subgroups \S \ref{normalE}, \S \ref{SubCom} 

Coxeter Groups \S \ref{ACGE}, \S \ref{WildPro}

Davis Manifold Groups \S \ref{WildPro}

Eventually Injective Endomorphisms  \S \ref{AHNNE}

GL($\mathbb Z, n$)  \S \ref{OUT} 

Graph Products of Groups \S \ref{GPofG}

Groups with no Free Subgroup of Rank 2 \S \ref{noF2}

Groups of type $F_n$ or $FP_n$ \S \ref{Pr2eq}

HNN-Extensions of Groups \S \ref{Comb1R}

Holomorphically Convex Groups \ref{Holo}

Inward tame ANR's and Manifolds \ref{CC}

Knot Groups \S \ref{KGs}

Lamplighter Group and Extended Lamplighter Group \S \ref{NonSS}

Linear Groups \S \ref{AHNNE}

Linear Torsion-Free Projective Groups \S \ref{Proj}

Manifolds \S \ref{CC}

Mapping Class Groups  \S \ref{OUT} 

Metabelian Groups \S \ref{SolMet}

Metanilpotent Groups \S \ref{SolMet}

Normal and Subnormal Subgroups \S \ref{normalE}, \S \ref{SubCom} 

One Relator Groups \S \ref{Comb1R}

Out($F_n$)  \S \ref{OUT} 

Polynomial Growth \S \ref{SolMet}

Polycyclic Groups \S \ref{SolMet}

Relatively Hyperbolic Groups \S \ref{RHGs}

Sidki Doubles \S \ref{Sed}

Solvable Groups \S \ref{SolMet}

Stallings Group \S \ref{BSgps}

Thompson's Group $F$ \S \ref{TGF}

Three Manifold Groups \S \ref{CC}

Weakly Chained Space \S \ref{WHyp}

Word Hyperbolic Groups \S \ref{WHyp}

\bibliographystyle{amsalpha}
\bibliography{paper}{}

\end{document}